\documentclass[12pt,reqno]{amsart}
\usepackage{amsmath,amsthm,amsfonts,amssymb,amscd,  multicol, verbatim, enumerate,booktabs}
\usepackage{verbatim}
\usepackage{fullpage}
\input{xy}
\usepackage{color}
\usepackage{graphicx}
\graphicspath{ {./images/} }
\usepackage[english]{babel}
\usepackage{cite}
\usepackage{amsfonts}
\usepackage{latexsym}
\usepackage{amsopn}
\usepackage[all]{xy}
\usepackage{mathrsfs}
\usepackage{float}
\usepackage{amscd}
\usepackage{hyperref} 
\usepackage{tikz-cd}
\usepackage{ulem}

\newcommand{\id}{\mathrm{id}}

\newcommand{\CC}{\mathbb{C}}

\newcommand{\NN}{\mathbb{N}}

\newcommand{\FF}{\mathcal{F}}
\newcommand{\GG}{\mathcal{G}}

\newcommand{\OO}{\mathcal{O}}

\newcommand{\RR}{\mathbb{R}}

\newcommand{\Vu}{\mathcal{V}}

\newcommand{\Sn}{\mathbb{S}_n}

\newcommand\PS{\mathcal{PS}}

\newcommand{\sym}{\mathrm{Sym}}
\newcommand{\PPP}{\mathcal{P}}
\newcommand{\FP}{\mathcal{FP}}

\newcommand{\so}{{\mathbf s}_0}

\numberwithin{equation}{section}

\newtheorem{thm}{Theorem}[section]
\newtheorem{lem}[thm]{Lemma}
\newtheorem{prop}[thm]{Proposition}
\newtheorem{cor}[thm]{Corollary}
\newtheorem{remark}[thm]{Remark}
\newtheorem{example}[thm]{Example}
\newtheorem{defn}[thm]{Definition}

\title{Truncated factorized perverse sheaves on $\sym({\mathbb C})$}
\author{Giovanna Carnovale, Francesco Esposito, Lleonard Rubio y Degrassi}

\begin{document}

\begin{abstract}
Kapranov and Schechtman defined the  category $\FP$ of factorized perverse sheaves on $\sym(\CC)$ smooth along the stratification given by multiplicities and with values in a braided monoidal category $\Vu$. We define for each $d\in\NN$ the category  $\FP^{\leq d}$ of factorized perverse sheaves on $\coprod_{n\leq d}\sym^{n}(\CC)$ and  the category $\FP_{\leq d}$ of factorized perverse sheaves on the open subset of  $\sym(\CC)$ consisting of multi-sets with multiplicities bounded by $d$.

We prove that the natural restriction functor from $\FP_{\leq d}$ to $\FP^{\leq d}$ 
is an equivalence for any $d\in \NN$, and that $\FP^{\leq 1}$ and $\FP_{\leq 1}$ are equivalent to $\Vu$. We show that the full direct image $*$, the extension by zero $!$ and the intermediate extension $!*$ induce functors from $\FP_{\leq d}$ to $\FP$. 

In addition, we show that the  families $(\FP^{\leq d})_{d\in\mathbb N}$ and $(\FP_{\leq d})_{d\in\mathbb N}$  fit into systems of categories, compatible with restrictions and extensions, whose inverse limit is $\FP$.
\end{abstract}

\maketitle

“Le secret d'ennuyer est celui de tout dire.” (Voltaire)

\section{Introduction}

Factorized perverse sheaves have been introduced in geometric representation theory in the spirit of localization: categories of representations of relevant objects  are proved to be equivalent to categories of factorizable perverse sheaves on suitable spaces, \cite{BFS,Ga}. Kapranov and Schechtman have studied in detail in \cite{KS3} the category $\FP$ of factorized perverse sheaves on the infinite-dimensional space $\sym(\CC)$ of monic polynomials with coefficients in $\CC$, stratified by multiplicity of roots, and with values in a braided monoidal category $\Vu$. The objects are pairs given by a perverse sheaf on $\sym(\CC)$ and a family of isomorphisms of sheaves that take into account the monoid structure of $\sym(\CC)$. 

\medskip

The description of such perverse sheaves can be made very concrete: the category of perverse sheaves on $\CC^n$ is equivalent to the category of representations of an explicit quiver with monomial relations, whose vertices correspond to real faces in the (real) hyperplane arrangement corresponding to the reflection representation of $\Sn$ on $\CC^n$, \cite{KS1}. From this description one can recover a description of perverse sheaves on the graded component $\sym^n(\CC)$. Indeed, the latter can be embedded into the category of $\Sn$-equivariant perverse sheaves on $\CC^n$, \cite[Proposition 2.3.3]{KS3}.

\medskip

The main result in \cite{KS3} is that the category $\FP$ is equivalent to the category of connected bialgebras in $\Vu$, and that relevant objects in this category, such as the tensor algebra, the cotensor (big shuffle) algebra and the Nichols (shuffle) algebra of an object $V$ in $\Vu$ correspond respectively to the full direct image $*$, the extension by zero $!$ and intermediate extension $!*$ of the collection of local systems on the configuration spaces $\sym_{\neq0}^n(\CC)$ associated with the $n$-th tensor power of $V$. This establishes a relation between natural algebraic constructions in Hopf algebra theory and natural geometric constructions in the theory of perverse sheaves: this relation can  serve as a bridge between the two theories.

\medskip

Motivated by this connection, one is led to ask  what category of algebraic structures can be related to perverse sheaves on the family of configuration spaces of $n$ points for each $n$, i.e., the space of multiplicity-free monic polynomials. More generally, for $d\in\NN$, we wish to understand what category of algebraic structures is related to perverse sheaves on the family of open subsets $\sym^n_{\leq d}(\CC)$, corresponding to monic polynomials of degree $n$ whose root multiplicities are bounded by $d$. 

With this idea in the back of our minds, in the present paper we introduce for each $d\in\NN$ the categories $\FP_{\leq d}$ and  $\FP^{\leq d}$ 
of factorized perverse sheaves  on $\sym_{\leq d}(\CC)=\coprod_{n\in\NN}\sym^n_{\leq d}(\CC)$ and  $\sym^{\leq d}(\CC)=\coprod_{n\leq d}\sym^n(\CC)$, respectively, and develop their theory. 

\medskip

The main results of the paper are the following.

\medskip

\noindent{\bf Theorem A. }Let $d\in\NN_{\geq1}$.  The natural restriction functor induces an equivalence between $\FP_{\leq d}$ and $\FP^{\leq d}$. (Theorems \ref{prop:def-funtori-restrizione},  \ref{thm:equivalence}).

\medskip

In other words, the factorization data on a perverse sheaf $\FF$ on $\sym^{\leq d}(\CC)$ contains enough information to extend
$\FF$ to a factorized sheaf on the dense open subset $\sym_{\leq d}(\CC)$ of $\sym(\CC)$.

In the special case of $d=1$, corresponding to the configuration spaces,  the two categories $\FP_{\leq 1}$ to $\FP^{\leq 1}$ are isomorphic to $\Vu$, a result that was implicitly  contained in \cite[Theorem 3.3.3]{KS3}.

\medskip

We also construct extension functors, and prove natural properties.

\medskip

\noindent{\bf Theorem B. } Let $d\in\NN_{\geq1}$. The full direct image $*$, the extension by zero $!$ and the intermediate extension $!*$ induce functors from $\FP_{\leq d}$ to $\FP$. In addition,  the extension by zero $!$ functor is left adjoint to the natural restriction functor, (Theorems \ref{prop:def-estensione}, \ref{thm:extension-geometric}).

\medskip

Finally, we analyse the limit of these sequences of categories and functors.

\medskip

\noindent{\bf Theorem C. }The families $(\FP^{\leq d})_{d\in\mathbb N}$ and $(\FP_{\leq d})_{d\in\mathbb N}$  fit naturally into systems of categories that are compatible with restriction functors, and whose inverse limit is $\FP$ (Theorem \ref{thm:limit}).

\medskip  

We expect that  Kapranov and Schechtman's dictionary can be generalized to establish an equivalence between the categories $\mathcal{CB}^{\leq d}(\Vu)$ of  connected bialgebras modulo $(d+1)$ in $\Vu$ that we introduced in \cite{CERyD-A}  and the categories $\FP_{\leq d}$ and $\FP^{\leq d}$. We plan to construct a family of equivalences of categories $\FP_{\leq d}\to \mathcal{CB}^{\leq d}(\Vu)$ whose limit is the equivalence in \cite{KS3}. This would allow to translate the algebraic construction of $d$-th approximation of a graded algebra introduced in \cite{CERyD-A} to a geometric construction in terms of extensions sheaves on $\FP_{\leq d}$ and can hopefully be applied to questions concerning finite generation of Nichols algebras.  

\medskip

The paper is structured as follows: in Section 2 we recall basic facts on operads, the operad of little $2$-cubes, Deligne's interpretation of braided monoidal categories, the (outer) tensor product construction of perverse sheaves with values in a braided monoidal category $\Vu$. Next, in Section 3,  we focus  on the stratified space $\sym(\CC)$ and the open subsets of our interest and introduce the relevant functors  on perverse sheaves: restriction, extension, tensor product, multiplication embedding, and the monodromy isomorphism of functors. Then we are in a position to introduce factorization data on perverse sheaves on $\sym(\CC)$, $\sym^{\leq d}(\CC)$ and $\sym_{\leq d}(\CC)$, and the categories $\FP$, $\FP^{\leq d}$ and $\FP_{\leq d}$ of factorized perverse sheaves in Section 4. We show in Proposition \ref{prop:factvb} that factorization data are determined by a smaller family of conditions. We also show that the restriction and extension functors  at the level of perverse sheaves from Section 3 induce functors at the level of factorized sheaves with natural properties.  

In Section 5 we show that the restriction functors $\FP_{\leq d}\to \FP^{\leq d}$ are indeed equivalences (Theorem \ref{thm:equivalence}) providing an explicit quasi-inverse. In the last section, we show  that the inverse limit of the families $(\FP_{\leq d})_{d\in \NN}$ and $(\FP^{\leq d})_{d\in \NN}$ is $\FP$. 
A list of the used symbols  is to be found at the end of the paper. 

\section{Preliminaries}

In this section we recall the basic notions on braided monoidal categories, operads, and tensor products of perverse sheaves.

\subsection{Braided monoidal categories}
With the purpose of fixing notation, we recall the definitions of monoidal category and braiding. The reader is referred to \cite{EGNO} for a complete treatment. 

A monoidal category structure on a category $\mathcal{C}$ is the datum of a bifunctor 
    $\otimes: \mathcal{C}\times\mathcal{C}
    \rightarrow \mathcal{C}$, called tensor product, a natural isomorphism $a\colon(-\otimes -)\otimes\to -\otimes(-\otimes -)$, called the associativity constraint, an object $\mathbf{1}_{\mathcal{C}}$ of $\mathcal{C}$, and an isomorphism $\iota\colon \mathbf{1}_{\mathcal{C}}\otimes \mathbf{1}_{\mathcal{C}}\to \mathbf{1}_{\mathcal{C}}$, called unit constraint,  subject to the pentagon and unit compatibility axioms \cite[(2.2), (2.3)]{EGNO}. A monoidal category is said to be strict if for all objects $A,B,C$ in ${\mathcal C}$ there are equalities $(A\otimes B)\otimes C=A\otimes(B\otimes C)$ and $A\otimes \mathbf{1}_{\mathcal{C}}=A=\mathbf{1}_{\mathcal{C}}\otimes A$ and the associativity and unit constraints are the identity maps. 

Given a monoidal category structure  $(\otimes,\mathbf{1}_{\mathcal{C}},a,\iota)$ on ${\mathcal C}$, the opposite  monoidal category structure 
on $\mathcal{C}$ is given by the bifunctor $\otimes^{op}$ such that $A \otimes^{op}B:=B\otimes A$ for any pair of objects $A,B$ in ${\mathcal C}$,  the associator $a^{op}_{A,B,C}=a^{-1}_{C,B,A}$, together with $\mathbf{1}_{\mathcal{C}}$, and $\iota$.

A braiding on a monoidal structure $(\mathcal{C},\otimes,\mathbf{1}_{\mathcal{C}})$ is an isomorphism of functors $R:\otimes \rightarrow \otimes^{op}$, satisfying the following equations on triples of objects $A,B$ and $C$: 
\begin{align}\label{eq:braiding}
   a_{B,C,A}\circ R_{A,B\otimes C}\circ a_{A,B,C}&=(\id_B \otimes R_{A,C})\circ a_{B,A,C}(R_{A,B}\otimes \id_C),\\
  a^{-1}_{C,A,B}\circ  R_{A\otimes B,C}a^{-1}_{A,B,C}&=(R_{A,C}\otimes \id_B)\circ a^{-1}_{A,C,B}\circ(\id_A\otimes R_{B,C}).  
\end{align}
A monoidal category with a braiding is called a braided monoidal category. A braided monoidal category is called symmetric if $R_{A,B}\circ R_{B,A}= \id_{B\otimes A}$ for any pair of objects $A$ and $B$ in ${\mathcal C}$.

If ${\mathcal C}$ is strict, equation \eqref{eq:braiding} implies the Yang-Baxter equation on ${\mathcal C}^3$:
\begin{equation}\label{eq:braid}    (\id\otimes R) \circ (R\otimes \id) \circ (\id \otimes R) =
    (R\otimes \id) \circ (\id \otimes R) \circ (R\otimes \id) \end{equation}
    where $\id $ denotes the identity functor.

By Mac Lane's theorem, every monoidal category is monoidally equivalent to a strict one \cite[Theorem 2.8.5]{EGNO}, so we will treat all monoidal categories as strict ones. From now on, ${\mathcal S}$ will always denote a symmetric category, and $\Vu$ will always denote a braided monoidal category,  possibly with additional properties. 

\begin{example}\label{ex:top}
    {\rm The category of topological spaces  with tensor product, unit object, and braiding given respectively by direct product of spaces,  the one-point set and the flip of components is a symmetric category. We denote this braided monoidal category by $\mathrm{Top}^{\prod{}}$. An alternative symmetric category structure on the category of topological spaces is obtained by taking the disjoint union of spaces as tensor product  and the flip of components as braiding. We  denote this braided monoidal category by  $\mathrm{Top}^{\coprod{}}$. The category whose objects are groupoids is a symmetric category with tensor product given by direct product of categories and flip of components as braiding. We denote this braided monoidal category by $\mathrm{Gr}$.}
\end{example}

\subsection{Operads} In this section we introduce the  notions of an operad, an algebra over an operad and the  operad of  little $2$-cubes.

\begin{defn}
\label{def:operad}
    A symmetric operad with values in a symmetric  category $(\mathcal{S},\otimes,{\mathbf 1_{\mathcal{S}}},R)$ is the datum of 
    \begin{enumerate}
        \item[(i)] an object $\mathcal{O}(n)$ of $\mathcal{S}$, for every integer $n\geq 0$, with $\OO(0)$ being the unit object $\mathbf 1_{\mathcal{S}}$.
         \item[(ii)] a unit morphism $\eta: {\mathbf 1_{\mathcal{S}}} \to \mathcal{O}(1)$
        \item[(iii)] a morphism, called operadic composition, 
        $$
        \mathcal{O}(n)\otimes \bigotimes_{i=1}^n \mathcal{O}(m_i) \rightarrow \mathcal{O}(m),
        $$
         defined for any $n\in\NN_{\geq 1}$ and $m_1,\dots, m_n$ such that $m = \sum_{i=1}^n m_i$.
        \item [(iv)]  an action of the symmetric group $\Sn$ on $\mathcal{O}(n)$ for each $n\in\NN_{\geq1}$. 
    \end{enumerate}
   such that natural equivariance, unit, and associativity relations hold, see \cite[Fig. 1.1, 1.2, 1.3]{Fresse}. 
 \end{defn}

 Rather than saying that $\mathcal{O}$  is an operad with  values in $\mathcal{S}$ we will also  say that $\mathcal{O}$ is an operad in $\mathcal{S}$ or that it  is $\mathcal{S}$-valued. If ${\mathcal S}=\mathrm{Top}^{\prod{}}$ we will say that $\mathcal O$ is a topological operad.
   
\begin{defn}\cite[\S 1.1.2]{Fresse}
(\cite[Fig. 1.1]{Fresse})    Let $\mathcal{O}$ and $\mathcal{O}^{\prime}$ be operads in $\mathcal{S}$. A morphism of operads $\phi:\mathcal{O} \to \mathcal{O}^{\prime}$  is a sequence of $\Sn$-equivariant morphisms $\phi_n:\mathcal{O}(n) \to \mathcal{O}^{\prime}(n)$ in $\mathcal S$ which commute with the operadic composition and preserve the unit. A morphism of operads $\phi$ is an isomorphism if  $\phi_n$ is an isomorphism for each $n\in\NN$.
\end{defn}

\begin{remark}\label{remark:operad-functor} {\rm \cite[\S 1.1.4, Proposition ~3.1.1]{Fresse}
    \begin{enumerate}
        \item[(i)] 
        If the objects in ${\mathcal S}$ are sets, operadic composition  on elements  is denoted by
      \begin{align*}
        \eta = \phi \circ (\psi_1, \ldots, \psi_n),&&\phi\in\mathcal O(n),\,\psi_l\in \mathcal O(m_l), \ l=1,\,\ldots,\,n.
      \end{align*}

        \item[(ii)] Given a braided monoidal functor $F\colon \mathcal{S} \rightarrow \mathcal{S}'$ between symmetric categories and an  operad $\mathcal{O}$ in $\mathcal{S}$, composition with $F$ yields an operad $\mathcal{O}'=F\mathcal{O}$  in $\mathcal{S}'$.
    \end{enumerate}}
\end{remark}

\begin{defn}
\label{O-algebra}
Let $\mathcal{O}$ be an operad with values in  $\mathcal{S}$.
An algebra over $\mathcal{O}$, or an $\mathcal{O}$-algebra,  is an object $A$ of  $\mathcal{S}$ together with morphisms
        $$\lambda_n: \mathcal{O}(n)\otimes A^{\otimes n}\to A$$
for all $n\in\mathbb{N}$ that satisfy  equivariance, associativity relations and unit relations as in \cite[Figures 1.7, 1.8, 1.9]{Fresse}.
\end{defn}
  \subsubsection{Little $2$-cubes operad}\label{sec:little}

In the sequel we will mainly work with the operad of little $2$-cubes which we now recall.
From now on $I:=(0,1)$ and $U:=I\times I\subset\mathbb{R}^2\cong\mathbb{C}$ is  the open unit square.\label{UU}

\begin{defn}\label{unary}
 A {unary linear embedding} is a map 
    $$
    f:U\rightarrow U
    $$
    which is the restriction of an affine map $\Tilde{\phi}: \mathbb{C}\rightarrow \mathbb{C}$ of the form
    $z\mapsto az+b$, with $a\in \mathbb{R}_{>0}$ and $b\in \mathbb{C}$.
    We denote by $\id_U$ the {identity unary linear embedding}, that is, the linear embedding such that  $a=1$ and $b=0$.
    More generally, for $n\in\NN_{\geq 1}$ and $U^{\coprod n}:=\underbrace{U\amalg \cdots\amalg U}_{n \mbox{\tiny{ times }}}$ one defines an $n$-ary linear embedding to be a map
    $$
    \varphi:= \varphi_1 \amalg \cdots \amalg \varphi_n \colon U^{\coprod n}\rightarrow U
    $$
     where each $\varphi_j: U\rightarrow U$ for $j=1,\,\ldots,\, n$ is a unary linear embedding and 
     $\overline{\varphi_j(U)}\cap \overline{\varphi_l(U)}=\emptyset$ for $1\leq j\neq l\leq n$. 
\end{defn}
For $n\geq 1$ we denote by $E_2(n)$ the set of $n$-ary linear embeddings of $U$. For uniformity and later purposes we will also consider the set $E_2(0)$ consisting only of the trivial embedding $\varphi_0 \colon U^{\coprod 0}=\emptyset\to U$.

Each set $E_2(n)$ is endowed with the compact-open topology, cf.\cite[4.1.2]{Fresse}.
The image of a unary linear embedding is a cube $U'\subset U$ and the map from the set of unary linear embeddings to the set of cubes contained in $U$ is a bijection (in fact, a homeomorphism). For $n>1$, associating to an $n$-ary linear embedding  the $n$-tuple of images of all unary components  yields a surjection between the set of $n$-ary linear embeddings and the set of $n$-tuples of subcubes of $U$ with disjoint closures. Two $n$-ary embeddings have the same image if and only if they differ by a permutation of the components.

The natural left action of the symmetric group $\Sn$ on $U^{\coprod n}$ gives a right action on $E_2(n)$ namely, for $\sigma\in \Sn$ and  $\varphi\in E_2(n)$ we have \begin{equation}\label{eq:action-varphi}\varphi^\sigma:=\varphi\circ\sigma.\end{equation}

Given an $m_l$-ary linear embedding $\psi_l$ for every $l=1,\,\ldots,\,n$, and an $n$-ary linear embedding $\varphi$, their composition is the 
$m$-ary linear embedding $\eta$ with $m=\sum_{l=1}^nm_l$ defined by the diagram
\begin{equation}\label{eq:linemb}
\begin{tikzcd}[row sep=huge]
U^{\coprod n} \arrow[rr, "\varphi"]  & &U \\
\coprod_{l=1}^{n} U^{\coprod m_l} \arrow[u, "\coprod_{l=1}^n \psi_l"] \arrow[rr, "="] & &
U^{\coprod m} \arrow[u, "\eta"]
\end{tikzcd} 
\end{equation}

This composition of linear embeddings of $U$, together with the permutation action and the identity unary embedding make 
   $E_2 = (E_2(n))_{n\in \mathbb{N}} $ into 
   an  operad with values in 
     $\mathrm{Top}^{\prod{}}$
    called the {operad of little 2-cubes}.

\begin{remark}
  {\rm   Replacing linear embeddings of open cubes by embeddings of open disks gives the notion of operad of  little $2$-disks, which is equivalent to $E_2$.}
\end{remark}

Let $Y$ be an object in $\mathrm{Top}^{\prod{}}$. Its fundamental groupoid $\Pi_1(Y)$ is the groupoid whose objects are the points of $Y$ and its morphisms between two points are the homotopy classes of paths between them. 
It can be checked that the assignment $Y\mapsto \Pi_1(Y)$ induces a braided monoidal functor $\mathrm{Top}^{\prod{}}\to \mathrm{Gr}$,  see \cite[Proposition 5.3.2]{Fresse}. 
Applying Remark \ref{remark:operad-functor} to $F=\Pi_1$ and $E_2$ we obtain \label{p:groupoid}
the $\mathrm{Gr}$-valued operad $\Pi_1 E_2$, called the fundamental groupoid of the operad of little 2-cubes $E_2$.
Note that $\Pi_1 E_2(n)$ is now an object in the category of groupoids for any $n\in \NN$. The objects of $\Pi_1 E_2(n)$  are  $n$-ary linear embeddings and the morphisms are homotopy classes of paths between pairs of $n$-ary embeddings. 
In this case the operadic composition is a functor, i.e., a 
morphism in the category of groupoids. The (right) $\Sn$-action on objects is as in \eqref{eq:action-varphi}.  On morphisms, it is given as follows. For $\varphi,\, \varphi'\in E_2(n)$ and   a path  $\gamma\colon [0,1]\to E_2(n)$ from $\varphi$ to $\varphi'$, the action of $\sigma\in  \Sn$  on $\gamma$ gives the induced homotopy class of paths $\gamma^{\sigma}= \gamma\circ\sigma$ from  $\varphi^{\sigma}$ to ${\varphi'}^{\sigma}$. 

\begin{remark}\label{piE1}{\rm \begin{enumerate}
\item The space $E_2(1)$ is contractible, hence, there is a unique homotopy class of paths  between any pair of elements in $E_2(1)$.
In particular, there is a unique morphism in $\Pi_1 E_2(1)$ between any unary linear embedding and $\id_U$.
\item Any path $\gamma\colon [0,1]\to E_2(n)$ from $\varphi$ to $\varphi'$ is homotopy equivalent to a composition of paths of the form $\vartheta\circ(\gamma_1,\ldots,\gamma_n)$ where $\theta\in E_2(n)$ and $\gamma_j\colon [0,1]\to E_2(1)$ for $j=1,\ldots,\,m$ are paths between $1$-ary embeddings.
\end{enumerate}}
\end{remark}

\begin{defn}\label{def:linearemb}
    A linear embedding $\varphi\in E_2(n)$ is said to be {vertically disjoint} if the projection to the imaginary axis of the image of $\varphi$ is the union of exactly $n$ disjoint open intervals.
    \end{defn}
   If $\varphi$ is vertically disjoint, then $\varphi^\sigma$ is again vertically disjoint for any $\sigma\in \Sn$. 

\medskip
   
    We denote by $E_2^v (n)$ the locus of vertically disjoint linear embeddings in $E_2(n)$ and by  $(\Pi_1 E_2)^v(n)$  the full subcategory of $\Pi_1 E_2(n)$ whose objects are those in $E_2^v (n)$. These groupoids assemble into an operad in $\mathrm{Gr}$, denoted  by $(\Pi_1 E_2)^v$. 
    By definition, morphisms are homotopy classes of paths in the topological space $E_2(n)$ with endpoints in $E^v_2(n)$. 
     The symmetric group action on $\Pi_1 E_2(n)$ induces an action on $(\Pi_1 E_2)^v(n)$ for any $n$. Operadic composition in  $\Pi_1 E_2(n)$ yields operadic composition in $(\Pi_1 E_2)^v$ at the level of objects because it preserves verticality, and at the level of morphisms because  $(\Pi_1 E_2)^v(n)$ is a full subcategory. 

     \medskip

\begin{remark}\begin{enumerate}
\item {\rm The inclusion functors $(\Pi_1 E_2)^v(n)\to \Pi_1 E_2(n)$ for any $n$ are in fact equivalences, so $(\Pi_1 E_2)^v$ is equivalent to $\Pi_1 E_2$. These two operads are also equivalent to the colored braid operad $CoB$, \cite[Theorem 5.3.4]{Fresse}.
  The topological operad $E_2^v = (E_2^v(n))_{n\in \mathbb{N}}$ is a suboperad of $E_2$ which is equivalent to the operad $E_1$ of intervals of the real line.} 

 \item {\rm Given $\varphi\in E_2^v (n)$, for any $l\in\{1,\,\ldots,\,n-1\}$, the $l$-th elementary braiding is the morphism in $(\Pi_1 E_2)^v(n)$ from $\varphi$ to $\varphi^{(l, l+1)}$ obtained by rotating the positions of the images of the $l$ and $(l+1)$-th cubes counter-clockwise.  Any morphism in $(\Pi_1 E_2)^v(n)$ is obtained by composing elementary braidings. }
 \end{enumerate}
\end{remark}

\subsection{ Deligne's interpretation of braided monoidal categories}

In the sequel we use the  following characterization, due to Deligne, of a braided monoidal  category in terms of the operad $\Pi_1 E_2$. We recall it, using the formulation of \cite[Proposition 3.2.1]{KS3}, (cf. also \cite [\S 6.2.7]{Fresse}).

\begin{thm}\label{thm:del}(Deligne) 
Let $\mathcal C$ be a category. A (strict) braided monoidal category structure on $\mathcal{C}$ is equivalent to a structure of $\Pi_1 E_2$-algebra on $\mathcal{C}$ in the category of  categories.
\end{thm}

In order to fix notation we spell out explicitly what a $\Pi_1 E_2$-algebra structure on $\mathcal{C}$ is. The functors  
   $\Pi_1E_2(n)\otimes \mathcal{C}^n\to \mathcal{C}$ for each $n$ at the level of objects and morphisms translate, respectively, into the datum of: 
   \begin{itemize}
       \item  a functor 
       \(
       \otimes_{\varphi}: \mathcal{C}^n \rightarrow \mathcal{C}
       \) for every $\varphi\in E_2(n)$,
       \item  an isomorphism of functors
       \(
     R_\gamma\colon  \otimes_{\varphi} \rightarrow \otimes_{\varphi'}
       \) for every homotopy class of paths $\gamma$ in $E_2(n)$ from $\varphi$ to $\varphi'$, satisfying:
   \end{itemize}
    
   \begin{itemize}
   \item[$\ast$] compatibility of $R$ with respect to composition of paths in $E_2(n)$;
       \item[$\ast$] compatibility with operadic composition, that is:
       for the linear embeddings $\varphi\in E_2(n)$ and $\psi_1\in E_2(m_1),\,\ldots,\,\psi_n\in E_2(m_n)$ with $m = m_1 + \ldots +m_n$, 
the following equality holds:
\begin{equation}
\otimes_{\varphi\circ (\psi_1, \ldots, \psi_n)} = \otimes_\varphi \circ (\otimes_{\psi_1}, \ldots, \otimes_{\psi_n})
\end{equation}
and similarly for the isomorphisms of functors $R$.
        \item[$\ast$] equivariance with respect to $\Sn$-action, that is: if $\gamma$ is a path from $\varphi$ to $\psi$ and $\sigma\in \Sn$, then $\gamma^\sigma$ is a path from $\varphi^\sigma$ to $\psi^\sigma$; the functors $\otimes_{\varphi^\sigma}$ and $\otimes_{\psi^\sigma}$ are identified, respectively, with $\otimes_{\varphi}\circ\sigma$ and $\otimes_{\psi}\circ\sigma$, and, through this identification, for any $n$-tuple of objects $V_1,\,\ldots,V_n$ in $\Vu$ the isomorphisms 
\begin{align}
R_{\gamma;\sigma(V_1,\,\ldots,V_n)}&\colon \otimes_{\varphi}\sigma(V_1,\,\ldots,\,V_n)\to \otimes_{\psi}\sigma(V_1,\,\ldots,\,V_n)
\end{align}
and 
\begin{align}
R_{\gamma^\sigma;(V_1,\,\ldots,V_n)}&\colon \otimes_{\varphi^\sigma}(V_1,\,\ldots,\,V_n)\to \otimes_{\psi^\sigma}(V_1,\,\ldots,\,V_n)
\end{align}
are identified.
   \end{itemize}
For the reader's convenience we now give a sketch of how to construct such a datum from a braided monoidal category $(\mathcal{C}, \otimes, \mathbf 1_{\mathcal{C}}, R)$.  The category $\mathcal{C}^0$ is a terminal object in the category of categories, thus, a functor $\mathcal{C}^0 \rightarrow \mathcal{C}$ corresponds exactly to giving an object of $\mathcal{C}$, so $\otimes_{\varphi_0}$ is the constant functor with image ${\mathbf 1}_{\mathcal C}$. We also set $\otimes_{\id_U} := \id_{\mathcal{C}}$ and since for any $\varphi\in E_2(1)$ there is a unique homotopy class of paths from $\varphi$ to $\id_U$, the functor $\otimes_{\varphi}$ is canonically isomorphic to $\id_{\mathcal C}$. Let now $n\geq 2$. The connected components of $E_2^v(n)$ are contractible and can be naturally identified with the permutations of the integers $1, 2, \ldots, n$. Furthermore, the set of homotopy classes of paths in $E_2(n)$ joining a point in a connected component of $E^v_2(n)$ to a point in another connected component of $E^v_2(n)$ can be naturally identified with an appropriate set of colored braids.
    Observe now that, since $\otimes$ satisfies the braid equation, it is readily seen that the structure
    $(\mathcal{C}, \otimes, \mathbf 1_{\mathcal{C}}, R)$ is equivalent to the structure of a $CoB$-algebra on $\mathcal{C}$. This implies that one can define $\otimes_\varphi$ for $\varphi\in E^v_2(n)$ in terms of tensor product of $n$ objects reordered appropriately. Thanks to the equivalence of $(\Pi_1 E_2)^v$ and $\Pi_1 E_2$ we extend the $(\Pi_1 E_2)^v$-algebra structure on $\mathcal{C}$ to a $\Pi_1 E_2$-algebra structure on $\mathcal{C}$.

\subsection{Tensor products and braidings for perverse sheaves}

From now on the braided  monoidal category $\mathcal{V}$ with braiding $R$ will always be abelian and $k$-linear and the tensor product $\otimes$ will always be exact. 

\medskip

We denote by $\mathcal{S}trat$ \label{p:strat} the category of stratified spaces, we refer to \cite[1.1, 1.2]{GMcP} for unexplained terminology. Objects in $\mathcal{S}trat$ are given by pairs $(Y,S)$ where $Y$ is a topological space and $S$ is a stratification of $Y$ with even (real) dimensional strata.  The morphisms are given by stratified maps.  For a stratified space  $(Y,S)$  we consider the constructible bounded derived category $D^b _{\! c} (Y, S, \mathcal{V})$ of complexes of sheaves on $Y$ with values in $\mathcal{V}$ and cohomology locally constant along $S$,  \cite[\S 1.4 A]{KS3}. 
It has a formalism of the six operations: $\mathcal{H}om$, $\otimes$, $f_\ast$, $f_{!}$, $f^\ast$, $f^!$, see \cite{dCM}. 
We consider the abelian subcategory $\mathcal{PS}(Y,S, \Vu)$ of
${D^b _c}(Y,S, \Vu)$  consisting of  (middle) perverse sheaves.  Perverse sheaves are well-behaved under duality. We will make use of the  perverse truncation functors ${}^{p}\tau_{\geq 0}$, ${}^p\tau_{\leq0}$ and ${}^p{(\ )} ={}^{p}\tau_{\leq 0}{}^{p}\tau_{\geq 0}(\ )$ on the category ${D^b _c}(Y, S, \Vu)$. For the case of coefficients over a field,  see \cite{BBD}, \cite[\S 5.8]{dCM}. 

\medskip

    An indexed countable direct sum of objects
    $Y =  \coprod_{n\in \mathbb{N}} Y(n)$ in a category $\mathcal C$ is said to be graded (more precisely $\mathbb{N}$-graded). The object $Y(n)$ indexed by $n$ is called the graded component of degree $n$. If $\mathcal C=\mathcal{S}trat$, we obtain the category of graded stratified spaces,  whose morphisms are graded of degree zero, i.e., map graded components to graded components corresponding to the same index. This category is symmetric  if endowed with the operation of direct product of spaces, the product stratification, and the usual flip.
Indeed, if $Y = \coprod_{p\in\NN} Y(p)$ and 
$Y' = \coprod_{q\in\NN} Y'(q)$, then
\begin{equation}\label{eq:grading=product}
Y\times Y' = 
\coprod_{n\in\NN} \left(\coprod_{p+q=n} 
Y(p) \times Y'(q)\right),
\end{equation}
is a grading of $Y\times Y'$, compatible with the product stratification, and it is straightforward to verify the conditions of a symmetric category. 
Given a graded space $Y = \coprod_{n\in\NN} Y(n)$, for any $d\in \NN$ the subspace $Y({\leq d}) := \coprod_{n\leq d}Y(n)$ is open in $Y$.

\subsubsection{The outer tensor product $\boxtimes$} \label{tensprod}
For any stratified space $(Y, S)$, the category of sheaves on $Y$ with values in $\mathcal{V}$ is endowed with a braided monoidal structure, which readily extends to complexes of sheaves on $Y$ with values in $\mathcal{V}$, provided $R$ is modified by a sign according to the Koszul rule. By bi-exactness of $\otimes$, the structure passes to  $D^b_{\! c} (Y, S, \mathcal{V})$, making $D^b_{\! c} (Y, S, \mathcal{V})$ a braided monoidal category \cite[\S 3.2 B]{KS3}. By a slight abuse of notation, we denote the induced tensor product by $\otimes$. By Deligne's philosophy, there is a functor $\otimes_\varphi$ for any $\varphi\in E_2(n)$. For any stratified morphism $f\colon (Y, S)\to (Y', S')$ the pull-back functor $f^*\colon D^b_{\! c} (Y', S', \mathcal{V})\to D^b_{\! c} (Y, S, \mathcal{V})$ is braided monoidal, where the coherence maps are the natural identifications.

A drawback in this construction is that the subcategory  of perverse sheaves is not closed with respect to $\otimes$.
On the other hand,  it is suitable for defining an outer tensor product $\boxtimes$, whose construction we recall from {\it loc. cit.}.

\medskip

For 
$\mathcal{F}\in 
{D}^b_{\! c} (Y, S, \mathcal{V})$
and 
$\mathcal{G}\in 
{D}^b_{\! c} (Y', S', \mathcal{V})$, 
 the object $\mathcal{F} \boxtimes \mathcal{G}$ of ${D}^b_{\! c} (Y\times Y', S\times S', \mathcal{V})$ is defined as:
\begin{equation}\label{eq:outer}
\mathcal{F} \boxtimes \mathcal{G} := pr_Y^\ast \mathcal{F} \otimes pr_{Y'}^\ast \mathcal{G},
\end{equation}
where $pr_Y, pr_{Y'}$ are the canonical projections from  $Y\times Y'$ to $Y$ and $Y'$, respectively, and the tensor product $\otimes$ is induced by the tensor product of complexes of sheaves with values in $\Vu$. By \cite[\S 3.2 B] {KS3} the external product of perverse sheaves is again perverse.

\medskip

For any $\varphi\in E_2(n)$ and any $n$-tuple of stratified spaces $(Y_1,S_1)$,$\ldots$, $(Y_n,S_n)$, with natural projections $p_i\colon Y_1\times \cdots\times Y_n\to Y_i$ for $i=1,\,\ldots,\,n$, Deligne philosphy applied to $D^b_c(\prod _{i=1}^nY_i,\prod_{i=1}^n S_i,\Vu)$ gives then a functor \label{boxphi}
\begin{equation*}
 \boxtimes_{\varphi}:=  \otimes_{\varphi}(p_1^*,\ldots,\,p_n^*)= \colon \prod_{i=1}^ nD^b_c(Y_i,S_i,\Vu)\to D^b_c(\prod _{i=1}^nY_i,\prod_{i=1}^n S_i,\Vu)
\end{equation*}
for any $\varphi\in E_2(n)$. These functors are compatible with operadic composition in $E_2$ and induce functors  \begin{equation}\label{eq:boxphi}\boxtimes_{\varphi}\colon \prod_{i=1}^n{\PS}(Y_i,S_i, \mathcal{V})\to  \PS(\prod _{i=1}^nY_i,\prod_{i=1}^n S_i,\Vu).\end{equation}

Furthermore, for every path $\gamma$  from $\varphi$ to $\varphi'$  in $E_2(n)$ there is an isomorphism of functors
$\underline{R}_\gamma: \boxtimes_\varphi 
\rightarrow
\boxtimes_{\varphi'}$  given on any $n$-tuple of objects $\GG_i\in D^b _{\! c} (Y_i, \mathcal{V})$ for $i=1,\,\ldots,\,n$,  by
\begin{equation}\label{eq:defR}
\underline{R}_{\gamma, (\GG_1,\,\ldots, \GG_n)}:=R_{\gamma;(p_1^*\GG_1,\ldots,p_n^*\GG_n)}
\end{equation}
where $R_\gamma$ is the braiding in 
$D^b _{\! c} (Y_1\times \cdots\times Y_n,  S_1\times\cdots\times S_n, \mathcal{V})$. 
The isomorphisms $\underline{R}_\gamma$ depend only on the homotopy class of the path $\gamma$ and for any $\sigma\in \Sn$ and any $n$-tuple of objects $\GG_1,\,\ldots,\,\GG_n$ there holds 
\begin{equation}\label{eq:equivarianceR}
\underline{R}_{\gamma^{\sigma}; (\GG_1,\ldots,\,\GG_n)}=\underline{R}_{\gamma;\sigma(\GG_1,\,\ldots,\,\GG_n)},
\end{equation} by Theorem~\ref{thm:del} applied to the braided monoidal category $D^b _{\! c} (Y_1\times \cdots\times Y_n, S_1\times\cdots\times S_n, \mathcal{V})$.

\begin{lem}\label{lem:equivarianceRbox}Let $\varphi$ and $\psi\in E_2(n)$ and let $\gamma$ be a path from $\varphi$ to $\psi$. Let $(Y,S)$ be a stratified space. For $\sigma\in\Sn$, let $\sigma^*\colon D^b_{\! c}(Y^n, S^n,\Vu)\to D^b_{\! c}(Y^n, S^n,\Vu) $ be the pull-back functor of the permutation morphism $\sigma\colon Y^n\to Y^n$. Then, for any $n$-tuple of objects $\FF_i\in D^b_{\! c}(Y,S,\Vu)$ for $i=1,
\ldots, n$ there holds
\begin{equation}
    \sigma^*\underline{R}_{\gamma;(\FF_1,\ldots,\,\FF_n)}=\underline{R}_{\gamma^\sigma;(\FF_{\sigma^{-1}(1)},\ldots,\,\FF_{\sigma^{-1}(n)})}.
\end{equation}
    \end{lem}
\begin{proof}
Applying \eqref{eq:defR} gives
\(\sigma^*\underline{R}_{\gamma;(\FF_1,\ldots,\,\FF_n)}=\sigma^*R_{\gamma;(p_1^*\FF_1,\ldots,\,p_n^*\FF_n)}\).
Since $\sigma^*$ is the pull-back along a morphism, it is  braided monoidal, hence up to natural equivalences
\begin{align*}
\sigma^*R_{\gamma;(p_1^*\FF_1,\ldots,\,p_n^*\FF_n)}&=R_{\gamma;(\sigma^*p_1^*\FF_1,\ldots,\,\sigma^*p_n^*\FF_n)}=R_{\gamma;((p_1\sigma)^*\FF_1,\ldots,\,(p_n\sigma)^*\FF_n)}=R_{\gamma;(p_{\sigma(1)}^*\FF_1,\ldots,\,p_{\sigma(n)}^*\FF_n)}\\
&=R_{\gamma;\sigma(p_1^*\FF_{\sigma^{-1}(1)},\ldots,\,p_n^*\FF_{\sigma^{-1}(n)})}=R_{\gamma^\sigma;(p_1^*\FF_{\sigma^{-1}(1)},\ldots,\,p_n^*\FF_{\sigma^{-1}(n)})}\\
&=\underline{R}_{\gamma^\sigma;(\FF_{\sigma^{-1}(1)},\ldots,\,\FF_{\sigma^{-1}(n)}).}
\end{align*}
where for the second last equivalence we have used \eqref{eq:equivarianceR}.
\end{proof}

\begin{remark}\label{rem:D-monoidal}{\rm Combining the functors $\boxtimes$ when $(Y,S)$ and $(Y',S')$ vary, gives a bi-functor on the category ${D}^b _{\! c} (\phantom{a}, \mathcal{V})$ 
whose objects are pairs $((Y,S), C)$ where $(Y,S)$ is a stratified space and $C$ is an object in  $D^b _{\! c} (Y, S, \mathcal{V})$ and whose morphisms are pairs $(f,g)$ where $f$ is a morphism of stratified spaces and $g$ is an $f$-morphism. The category is
fibered over $\mathcal{S}trat$ and it can be shown that the tensor product and associativity constraints are compatible with the fiber structure, making  $D^b _{\! c} (\phantom{a}, \mathcal{V})$ into a monoidal fibered category, see \cite{ayoub,ter} for precise definitions. Similarly, one defines the fibered subcategory $\mathcal{P}erv(\phantom{a}, \mathcal{V})$ of perverse sheaves, which is a  monoidal fibered subcategory of $D^b _{\! c} (\phantom{a}, \mathcal{V})$. Dropping condition \cite[Definition 2.6, (cITS-1)]{ter}, one obtains the notion of braided monoidal fibered category. 

For stratified spaces $(Y,S)$ and $(Y',S')$, let $pr_Y, pr_{Y'}, pr'_Y$, and  $pr'_{Y'}$ be the canonical projections from  $Y\times Y'$ to $Y$ and $Y'$, and from $Y'\times Y$ to $Y$ and $Y'$, respectively. For
$\mathcal{F}\in 
{D}^b_{\! c} (Y, S, \mathcal{V})$
and 
$\mathcal{G}\in 
{D}^b_{\! c} (Y', S', \mathcal{V})$, the natural intertwining map 
\begin{equation}\label{eq:braiding-gira}
    \mathcal{F} \boxtimes \mathcal{G}= pr_Y^\ast \mathcal{F} \otimes pr_{Y'}^\ast \mathcal{G}
 \rightarrow
 (pr'_{Y'})^\ast \mathcal{G} \otimes (pr'_Y)^\ast \mathcal{F}=\mathcal{G} \boxtimes \mathcal{F} 
\end{equation}
is an isomorphism of complexes $D^b_{\! c} (Y\times Y', S\times S', \mathcal{V})\to D^b_{\! c} (Y'\times Y, S'\times S, \mathcal{V})$ 
in  $D^b _{\! c} (\phantom{a}, \mathcal{V})$ 
over the symmetry
\[
\sigma\colon Y\times Y' \rightarrow Y' \times Y \ , \ 
(y,y')\mapsto (y',y).
\]

The natural transformations \eqref{eq:braiding-gira} combine to turn $D^b _{\! c} (\phantom{a}, \mathcal{V})$ and 
$\mathcal{P}erv(\phantom{a}, \mathcal{V})$ into braided monoidal fibered categories. The functors $\boxtimes_{\varphi}$ could be seen as an application of Deligne philosophy to the braided monoidal structures in these categories.} 
\end{remark}

Let $(Y,S)$ be a stratified space. For $n\in\NN_{\geq 2}$, any $\sigma \in\Sn$ acts on $(Y^n,S^n)$ so the pull-back functor yields functors $\sigma^*\colon
\mathcal{PS}(Y^n,S^n,\Vu)\to \mathcal{PS}(Y^n,S^n,\Vu)$.  We spell out the compatibility of the functors $\boxtimes_{\varphi}$ with respect to the permutation actions on $E_2(n)$ and  $\mathcal{PS}(Y^n,S^n,\Vu)$.

\begin{lem}\label{lem:tensoreq}
 Let $(Y,S)$ be a stratified space, $n\in\NN_{\geq 2}$, $\varphi\in E_2(n)$, and $\sigma\in \Sn$. Consider $\boxtimes_{\varphi}: \mathcal{PS}(Y,S,\Vu)^n \rightarrow \mathcal{PS}(Y^n,S^n,\Vu)$. Then the following equivalence of functors holds
 \begin{equation}\label{eq:boxtimes-sigma}\sigma^*\circ \boxtimes_{\varphi}=\boxtimes_{\varphi^{\sigma}}\circ \sigma^{-1}.\end{equation} In particular, if $\FF \in \mathcal{PS}(Y,S,\Vu)$, 
then \begin{equation}\label{eq:identification}\boxtimes_{\varphi^\sigma}(\FF,\ldots,
\,\FF)=\sigma^*(\boxtimes_{\varphi}(\FF,\,\ldots,\,\FF)).\end{equation}
\end{lem}
\begin{proof}
 Let $\mathcal{F}_1, \dots, \mathcal{F}_n \in \mathcal{PS}(Y,S,\Vu)$. Let $p_l\colon Y^n\to Y$ be the
  $l$-th canonical projection and let  $\otimes_{\varphi}: \mathcal{PS}(Y^n, S^n,\mathcal{V})^n \to \mathcal{PS}(Y^n,S^n, \mathcal{V})$ be as in \eqref{eq:boxphi}. Then 
    \begin{equation}
    \begin{split}
    \sigma^*\circ \boxtimes_{\varphi}(\mathcal{F}_1,\dots, \mathcal{F}_n)&=
    \sigma^*\left(\otimes_{\varphi}(p_1^*\mathcal{F}_1, \dots, p_n^*\mathcal{F}_n)\right)=
    \otimes_{\varphi}(\sigma^*p_1^*\mathcal{F}_1, \dots, \sigma^*p^*_n\mathcal{F}_n)\\
    &= \otimes_{\varphi}(p_{\sigma(1)}^*\mathcal{F}_1, \dots, p^*_{\sigma(n)}\mathcal{F}_n)= \otimes_{\varphi^{\sigma}}(p_1^*\mathcal{F}_{\sigma^{-1}(1)}, \dots, p^*_n\mathcal{F}_{\sigma^{-1}(n)})\\
    &=\boxtimes_{\varphi^{\sigma}}\circ \sigma^{-1}(\mathcal{F}_1, \dots, \mathcal{F}_n)
         \end{split}
    \end{equation}
 where the second equality follows from functoriality of $\sigma^*$ and in the last equality we have used that $\otimes_{\varphi^{\sigma}}=\otimes_{\varphi}\circ \sigma$. A similar argument holds for morphisms. 
\end{proof}

\subsubsection{Some equivalences} \label{mndrmy}
We recall a few basic equivalences and natural isomorphisms for $D^b_c(Y,S,\Vu)$ that will be needed in the sequel for the definition of a monodromy functor.

\medskip

For $a,\,b\in {\mathbb R}$ and $l\in \NN_{\geq 1}$, we consider $(a,b)$ as a  stratified space with trivial stratification, that we denote by $S_0$. For $u\in (a,b)^l$, and $Y$ a topological space, we denote by  \[i_u : Y \rightarrow Y\times (a,b)^l\] the inclusion $y\mapsto (y,u)$.
If $(Y,S)$ is stratified, then $Y\times (a,b)^l$ is equipped with the product stratification $S\times S_0^l$.

\begin{lem}\label{lem:equiv}
Let $(Y,S)$ be a stratified space and let 
\[
\pi: Y\times (a,b)^l \rightarrow Y
\]
be the canonical projection. 
Then, for any $u\in (a,b)^l$  the functors $\pi^\ast$ and $i_u^\ast$ establish mutually quasi-inverse equivalences between 
$D^b_{\! c} (Y, S, \mathcal{V})$ and $
D^b_{\! c} (Y\times (a,b)^l, S\times S_0^l, \mathcal{V})$.
\end{lem}
\begin{proof}
    It suffices to prove the statement for $l=1$, the general case follows from iteration. The functor $\pi^\ast$ is an equivalence by \cite[Lemma~4.1]{Treu}. Furthermore, the equality
    \[
    \pi \circ i_u = id_Y
    \]
    implies that there is an equality $i_u^\ast \circ \pi^\ast = id_{D^b_{\! c} (Y, S, \mathcal{V})}$. As $i_u^\ast$ is a quasi-inverse to $\pi^\ast$, it is also an equivalence.
 \end{proof}
\begin{prop} \label{lem:equiv2}Let $(Y,S)$ and $(Y',S')$ be  stratified spaces. For a real interval $(a,b)$ and $l\in\NN$, let 
\[
\alpha: Y\times (a,b)^l \rightarrow Y'
\]
be a stratified map and $\pi\colon Y\times (a,b)^l\to Y$ be the natural projection. For $u\in (a,b)^l$, let $\alpha_u:=\alpha\circ i_u$ and let $\alpha_u^*\colon D^b_{\! c} (Y', S', \mathcal{V})\to D^b_{\! c} (Y, S, \mathcal{V})$ be its pull-back functor. 
Then, for every $u,\,v\in (a,b)^l$, and every $\FF$ in $D^b_c(Y',S',\Vu)$, the functors $i_u^*$ and $i_v^*$ induce identifications
\begin{equation}\label{eq:natural}
\begin{tikzcd}
\mathrm{Hom}_{D^b_c(Y\times (a,b)^l,S\times S_0^l,\Vu)}(\pi^*\alpha_u^*\FF,\alpha^*\FF)\arrow[dd, swap,"i_u^*"] \arrow[rdd, "i_v^*"] \\
&&\\
\mathrm{Hom}_{D^b_c(Y,S,\Vu)}(\alpha_u^*\FF,\alpha_u^*\FF)&
\mathrm{Hom}_{D^b_c(Y,S,\Vu)}(\alpha_u^*\FF,\alpha_v^*\FF)
\end{tikzcd} 
\end{equation}
and the images in $\mathrm{Hom}_{D^b_c(Y,S,\Vu)}(\alpha_u^*\FF,\alpha_v^*\FF)$ of the identity in $\mathrm{Hom}_{D^b_c(Y,S,\Vu)}(\alpha_u^*\FF,\alpha_u^*\FF)$ for all $\FF$ combine to give a natural isomorphism of functors $\alpha_u^\ast  \rightarrow \alpha_v^\ast$. 
\end{prop}
\begin{proof}
Let $u,\,v\in (a,b)^l$ and let $\mathcal{F}$ be an object 
of 
$D^b_{\! c} (Y', S', \mathcal{V})$. Its pull-back $\alpha^\ast \mathcal{F}$  and $\pi^\ast\alpha_u^\ast \mathcal{F}$ 
are objects of 
$D^b_{\! c} (Y\times (a,b)^l, S\times S^l_0, \mathcal{V})$. Applying the fully faithful functor $i_u^\ast$ to both gives the object 
$\alpha_u^\ast\mathcal{F}$, and the natural identification of Hom's on the left. Applying the fully faithful functor $i_v^*$ making use of the equivalence $\pi\circ i_v=\id_Y$, we obtain the natural identification of Hom's on the right. The last statement follows from naturality of the above identifications. 
\end{proof}

We call the isomorphisms of functors as in Proposition \ref{lem:equiv2} the monodromy isomorphism.

\begin{remark}\label{rem:open-embedding}{\rm
\begin{enumerate}
    \item 
Let $(a,b)$ be a real interval, $(Y,S)$, $(Y',S')$ and $(Y{''},S{''})$ be  stratified spaces, and let
\[
\alpha: Y\times (a,b) \rightarrow Y',\quad{\rm and} \quad f\colon Y'\to Y^{''} 
\]
be stratified morphisms, with $f$ an open embedding. Then $\alpha'=f\circ \alpha\colon Y\times (a,b) \rightarrow Y^{''}$ is also a stratified morphism. A direct verification shows that for $u,\,v\in (a,b)$  the monodromy isomorphisms of functors $M_{u,v}$ from $(\alpha\circ i_u)^*$ to $(\alpha\circ i_v)^*$ and $M'_{u,v}$ from $(\alpha'\circ i_u)^*$ to $(\alpha'\circ i_v)^*$  satisfy 
$M'_{u,v,\FF}=M_{u,v,f^*\FF}$ for any object $\FF$ in $D^b_c(Y^{''},S^{''},\Vu)$.
\item The monodromy isomorphism of functors $\alpha_u^\ast  \rightarrow \alpha_v^\ast$ is compatible with shifts in the derived categories, that is, identifying 
\begin{align*}
\mathrm{Hom}_{D^b_c(Y,S,\Vu)}(\alpha_u^*(\FF[1]),\alpha_v^*(\FF[1]))&=\mathrm{Hom}_{D^b_c(Y,S,\Vu)}((\alpha_u^*\FF)[1],(\alpha_v^*(\FF))[1])\\
&=\mathrm{Hom}_{D^b_c(Y,S,\Vu)}(\alpha_u^*(\FF),\alpha_v^*(\FF))[1]
\end{align*}
there holds $M_{u,v,\FF[1]}=M_{u,v,\FF}[1]$.
\end{enumerate}}
\end{remark}

\section{Perverse sheaves on \texorpdfstring{$\sym(\mathbb C)$}{Sym(C)}}
\subsection{Basic geometric definitions}\label{sec:sym-prod}
In this section we introduce the relevant spaces on which we will consider perverse sheaves. 
Unless otherwise stated, $n,m$ and $d$ are elements in $\NN$. 

  For a topological space $Y$, 
    the $n$-th symmetric product of $Y$, denoted by $\sym^n (Y)$, is defined to be the quotient $Y^n/{\mathbb S}_n$ of the space $Y^n$ by the natural action of the symmetric group ${\mathbb S}_n$, with the convention that $\sym^0(Y)$ for any $Y$ consists of a single point that we denote by ${\mathbf s}_0$ \label{s0} and $\sym^n(\emptyset)=\emptyset$ for any $n\geq 1$. Then $\mathrm{Sym}(Y):=\coprod_{n\geq 0}  \mathrm{Sym}^n (Y)$ is a graded topological space. If $Y$ is a non-empty open subset of $\mathbb C$, then $\mathrm{Sym}(Y)$ can be identified with the space of monic polynomials with roots in $Y$ and
    $\mathrm{Sym}^n (Y)$ with the subset of monic polynomials of degree $n$.  
    If $Y=\mathbb C$, then each graded component is an affine space of dimension $n$ by the Chevalley-Shephard-Todd theorem. 

\medskip

Recall that $I=(0,1)$ and $U=I\times I\subset{\mathbb R}^2 \simeq{\mathbb C}$. The spaces we focus on are

\begin{equation}\label{eq:X}
 \sym(\mathbb C),\quad \sym^n(U),\quad \sym(U),
\end{equation}
and, more generally, $\sym(U')$ and $\sym^n(U')$ for any open subset $U'$ of ${\mathbb C}$. The spaces  $\sym^n(U')$ are stratified by the so-called diagonal stratification \label{stratification} $\Sigma$, that is, multisets of $n$ points in $U'$ are partitioned into strata according to the multiplicities of occurrence of each point. Then, strata are in bijection with the set $\mathrm{P}(n)$ of  partitions of $n$. For $\lambda\in{\mathrm P}(n)$ and $U'\subset \mathbb C$ we denote by $\sym_\lambda(U')$ the corresponding stratum in $\sym(U')$. 

\medskip

For a non-empty open subset $U'$ of $\CC$, we denote by $\mathcal{P}_{U'}(n)$ the abelian category of perverse sheaves on $(\mathrm{Sym}(U'))^n$  with values in $\mathcal{V}$, smooth along the  stratification induced by multiplicities. \label{page:P(n)}

\begin{lem}\label{lem:eqperv}Let $U'$ be a non-empty open rectangle in $\CC$ and let $\iota\colon\sym(U')\to\sym(\CC)$ be the natural inclusion.
    Then, for any $n\in\NN$ the functor 
    $$(\iota^n)^\ast : \mathcal{P}_{\mathbb{C}}(n) \rightarrow \mathcal{P}_{U'}(n)$$
    is an equivalence.
\end{lem}
\begin{proof}
By \cite[Theorem~4.4]{Treu}  the stack of perverse sheaves on $\mathrm{Sym}(\mathbb{C})^n$ is constructible, \cite[Definition~3.8]{Treu}. Furthermore, the inclusion $\mathrm{Sym}(U')^n\rightarrow \mathrm{Sym}(\mathbb{C})^n$ is a stratified homotopy equivalence, \cite[Definition~3.7]{Treu}. Therefore the statement follows from \cite[Theorem~3.13]{Treu}.
\end{proof}

According to Lemma~\ref{lem:eqperv}, there is no loss of information by restricting perverse sheaves on $\mathrm{Sym}(\mathbb{C})$ to the open subset $\mathrm{Sym}(U)$. In particular, in most of the results in \cite{KS3} we may replace $\mathrm{Sym}(\mathbb{C})$ by $\sym(U)$. For this reason, we will mainly focus on $\sym(U)$ instead of  $\mathrm{Sym}(\mathbb{C})$.  For simplicity of notation, we will drop $U$ as subscript, that is, we set $\mathcal{P}(1):=\mathcal{P}_U(1)$.

For any $d\in\NN$ and any open subset $U'\subseteq {\mathbb C}$ we consider the dense open subset $\mathrm{Sym}^n_{\leq d}(U')$ of $\sym^n(U')$ consisting of multisets of points in $U'$ with multiplicity bounded by $d$.
We then have the following open truncations of $\sym(U')$:
\begin{equation}\label{eq:dense-truncation}\mathrm{Sym}_{\leq d}(U')  =  \coprod_{n\in\NN}\mathrm{Sym}^n_{\leq d}(U'),\quad \mathrm{Sym}^{\leq d}(U')
    =  \coprod_{n\in\NN_{\leq d}}\mathrm{Sym}^n(U').\end{equation}
    We observe that, while $\sym^{\leq d}(U')$ is finite-dimensional, the  open subset $\sym_{\leq d}(U')$ is dense in $\sym(U')$ for any $d\in\NN_{\geq1}$. Indeed, $\sym_{\neq}(U'):=\sym_{\leq 1}(U')$ is the configuration space of unordered points in $U'$, which is dense in $\sym(U')$, and $\sym_{\neq}(U')\subseteq\sym_{\leq d}(U')$ for any $d\in\NN_{\geq 1}$.

    \medskip
   
    By construction, $\sym^{\leq d}(U')$ and $\sym_{\leq d}(U')$ are unions of strata of $\sym(U')$.
More precisely, for $\lambda$ a finite sequence of non-negative integers, we denote by $d(\lambda)$ the depth \label{depth} of $\lambda$, that is, the maximum value of its components, and by $|\lambda|$ the sum of its components. Then, for $U'=U$ we have:

\begin{equation}\label{eq:truncated-strata}
\mathrm{Sym}^{\leq d}(U)=\coprod_{n\in\NN_{\leq d}}\coprod_{\lambda\in \mathrm{P}(n)}\sym_\lambda(U),\quad 
\sym_{\leq d}(U)=\coprod_{n\in\NN}\coprod_{\genfrac{}{}{0pt}{2} {\lambda\in \mathrm{P}(n),}{d(\lambda)\leq d}}\sym_\lambda(U).\end{equation}

If $U', U''$ are open in $\mathbb{C}$, with $U'\subset U''$, then the natural map \(\mathrm{Sym}(U')\rightarrow \mathrm{Sym}(U'')\)
is an open inclusion of stratified spaces, and similarly for the truncated analogues. 

In addition, for $c,\,d\in \NN$ with $d\leq c$
we denote by
\begin{align}\label{eq:j-inclusions}
j_{\leq d,\leq c}&\colon \mathrm{Sym}_{\leq d}({\mathbb C})
    \hookrightarrow
    \mathrm{Sym}_{\leq c}({\mathbb C})
&\mbox{ and }&&j_{\leq d}\colon  
    \mathrm{Sym}_{\leq d}({\mathbb C})
    \hookrightarrow
    \mathrm{Sym}({\mathbb C})\end{align}
    the canonical inclusions. We will use the same notation to denote their restrictions to the graded component $\sym^n_{\leq d}(U)$. Clearly, 
    \begin{equation*}j_{\leq d}=j_{\leq c}\circ j_{\leq d,\leq c} \mbox{ and } j_{\leq d,\leq c}\circ j_{\leq e,\leq d}= j_{\leq e,\leq c}\mbox{ for }e\leq d\leq c.\end{equation*}

\medskip 
The symmetric product yields a braided  monoidal functor
    $$\mathrm{Sym}(-): \mathrm{Top}^{\coprod}
    \rightarrow  \mathrm{Top}^{\prod}\  $$
sending open embeddings to open embeddings, hence, applying the functor $\sym$ to a linear embedding $\varphi: U^{\coprod n}\rightarrow U$ one gets an open embedding,  called the multiplication embedding,
\begin{equation}\label{a-phi}
a_\varphi : \mathrm{Sym}(U) ^n \rightarrow \mathrm{Sym}(U).
\end{equation}
The maps $a_\varphi$, for $\varphi\in E_2(n)$ may be composed and form an operad isomorphic to $E_2$, and with values in $\mathrm {Top}^{\prod}$. For each $n$, the isomorphism of operads sends a linear embedding $\varphi$ to the multiplication embedding $a_\varphi$. The operadic composition on these $a_\varphi$ is obtained by applying $\mathrm{Sym}$ to the diagram (\ref{eq:linemb}), yielding the diagram

\begin{equation}\label{eq:symemb}
\begin{tikzcd}[row sep=huge]
\mathrm{Sym}(U)^{ n} \arrow[rr, "a_\varphi"]  & &\mathrm{Sym}(U) \\
\prod_{l=1}^{n} \mathrm{Sym}(U)^{ m_l} \arrow[u, "\prod_{l=1}^n a_{\psi_l}"] \arrow[rr, "="] & &
\mathrm{Sym}(U)^{ m} \arrow[u, "a_\eta"]
\end{tikzcd} 
\end{equation}
Observe that since  for any $\sigma\in {\mathbb S}_n$ and any $\varphi\in E_2(n)$, there holds \begin{equation}\label{eq:equivariance-a-phi}a_{\varphi^\sigma}=a_{\varphi}\circ\sigma\end{equation} 
the above mentioned isomorphism is of symmetric operads.
The restriction of the multiplication embedding $a_\varphi$ to 
    $(\sym_{\leq d}(U))^n$, or equivalently, the application of the functor 
    $\mathrm{Sym}_{\leq d}(-)$ to the linear embedding $\varphi$, gives
the open embedding 
    \begin{equation}\label{a-phi-leq}
    a_{\varphi,\leq d} : 
    (\sym_{\leq d}(U))^n
    \rightarrow
   \sym_{\leq d}(U).
    \end{equation}

Since $\sym(U)$ is a graded space,  $(\sym(U))^n$ is graded for every $n\geq 1$. The open subset $\sym(U)^n(\leq d)$ of $\sym(U)^n$ is given by
     \[\sym(U)^n({\leq d})=
\coprod_{l\in\NN_{\leq d}} \coprod_{m_1+\cdots+m_n=l} \mathrm{Sym}^{m_1}(U) \times\cdots\times \mathrm{Sym}^{m_n}(U).\]
    Moreover, for every linear embedding $\varphi: U^{\coprod n}\rightarrow U$, the corresponding map 
    \(a_\varphi \)
     is a graded map.  The restriction to $\sym(U)^n({\leq d})$ of the (graded) morphism  $a_{\varphi}:\sym(U)^n \rightarrow \sym(U)$ gives a morphism  \begin{equation}\label{a-phi-leq-alto}a_{\varphi}^{\leq d}\colon (\sym(U)^n)({\leq d}) \longrightarrow 
\sym^{\leq d}(U)\end{equation}

     Analogous assertions hold replacing $U$ by any open subset $U'$ of ${\mathbb C}$.

\subsection{Relevant categories and functors}
In this section we introduce the  notions and  properties that are instrumental for the definition of factorized perverse sheaves on $\sym(U)$, $\mathrm{Sym}_{\leq d}(U)$ and $\sym^{\leq d}(U)$ in  Section \ref{facdata}. 

\subsubsection{Restriction and extension}\label{sec:pervcat}

 For $d\in \NN$ and $n\in\NN_{\geq1}$ we denote by \label{perverse-X}
$\mathcal P(n)$,   $\mathcal{P}_{\leq d}(n)$ and $\mathcal{P}^{\leq d}(n)$  the categories of perverse sheaves on  $\sym(U)^n$,
    $(\sym_{\leq d}(U))^n$ and $(\sym(U))^n({\leq d})$, respectively. We also consider the categories $\mathcal{P}=\coprod_{n\in\NN} \mathcal{P}(n)$,  $\mathcal{P}_{\leq d}:= \coprod_{n\in\NN} \mathcal{P}_{\leq d}(n)$ and  $\mathcal{P}^{\leq d}:=\coprod_{n\in\NN} \mathcal{P}^{\leq d}(n)$, respectively. These categories are related by the following functors. 

    \medskip
    
 Let $c,d\in\NN$ with $d\leq c$. The pull-back of the open embeddings 
    \begin{align}\label{eq:embedding-j}j_{\leq d}^n\colon 
    (\sym_{\leq d}(U))^n
    \hookrightarrow
    \sym(U)^n,&&j_{\leq d,\leq c}^n\colon  (\sym_{\leq d}(U))^n\hookrightarrow (\sym_{\leq c}(U))^n\end{align} yields
functors \label{p:funtori-j}
\begin{align*}
&(j^n_{\leq d})^\ast \colon \mathcal{P}(n) 
\rightarrow
\mathcal{P}_{\leq d}(n)& (j^n_{\leq d,\leq c})^\ast \colon \mathcal{P}_{\leq c}(n) 
\rightarrow
\mathcal{P}_{\leq d}(n).
\end{align*}
We will sometimes omit the superscript $n$ to lighten the notation.  The left and right adjoints of these functors are respectively
\[
^p (j^n_{\leq d})_! : 
\mathcal{P}_{\leq d}(n) \
\rightarrow
\mathcal{P}(n) 
,\quad\textrm{ and }\quad
^p (j^n_{\leq d})_\ast : 
\mathcal{P}_{\leq d}(n) \
\rightarrow
\mathcal{P}(n).
\]
and
\[
^p (j^n_{\leq d,\leq c})_! : 
\mathcal{P}_{\leq d}(n) \
\rightarrow
\mathcal{P}_{\leq c}(n) 
,\quad\textrm{ and }\quad
^p (j^n_{\leq d,\leq c})_\ast : 
\mathcal{P}_{\leq d}(n) \
\rightarrow
\mathcal{P}_{\leq c}(n),
\]
\cite[\S 4.2.4]{BBD}, where $^{p}={}^p\tau_{\leq 0}\circ {}^p\tau_{\geq 0}$ is the perverse truncation, \cite[5.3]{dCM}. The latter is needed to ensure that we obtain a perverse sheaf because the open embeddings are not affine for $d>0$. Since $(j_{\leq d,\leq c})_*$
 and $(j_{\leq d})_*$ are left $t$-exact, the perverse truncations in these cases are $^{p}(j_{\leq d,\leq c})_*={}^p\tau_{\leq 0}(j_{\leq d,\leq c})_*$ and $^{p}(j_{\leq d})_*={}^p\tau_{\leq 0}(j_{\leq d})_*$. Dually, ${}^{p}(j_{\leq d,\leq c})_!={}^p\tau_{\geq 0}(j_{\leq d,\leq c})_!$ and ${}^{p}(j_{\leq d})_!={}^p\tau_{\geq 0}(j_{\leq d})_!$

\medskip

Taking the image of \(^p (j^n_{\leq d})_!\FF\) in  \(^p (j^n_{\leq d})_*\FF\), and the image of $^p (j^n_{\leq d,\leq c})_!\FF$  in $^p (j^n_{\leq d,\leq c})_*\FF$ for  $\FF\in\PPP_{\leq d}(n)$  gives the intermediate extension functors, \cite{BBD}\cite[Section 2.7]{dCM}:
\[^p (j^n_{\leq d})_{!\ast} : 
\mathcal{P}_{\leq d}(n) \
\rightarrow
\mathcal{P}(n),\quad\textrm{ and }\quad 
^p (j^n_{\leq d,\leq c})_{!*} : 
\mathcal{P}_{\leq d}(n) \
\rightarrow
\mathcal{P}_{\leq c}(n). 
\]

\medskip

For $d\leq c$, the pull-back functors along the inclusions 
\begin{equation}\label{eq:inclusions-i}
i^{\leq d}\colon \sym(U)^n({\leq d})\rightarrow
\sym(U)^n,\quad i^{\leq d,\leq c}\colon \sym(U)^n({\leq d})\rightarrow
\sym(U)^n({\leq c}),
\end{equation}
give rise to the truncation functors  \label{p:funtori-i}
    \[
    (\phantom{a})^{\leq d} : \mathcal{P}(n) \rightarrow \mathcal{P}^{\leq d}(n),\quad  (\phantom{a})^{\leq d,\leq c} : \mathcal{P}^{\leq c}(n)\rightarrow \mathcal{P}^{\leq d}(n).
    \]
 To lighten notation, we will usually write $(\phantom{a})^{\leq d}$ to indicate also its restriction to $\mathcal{P}_{\leq d}(n)$, and, when the context allows it, to indicate $(\phantom{a})^{\leq d,\leq c}$. 

\medskip
 
By construction, the diagram below is commutative for any $d\leq c$.

\begin{equation}\label{dgm:compatibility-limit}
\begin{tikzcd}[row sep=huge]
\mathcal{P}(n) \arrow[rrd,  "(\phantom{a})^{\leq c}"'] \arrow[rr, swap, "(j_{\leq c})^*"'] && \mathcal{P}_{\leq c}(n)  \arrow[d, "(\ )^{\leq c}"] \arrow[rr,  "(j_{\leq d,\leq c})^\ast"] &&  \mathcal{P}_{\leq d}(n)  \arrow[d, swap,  "(\ )^{\leq d}"']\\
&&\mathcal{P}^{\leq c}(n) \arrow[rr,  "(\ )^{\leq d,\leq c}"'] & & 
\mathcal{P}^{\leq d}(n) 
\end{tikzcd} 
\end{equation}

\subsubsection{The tensor product $\boxtimes_{\varphi}$}\label{sec:tens-phi} In this subsection we endow the categories $\mathcal{P}$,  $\mathcal{P}_{\leq d}$ and  $\mathcal{P}^{\leq d}$, for any $d\in\NN$ with a braided monoidal category structure. 

\medskip

For $n\in\NN$ the construction in \eqref{eq:boxphi} 
gives functors
 \[
\boxtimes_{\varphi}: \mathcal{P}^n
\rightarrow
\mathcal{P}, \quad \textrm{ and }\quad
\boxtimes_{\varphi}: \mathcal{P}_{\leq d}^n
\rightarrow
\mathcal{P}_{\leq d}
\]
indexed by the linear embeddings $\varphi\in E_2(n)$. For any path $\gamma$ from $\varphi$ to $\varphi'$ in $E_2(n)$, the isomorphism of functors $\underline{R}_\gamma$ as in \eqref{eq:defR} gives the braiding.

\medskip

In addition, for $n\in\NN$ and any $\varphi\in E_2(n)$, the external tensor product $\boxtimes_{\varphi}$ of objects in $\mathcal{P}^{\leq d}(n_1), \cdots,  \mathcal{P}^{\leq d}(n_r)$  gives a perverse sheaf on 
$\prod_{i=1}^r \left(\sym(U)^{n_i}({\leq d})\right)$ whose restriction $(-)^{\leq d}$ to the open subset $\sym(U)^{\sum_l n_l}({\leq d})$ is an object in $\mathcal{P}^{\leq d}(\sum_l n_l)$, giving a functor 
\[\overline{{\boxtimes}}_\varphi \colon (\mathcal{P}^{\leq d})^n\to \mathcal{P}^{\leq d}.\]

\medskip

\begin{lem}\label{lem:treni-tensor}Let $c,d,e,n\in \NN$ with $c\geq d\geq e$, $n\in\NN_{\geq1}$, and $\varphi\in E_2(n)$. Let  
$\bullet$ stand for any of the extension functors $!$, $*$ or $!*$. The diagrams below are commutative. 
\begin{equation}\label{dgm-tensor-upperstar}
\begin{tikzcd}[cramped, sep=small] 
 \mathcal{P}(1)^n \arrow[rr,swap,  "(j_{\leq c}^*)^n"']
\arrow[d, swap, "\boxtimes_\varphi"] 
& & 
 \mathcal{P}_{\leq c}(1)^n \arrow[d, "\boxtimes_\varphi"] 
\arrow[rr, swap,  "(j_{\leq d,\leq c}^*)^n"']&&
 \mathcal{P}_{\leq d}(1)^n \arrow[d, "\boxtimes_\varphi"] 
\arrow[rr,  swap,  "(\phantom{a})^{\leq d}"']&&
\mathcal{P}^{\leq d}(1)^n \arrow[d, "\overline{\boxtimes}_\varphi"]\arrow[rr, swap,"(\phantom{a})^{\leq e,\leq d}"']&&\mathcal{P}^{\leq e}(1)^n \arrow[d, "\overline{\boxtimes}_\varphi"] \\
\mathcal{P}(n) \arrow[rr, "(j^n_{\leq c})^\ast"']&&\mathcal{P}_{\leq c}(n) \arrow[rr, "(j^n_{\leq d,\leq c})^*"']&&\mathcal{P}_{\leq d}(n) \arrow[rr, "(\phantom{a})^{\leq d}"']&&\mathcal{P}^{\leq d}(n)\arrow[rr, "(\phantom{a})^{\leq e,\leq d}"']&&\mathcal{P}^{\leq e}(n)
\end{tikzcd}
\end{equation}
 
\begin{equation}\label{dgm:tensor-lower*and!}
\begin{tikzcd}
\mathcal{P}_{\leq d}(1)^n \arrow[d, swap, "\boxtimes_\varphi"] \arrow[rr, swap,  "^{p}(j_{\leq d,\leq c})^n_\bullet"'] &&\mathcal{P}_{\leq c}(1)^n \arrow[d, swap, "\boxtimes_\varphi"] \arrow[rr, swap,  "^{p}(j_{\leq c})^n_\bullet"'] &&\mathcal{P}(1)^n \arrow[d, swap, "\boxtimes_\varphi"] \\
\mathcal{P}_{\leq d}(n)\arrow[rr, swap,  "^{p}(j^n_{\leq d,\leq c})_\bullet"']&&\mathcal{P}_{\leq c}(n)\arrow[rr, swap,  "^{p}(j_{\leq c}^n)_\bullet"']&&\mathcal{P}(n) 
\end{tikzcd} 
\end{equation}
In addition, $\underline{R}$ is compatible with restrictions and extensions. 
\end{lem}
\begin{proof} By Theorem \ref{thm:del} it is enough to show that the ordinary external tensor product $\boxtimes$ is compatible with the restriction, truncation, and extension functors. Let ${\mathbf j}\colon Y\to Y'$ be either $j_{\leq d}\colon \sym_{\leq d}(U)\to \sym(U)$ or $j_{\leq d,\leq c}\colon \sym_{\leq d}(U)\to \sym_{\leq c}(U)$
and let  $p_l\colon  Y^n\to Y$ and $p'_l\colon  (Y')^n\to Y'$ for $l=1,\,\ldots,\,n$ be the natural projections. Transitivity \cite[p. 624]{dCM} gives
$p^*_l{\mathbf j}^*=(\mathbf j^n)^*(p')^*_l$ for each $l$, so commutativity of  \eqref{dgm-tensor-upperstar} follows from the definition of $\boxtimes$ and of $\overline{\boxtimes}$. Then compatibility with $\underline{R}$ follows from the construction of $\boxtimes$ and \eqref{eq:defR}.

\medskip

We consider now extensions, showing compatibility of $\boxtimes$ with $\mathbf{j}_\bullet$ in the case $n=2$ for readability. Let $\FF_1,\FF_2\in D^b_c(Y,\Sigma,\Vu)$, so $\FF_1\boxtimes\FF_2\in D^b_c(Y^2,\Sigma^2,\Vu)$. By \cite[Propositions 1.4.21,2.9.1]{achar}  there holds $(\mathbf j^2)_\bullet(\FF_1\boxtimes\FF_2)=\mathbf j_\bullet\FF_1\boxtimes \mathbf j_\bullet\FF_2$ for $\bullet=!$ or $*$, and commutativity of \eqref{dgm:tensor-lower*and!} in these cases follows from $t$-exactness of $\boxtimes$, \cite[Theorem 4.2.8]{BBD}.
    
We turn to $\bullet=!*$. If $\FF_1, \FF_2$ are perverse sheaves, the canonical morphism \cite[2.7]{dCM}
\begin{equation*}
(\mathbf j\times\mathbf j)_!(\FF_1\boxtimes\FF_2)\simeq \mathbf j_!\FF_1\boxtimes\mathbf j_!\FF_2\longrightarrow \mathbf j_*\FF_1\boxtimes\mathbf j_*\FF_2\simeq (\mathbf j\times\mathbf j)_*\FF_1\boxtimes\FF_2
\end{equation*}
is obtained by applying $\boxtimes$ to the canonical morphisms 
$\mathbf j_!\FF_1\to\mathbf j_*\FF_1$ and $\mathbf j_!\FF_2\to\mathbf j_*\FF_2$, giving compatibility also for the image that is, commutativity of \eqref{dgm:tensor-lower*and!} for $!*$. Now, $\underline{R}$ is defined at the level of complexes and $\Vu$ via \eqref{eq:defR}, and compatibility with extensions follows by construction.
\end{proof} 

Since $\overline{\boxtimes}_\varphi$ and ${\boxtimes}_\varphi$ are compatible, when clear from the context, we will omit the $\overline{\phantom{a}}$ symbol. 

\medskip 

More generally, if $m_1,\dots, m_n$ are positive integers such that $m:=\sum_{i=1}^n m_i$, then for any $\varphi\in E_2(n)$ we have a functor
\[
\boxtimes_{\varphi}: \mathcal{P}(m_1)\times  \mathcal{P}(m_2)\times \dots \times \mathcal{P}(m_n)
\rightarrow
\mathcal{P}(m), 
\]
and similarly for $\mathcal{P}_{\leq d}$ and $\mathcal{P}^{\leq d}$. By construction, these functors are compatible with operadic composition.

\medskip

\begin{lem}\label{lem:many-tensors}Let $c,d,n\in\NN$ with $c\geq d$, $n\in\NN_{\geq1}$, and $\varphi\in E_2(n)$. For $l=1,
\,\ldots,\, n$, let $m_l\in \NN_{\geq 0}$ and $m=\sum_l m_l$. Let ${\mathbf j}\colon \sym_{\leq d}(U)\to Y$ be either $j_{\leq d}$, or  $j_{\leq d,\leq c}$, 
and let $\bullet$ stand for  $!$, $*$ or $!*$. Then the diagrams below are commutative.
\begin{equation}\label{dgm:tensor-tanti}
\begin{tikzcd}
 \mathcal{PS}(Y^{m_1}, \Sigma^{m_1},\Vu)\times \dots \times\mathcal{PS}(Y^{m_n}, \Sigma^{m_n},\Vu)
\arrow[rr, "\boxtimes_\varphi"] 
\arrow[d, "\prod_{l=1}^n({\mathbf j}^{m_l})^\ast"']& 
&  \mathcal{PS}(Y^{m}, \Sigma^{m},\Vu) \arrow[d,swap, "({\mathbf j}^m)^\ast"']
\\
 \PPP_{\leq d}({m_1})\times \dots \times\PPP_{\leq d}(m_n)\arrow[rr, "\boxtimes_\varphi"] 
 & &
 \PPP_{\leq d}({m}).
\end{tikzcd} 
\end{equation}
\medskip
 \begin{equation}\label{dgm:big2} \begin{tikzcd}[row sep=huge]
\PPP_{\leq d}({m_1})\times \dots \times\PPP_{\leq d}({m_n})  \arrow[rr, "\boxtimes_\varphi"] 
    \arrow[d, "\prod_{l=1}^n{^p}(\mathbf j^{m_l})_\bullet"']& 
    & \PPP_{\leq d}({m})\arrow[d,  "{^p}(\mathbf j^m)_\bullet"]
    \\
   \mathcal{PS}(Y^{m_1},\Sigma^{m_1},\Vu)\times \dots \times\mathcal{PS}(Y^{m_n},\Sigma^{m_n},\Vu)  \arrow[rr, "\boxtimes_\varphi"] 
    & &
    \mathcal{PS}(Y^{m},\Sigma^{m},\Vu) 
    \end{tikzcd} 
    \end{equation}
\end{lem}
\begin{proof}Commutativity of the diagrams follows from arguments similar to those used in the proof of Lemma \ref{lem:treni-tensor}.
\end{proof}

\medskip

For ${\mathcal F}\in{\mathcal P}(1)$ or ${\mathcal P}_{\leq d}(1)$ and $\varphi\in E_2(n)$ we denote by ${\mathcal F}^{\boxtimes_{\varphi}n}$ the $n$-fold tensor product of ${\mathcal F}$ with itself.  For $\FF_1,\,\ldots,\,\FF_n\in {\mathcal P}(1)$ or ${\mathcal P}_{\leq d}(1)$ we will also set $\boxtimes_\varphi(\FF_l)_{l=1}^n$ for $\boxtimes_\varphi(\FF_1,\,\ldots,\,\FF_n)$, and similarly for morphisms.

\begin{remark}\label{rem:erre-operadic}
{\rm Let $d,m\in\NN$. Compatibility of the braiding $R$ in $D^b(\sym(U)^m,\Sigma^m,\Vu)$  with operadic composition gives compatibility of $\underline{R}$ in $\PPP$, $\PPP_{\leq d}$ and $\PPP^{\leq d}$ with operadic composition. More precisely, for $n\in\NN$ and for every $l=1,\,\ldots,
n$ let $m_l\in \NN$ and $m:=\sum_{l=1}^nm_l$ and let $\gamma_l$ be a path in $E_2(m_l)$ from $\varphi_l$ to $\psi_l$. For any $\theta\in E_2(n)$, composing operadically with the constant path $\theta$, we obtain a path  $\widetilde{\gamma}=\theta\circ(\gamma_1,\,\ldots,\,\gamma_n)$ from $\eta=\theta\circ (\varphi_1,\,\ldots,\,\varphi_n)$ to $\zeta=\theta\circ (\psi_1,\,\ldots,\,\psi_n)$. 
Then, for $Y=\sym(U), \sym_{\leq d}(U)$ or $\sym^{\leq d}(U)$, and perverse sheaves $\FF_{lq}$ on $Y$, for $l=1,\,\ldots,\,n$ and $q=1,\,\ldots,\, m_l$, and denoting by $p_{lq}\colon Y^q\to Y$ the $l$-th projection, the braiding 
$\underline{R}_{\widetilde{\gamma},((\FF_{lq})_{q=1}^{m_{l}})_{l=1}^m}=R_{\widetilde\gamma, ((p_{lq}^*\FF_{lq})_{q=1}^{m_{l}})_{l=1}^m)}$,  is the composition 
\begin{equation*}
\begin{tikzcd}
\boxtimes_{\eta}((\FF_{lq})_{q=1}^{m_{l}})_{l=1}^m) \simeq \boxtimes_{\theta}(\boxtimes_{\varphi_l} (\FF_{lq})_{q=1}^{m_{l}})_{l=1}^n \arrow[rr,"\boxtimes_{\theta}(\underline{R}_{\gamma_l, (\FF_{l1},\,\ldots,\,\FF_{m_{l}})})_{l=1}^n"] &&\hskip0.2cm\boxtimes_{\theta}(\boxtimes_{\psi_l} (\FF_{lq})_{q=1}^{m_{l}})_{l=1}^n \simeq \boxtimes_{\zeta} ((\FF_{lq})_{q=1}^{m_{l}})_{l=1}^m)
\end{tikzcd}
\end{equation*} where $\underline{R}_{\gamma_l, (\FF_{l1},\,\ldots,\,\FF_{lm_{l}})}={R}_{\gamma_l, (p_{l1}^*\FF_{l1},\,\ldots,\,p_{lm_l}^*\FF_{lm_{l}})}$. 

\medskip

Similarly, if $\gamma'$ is a path in $E_2(n)$ from $\theta$ to $\theta'$, composing it operadically with the constant path $(\varphi_1,\,\ldots,\,\varphi_n)$ in $\prod_{l=1}^nE_2(m_l)$
gives a path $\widetilde{\gamma'}$ in $E_2(m)$ from $\eta$ to $\eta'=\theta'\circ (\varphi_1,\,\ldots,\,\varphi_n)$.  Then, $\underline{R}_{\widetilde{\gamma'},((\FF_{lq})_{q=1}^{m_{l}})_{l=1}^m}$
is the composition 
\begin{equation*}
\begin{tikzcd}
\boxtimes_{\eta}((\FF_{lq})_{q=1}^{m_{l}})_{l=1}^m) \simeq \boxtimes_{\theta}(\boxtimes_{\varphi_l} (\FF_{lq})_{q=1}^{m_{l}})_{l=1}^n \arrow[rr,"\underline{R}_{\gamma', (\FF_{11},\,\ldots,\,\FF_{nm_{n}})}\;"] &&\boxtimes_{\theta}(\boxtimes_{\varphi_l} (\FF_{lq})_{q=1}^{m_{l}})_{l=1}^n \simeq \boxtimes_{\eta'} ((\FF_{lq})_{q=1}^{m_{l}})_{l=1}^m)
\end{tikzcd}
\end{equation*}}
\end{remark}

\begin{remark}\label{rem:functors-monoidal}
{\rm  The categories $\mathcal{P}$, $\PPP^{\leq d}$  
and  $\mathcal{P}_{\leq d}$ for each $d$ 
 are full braided monoidal subcategories of the braided monoidal category $\mathcal{P}erv(\phantom{a}, \mathcal{V})$ of perverse sheaves with values in $\Vu$ as in Remark \ref{rem:D-monoidal}. The $\boxtimes$ tensor products endow the categories 
$\mathcal{P}_{\leq d}$, $\mathcal{P}$ and $\mathcal{P}^{\leq d}$ with structures of $\Pi_1 E_2$-algebras. Lemma \ref{lem:many-tensors}  guarantees that the restriction and extension functors are monoidal, i.e., morphisms of 
$\Pi_1 E_2$-algebras.}
\end{remark}

\subsubsection{The functors $a_\varphi^\ast$}\label{sec:a}
Let $d,n\in \NN$ \and let $\varphi\in E_2(n)$. Pulling back along the multiplication embeddings $a_{\varphi}$,
$a_{\varphi,\leq d}$ and $a_{\varphi}^{\leq d}$ defined in \eqref{a-phi}, \eqref{a-phi-leq} and \eqref{a-phi-leq-alto}
yields respectively functors 
\begin{equation}\label{eq:aphi*}
a_{\varphi}^\ast : \mathcal{P}(1)
\rightarrow
\mathcal{P}(n),\quad 
a_{\varphi,\leq d}^\ast : \mathcal{P}_{\leq d}(1)
\rightarrow
\mathcal{P}_{\leq d}(n), \textrm{ and }  (a_{\varphi}^{\leq d})^\ast : \mathcal{P}^{\leq d}(1)
\rightarrow
\mathcal{P}^{\leq d}(n).\end{equation}

By \eqref{eq:equivariance-a-phi}, for any $\sigma\in {\mathbb S}_n$ there holds: 
\begin{equation}\label{eq:a-phi*-equivariance}
a_{\varphi^\sigma}^*=\sigma^*\circ a^*_\varphi,\quad (a_{\varphi^\sigma,\leq d})^*=\sigma^*\circ a_{\varphi,\leq d}^\ast\mbox{ and } (a_{\varphi^{\sigma}}^{\leq d})^\ast=\sigma^*\circ (a_{\varphi}^{\leq d})^\ast.\end{equation}
 More generally, we define the functors

\[
(\Pi_{l=1}^n a_{\psi_l})^\ast: \mathcal{P}(n) 
\rightarrow
\mathcal{P}(m),\quad (\Pi_{l=1}^n a_{\psi_l,\leq d})^\ast: \mathcal{P}_{\leq d}(n) 
\rightarrow
\mathcal{P}_{\leq d}(m),\] and \[(\Pi_{l=1}^n a_{\psi_l}^{\leq d})^\ast: \mathcal{P}^{\leq d}(n) 
\rightarrow
\mathcal{P}^{\leq d}(m)
\]
for $\psi_i\in E_2(m_l)$ for $l=1,\,\ldots n,$ and $\sum_{l=1}^n m_l = m$. 

\medskip

\begin{lem}\label{lem:b-passa}
Let $\varphi\in E_2(n)$ and, for $l=1,\,\ldots,\,n$, let $\psi_l\in E_2(m_l)$, with $m_l\in\NN$. Then, setting 
$\eta := \varphi \circ (\psi_1,\,\ldots,\,\psi_n)$, one has 
\begin{equation}\label{eq:composition-a}
a_\eta^\ast = 
(\prod_{l=1}^n a_{\psi_l})^\ast a_{\varphi}^\ast,
\quad
a_{\eta,\leq d}^\ast = 
(\prod_{l=1}^n a_{\psi_l,\leq d})^\ast a_{\varphi,\leq d}^\ast \ 
\mathrm{ and }\  
(a_\eta^{\leq d})^\ast = 
(\prod_{l=1}^n a_{\psi_l})^{\leq d,\ast }(a_{\varphi}^{\leq d})^\ast. \end{equation}
Furthermore, for $\FF\in \PPP(1)$, $\GG\in\PPP_{\leq d}(1)$ and ${\mathcal E}\in\PPP^{\leq d}(1)$ there holds
\begin{equation*}
\boxtimes_\varphi(a^*_{\psi_l}\FF)_{l=1}^n=(\prod_{l=1}^na_{\psi_l})^*\FF^{\boxtimes_\varphi n},\quad
\boxtimes_\varphi(a^*_{\psi_l,\leq d}\GG)_{l=1}^n=(\prod_{l=1}^na_{\psi_l,\leq d})^*\GG^{\boxtimes_\varphi n},\end{equation*}
and
\begin{equation*}\boxtimes_\varphi((a_{\psi_l}^{\leq d})^*{\mathcal E})_{l=1}^n=(\prod_{l=1}^na_{\psi_l}^{\leq d})^*\mathcal E^{\boxtimes_\varphi n}.
\end{equation*}
 \end{lem}
\begin{proof}
The first set of equalities follows from \eqref{eq:symemb} and functoriality of the pull-back, the second one follows from the definition of external tensor product. Indeed, for every $q=1,\,\ldots,n$ the $q$-th projections $p_{q}\colon \sym(U)^n\to\sym(U)$ and  $p'_{q}\colon \prod_{l=1}^n\sym(U)^{m_l}\to\sym(U)^{m_q}$
satisfy $p_q\circ \prod_{l=1}^n a_{\psi_l}=a_{\psi_q}\circ p'_q$. Hence, $(\prod_{l=1}^n a_{\psi_l})^*\circ p_q^*=(p'_q)^*\circ a_{\psi_q}^*$ so 
\begin{equation*}\boxtimes_{\varphi} (a^*_{\psi_l}\FF)_{l=1}^n=\otimes_{\varphi} (\prod_{l=1}^n a_{\psi_l})^*(p_q^*\FF)_{q=1}^n.\end{equation*} Now, $\otimes_{\varphi} (\prod_{l=1}^n a_{\psi_l})^*=(\prod_{l=1}^n a_{\psi_l})^*\otimes_{\varphi}$ by transitivity \cite[\S 5.8]{dCM}, giving the first of the three equalities. The result for the restrictions or truncations of the $a_{\psi_l}$'s follows in a similar fashion.
 \end{proof}

\begin{lem}\label{lem:a-treno}
Let $c,d,e\in\NN$ with $c\geq d\geq e\geq 0$, let $n\in\NN_{\geq1}$ and $\varphi\in E_2(n)$. Let $\bullet$ stand for $*$, $!$ or $!*$.
Then the following diagrams of functors are commutative
\begin{equation}\label{dgm:upperstar}
\begin{tikzcd}[row sep=huge]
 \mathcal{P}(1)
\arrow[d, swap, "a_\varphi^\ast"] 
\arrow[rr, swap, "(j_{\leq c})^\ast"']& 
&  \mathcal{P}_{\leq c}(1) \arrow[d, "(a_{\varphi,\leq c})^\ast"] 
\arrow[rr, swap, "(j_{\leq d,\leq c})^\ast"'] &&  \mathcal{P}_{\leq d}(1) \arrow[d, "(a_{\varphi,\leq d})^\ast"] 
\arrow[rr, swap, "(\phantom{a})^{\leq d}"'] &&
\mathcal{P}^{\leq d}(1) \arrow[d, "(a_\varphi ^{\leq d})^\ast"]\arrow[rr, swap, "(\phantom{a})^{\leq e,\leq d}"']&& \mathcal{P}^{\leq e}(1) \arrow[d, "(a_{\varphi}^{\leq e})^\ast"] \\
\mathcal{P}(n) \arrow[rr, "(j_{\leq c}^n)^\ast"']
&&\mathcal{P}_{\leq c}(n) \arrow[rr, "(j_{\leq d,\leq c}^n)^\ast"']&&
\mathcal{P}_{\leq d}(n) \arrow[rr, "(\phantom{a})^{\leq d}"']&&\mathcal{P}^{\leq d}(n)\arrow[rr,  "(\phantom{a})^{\leq e,\leq d}"']&&\mathcal{P}^{\leq e}(n)
\end{tikzcd} 
\end{equation}

\begin{equation}\label{dgm:lower*and!}
\begin{tikzcd}[row sep=huge]
 \mathcal{P}(1)
\arrow[d, "a_\varphi^\ast"] 
& 
&  \mathcal{P}_{\leq c}(1) \arrow[ll,  "^p(j_{\leq c})_\bullet"']\arrow[d, "(a_{\varphi\leq c})^\ast"] 
 &&  \mathcal{P}_{\leq d}(1) \arrow[d, "(a_{\varphi\leq d})^\ast"] 
\arrow[ll, "^p(j_{\leq d,\leq c})_\bullet"']\\
\mathcal{P}(n)
&&\mathcal{P}_{\leq c}(n)  \arrow[ll, swap,"^p(j_{\leq c}^n)_\bullet"'] &&
\mathcal{P}_{\leq d}(n) \arrow[ll, swap,"^p(j_{\leq d,\leq c}^n)_\bullet"']
\end{tikzcd} 
\end{equation}
\end{lem}
\begin{proof} By construction the functors $(\ )^{\leq d}$ and $(\ )^{\leq e,\leq d}$ are compatible with the pull-back along the restriction $a_{\varphi}^{\leq d}$ of $a_\varphi$ to $\sym(U)^n({\leq d})$ and similarly for the functors  $j_{\leq d}^\ast$  and $j_{\leq d,\leq c}^*$, giving commutativity of the first diagram. 

\medskip

Let ${\mathcal G}$ be in $\mathcal{P}_{\leq d}(1),\mathcal{P}_{\leq c}(1),\mathcal{P}^{\leq d}(1)$ or $\mathcal{P}^{\leq d}(1)$. Since $a_{\varphi}$ and its restrictions are smooth, by smooth base change \cite[p. 625]{dCM} we have $a^*_{\varphi}(j_{\leq d})_*(\mathcal{G})=(j^n_{\leq d})_{*}(a_{\varphi,\leq d})^*(\mathcal{G})$ and similarly for morphisms and for 
$(j_{\leq d,\leq c})_\ast$. 
In addition, $a^*_{\varphi}$ and its restrictions are $t$-exact because $a_\varphi$ is an open embedding, \cite[p. 612]{dCM}. Hence 
   \[ a^*_{\varphi} {^p(j_{\leq d})_\ast}{\mathcal G}=
    a^*_{\varphi} {^{p}\tau_{\leq 0}}(j_{\leq d})_*(\mathcal{G})= {^{p}\tau_{\leq 0}} a^*_{\varphi} (j_{\leq d})_*(\mathcal{G})= {^{p}\tau_{\leq 0}} (j^n_{\leq d})_{*}(a_{\varphi,\leq d})^*(\mathcal{G})={^p(j_{\leq d}^n)_\ast} (a_{\varphi,\leq d})^*{\mathcal G}\] and similarly for $(j_{\leq d,\leq c})_\ast$,
        obtaining commutativity of \eqref{dgm:lower*and!} with $\bullet=*$. 

\medskip

By base change \cite[p. 625]{dCM}  we also obtain $a^*_\varphi (j_{\leq d})_! 
\mathcal= (j^n_{\leq d})_! 
a^*_{\varphi\leq d}$ and similarly for $j_{\leq d,\leq c}$. As $a^*_{\varphi}$ and its restrictions are $t$-exact, the diagram \eqref{dgm:lower*and!} is commutative also for $\bullet=!$. 

\medskip

Combining with commutativity of \eqref{dgm:lower*and!} for $\bullet=*$  and exactness of $a^*_{\varphi}$ and its restrictions, we have commutativity also for $\bullet=!*$ (removing the unnecessary ${}^p$). \end{proof}

\medskip

More generally, functors $a_\varphi^*$ for $\varphi\in E_2(n)$ and their restrictions behave well with respect to products of spaces.

\begin{lem}\label{lem:a-tanti}
For $c,d\in\NN$ with $d\leq c$ let ${\mathbf j}\colon Y\to Y'$ be any of the open embeddings $j_{\leq c}$, $j_{\leq c,\leq d}$, $i^{\leq c,\leq d}$, or $i^{\leq d}$ from \eqref{eq:embedding-j} and \eqref{eq:inclusions-i}. Let $n,m_l\in\NN_{\geq1}$ for $l=1,\,\ldots, n$ and $m=\sum_{l=1}^nm_l$, let $\varphi\in E_2(n)$, $\psi_l\in E_2(m_l)$, and let $b'_{\varphi}\colon (Y')^n\to Y'$, and $b_{\varphi}\colon Y^n\to Y$,
respectively, $b'_{\psi_l}\colon (Y')^{m_l}\to Y'$ and  $b_{\psi_l}\colon (Y)^{m_l}\to Y$ be the  restrictions or truncations of  $a_{\varphi}$, respectively $a_{\psi_l}$,  to $(Y')^n$ and $Y^n$, respectively $(Y')^{m_l}$ and $Y^{m_l}$ as in \eqref{a-phi}, \eqref{a-phi-leq} or \eqref{a-phi-leq-alto}.
 Then the following diagram commutes.      
\begin{equation}\label{dgm:product-tanti}
\begin{tikzcd}[row sep=huge]
 \mathcal{PS}((Y')^n,\Sigma^n,\Vu)
\arrow[rr, "(\prod_{l=1}^n b'_{\psi_l})^*"] 
\arrow[d, "({\mathbf j}^n)^\ast"']& 
& \mathcal{PS} ((Y')^m,\Sigma^m,\Vu)\arrow[d, swap, "({\mathbf j}^{m})^\ast"']
\\
\mathcal{PS} (Y^n,\Sigma^n,\Vu)\arrow[rr, "(\prod_{l=1}^n b_{\psi_l})^\ast"] 
 & &
\mathcal{PS} (Y^m,\Sigma^m,\Vu)
\end{tikzcd} \end{equation}
Let $\bullet$ be $!$, $*$ or $!*$. Then, for ${\mathbf j}\colon Y\to Y'$ any of the embeddings $j_{\leq c}$ or $j_{\leq c,\leq d}$ from \eqref{eq:embedding-j} the following diagram commutes
 \begin{equation}\label{dgm:big1}
    \begin{tikzcd}[row sep=huge]
     \mathcal{PS}(Y^n,\Sigma^n,\Vu) 
    \arrow[rr, "(\prod_{l=1}^n b_{\psi_l})^\ast"] 
    \arrow[d, "{}{^p}(\mathbf j^n)_\bullet"']& 
    & \mathcal{PS}(Y^m,\Sigma^m,\Vu)\arrow[d,swap, "{^p}(\mathbf j^{m})_\bullet"']
    \\
   \mathcal{PS}((Y')^n,\Sigma^n,\Vu)\arrow[rr, "(\prod_{l=1}^n b'_{\psi_l})^*"] 
    & &
    \mathcal{PS}((Y')^m,\Sigma^m,\Vu) 
    \end{tikzcd} 
    \end{equation}
\end{lem}
\begin{proof}The arguments used to prove Lemma \ref{lem:a-treno} apply also to this case.
\end{proof}

\subsubsection{Monodromy} The arguments in Section \ref{mndrmy}, from which we retain notation, ensure existence of an isomorphism of functors given by monodromy along paths, that we now explain. 
\begin{lem}\label{lem:mndrmy}
    Let $n\geq 1$, let $\varphi$ and $\psi$ in $E_2(n)$, let $\gamma$ be a continuous path in $E_2(n)$ joining $\varphi$ to $\psi$, and let $a_\varphi^*$ and $a_{\psi}^*$ be the functors $\mathcal{P}(1)\to{\mathcal P}(n)$. Then, there is a natural monodromy isomorphism of functors
    \begin{equation}\label{eq:emmegamma}
  M_\gamma\colon a^\ast_\varphi \rightarrow  a^\ast_{\psi},
    \end{equation}
 which  depends only on the homotopy class of the path, inducing   isomorphisms of functors
  \begin{equation}\label{eq:emmegamma-d}
  M_\gamma\colon a_{\varphi,\leq d}^\ast \rightarrow  a_{\psi,\leq d}^\ast,\quad  M_\gamma\colon (a_\varphi^{\leq d})^\ast\rightarrow  (a_{\psi}^{\leq d})^\ast
    \end{equation}
    by restriction and truncation. 
The monodromy isomorphisms $M_\gamma$ for each path $\gamma$ in $E_2(n)$ are compatible with composition of paths in $E_2(n)$  and satisfy $M_{\gamma^\sigma}=\sigma^*(M_{\gamma})$ for any $\sigma\in\Sn$.  
\end{lem}
\begin{proof}
We see the path $\gamma$ as a continuous map $U^{\coprod n}\times [0,1]\to U$ such that for every $t\in [0,1]$  the map  $\gamma_t:=\gamma(-,t)\colon U^{\coprod n}\rightarrow U$ is in $E_2(n)$, with $\gamma(-,0)=\varphi$ and $\gamma(-,1)=\psi$. We can extend $\gamma$ to a continuous function $U^{\coprod n}\times (a,b)\to U$ for some $a<0$ and $b>1$, by taking $\gamma(-,t)=\varphi$ for $t<0$ and $\gamma(-,t)=\psi$ for $t>1$.  Then $\gamma$ induces a stratified map $a_{\gamma}\colon \sym(U)^n  \times (a,b) \rightarrow \sym(U)$ defined by $a_\gamma(-,t)=a_{\gamma(-,t)}$ for any $t\in (a,b)$. Proposition \ref{lem:equiv2} 
with $u=0$, $v=1$, $\alpha=a_{\gamma}$ gives a natural isomorphism  between the functors 
\begin{align*}a_{\gamma(-,0)}^*\colon D_c^b(\sym(U),\Sigma,\Vu)&\to D_c^b(\sym(U)^n,\Sigma^n,\Vu) \ {\rm and}\\
a_{\gamma(-,1)}^*\colon D_c^b(\sym(U),\Sigma,\Vu)&\to D_c^b(\sym(U)^n,\Sigma^n,\Vu).\end{align*} Since $a_\varphi^*$ and $a_\psi^*$ map the full subcategory $\PPP(1)$ of $D_c^b(\sym(U),\Sigma,\Vu)$ to the full subcategory $\PPP(n)$ of $D_c^b(\sym(U)^n,\Sigma^n,\Vu)$, the sought monodromy isomorphism  
$M_\gamma\colon a_\varphi^*=a_{\gamma(-,0)}^*\to a_{\gamma(-,1)}^*=a_{\psi}^*$ is induced from the monodromy isomorphism at the level of derived categories. Restriction to $\sym_{\leq d}(U)^n$ and truncation to $\sym(U)^n(\leq d)$ give the other monodromy isomorphisms.
  
\medskip

In order to prove that $M_\gamma$ depends only on the homotopy class of $\gamma$, we consider another path $\gamma'$ in $E_2(n)$ joining $\varphi$ to $\psi$ and homotopic to $\gamma$ via the continuous function $\theta: U^{\coprod n}\times (a,b)^2\rightarrow U$, such that $\theta(-,t,s)\in E_2(n)$ for any $s,t\in(a,b)$, and satisfying 
\begin{align*}&\theta(-,t,0)=\varphi(-), &&\theta(-,t,1)=\psi(-), &&\forall t\in(a,b), \\
&\theta(-,0,-)=\gamma(-,-), &&\theta(-,1,-)=\gamma'(-,-).\end{align*}
The map $\theta$ gives rise to a stratified map $a_\theta\colon \sym(U)^n\times(a,b)^2\to \sym(U)$ and Proposition \ref{lem:equiv2}, from which we adopt notation,   guarantees existence of the natural isomorphisms of functors $M_{\theta,u,v}\colon i_u^* a_\theta^*\to i_v^*a_\theta^*$ for every $u,v\in (a,b)^2$ defined as $i_v^*\alpha_{\theta,u}$ where for any $\FF\in\PPP(1)$, the isomorphism \[\alpha_{\theta,u,\FF}\colon\pi^*i_u^*a_\theta^*\FF\to a_\theta^*\FF\] is uniquely determined by the requirement that 
$i_u^*\alpha_{\theta,u,\FF}=\id_{i_u^*a_\theta^*\FF}$. Now, for any $s,t\in (a,b)$, taking $u=(s,0)$ and $v=(t,0)$ gives  $i_{(s,0)}^*a_\theta^*=a_\varphi^*=i_{(t,0)}^*a_\theta^*$ and by construction $M_{\theta,(s,0),(t,0)}=\id_{a_\varphi^*}$. Similarly, taking $u=(s,1)$ and $v=(t,1)$ gives $i_{(s,1)}^*a_\theta^*=a_\psi^*=i_{(t,1)}^*a_\theta^*$
and $M_{\theta,(s,1),(t,1)}=\id_{a_\psi^*}$.
Let $\iota_0\colon \sym(U)^n\times(a,b)\to \sym(U)^n\times (a,b)^2$ be given by $\iota_0(x,t)=(x,0,t)$.  Applying the fully faithful functor $\iota_0^*$ to $\alpha_{\theta,(0,0)}$ it is easily verified that $M_\gamma=M_{\theta,(0,0),(0,1)}$. Similarly, one checks that $M_{\gamma'}=M_{\theta,(1,0),(1,1)}$.

\medskip 

We claim that $\forall u,v,w\in (a,b)^2$ the following equality holds:
\begin{equation}\label{eq:comp-monodromy}
M_{\theta,u,w}=M_{\theta,v,w}\circ M_{\theta,u,v}.
\end{equation} Indeed, for any $\FF\in\PPP(1)$
the diagram below commutes
\begin{equation}\label{eq:triangle-monodromy}
\begin{tikzcd}
\pi^*i_u^*a_\theta^*\FF\arrow[dd, swap,"\pi^*i_v^*\alpha_{\theta,u,\FF}"] \arrow[rrdd, "\alpha_{\theta,u,\FF}"] \\
&&\\
\pi^*i_v^*a_\theta^*\FF\arrow[rr,"\alpha_{\theta,v,\FF}"]&&
a_\theta^*\FF
\end{tikzcd} 
\end{equation} because it commutes after application of the fully faithful functor $i_v^*$. Then, \eqref{eq:comp-monodromy} follows by applying the fully faithful functor $i_w^*$. Applying  \eqref{eq:comp-monodromy}  to the triples  $u=(0,0)$, $v=(0,1)$, $w=(1,1)$ and 
 $u=(0,0)$, $v=(1,0)$, $w=(1,1)$ gives the independence from homotopy. 

\medskip

Compatibility with composition of paths follows from \eqref{eq:comp-monodromy}.  We now show equivariance. Let $\sigma\in\Sn$ and let $\gamma^\sigma$ be the path from $\varphi^\sigma$ to $\psi^\sigma$. Then $a_{\varphi^\sigma}=a_\varphi\circ\sigma$ and $a_{\psi^\sigma}=a_\psi\circ\sigma$ and also $a_{\gamma^\sigma}=a_\gamma\circ (\sigma,\id)$, and finally $\sigma\pi=\pi\circ (\sigma,\id)$. Hence, for any $\FF\in D^b_c(\sym(U)^n,\Sigma,\Vu)$ the map corresponding to $\id\colon a_{\varphi^\sigma}^*\FF\to a_{\varphi^\sigma}^*\FF$ through the chain of identifications as in \eqref{eq:natural} with $\alpha=a_{\gamma^\sigma}$ is $\sigma^*( M_{\gamma,\FF})$.   
\end{proof}

\begin{remark}\label{rem:monodromy-product}
{\rm In a similar fashion, for $n, m_l\in \NN_{\geq 1}$, with $l=1,\,\ldots,n$, one defines a monodromy isomorphism for any path $\underline\gamma=(\gamma_1,\,\ldots,\,\gamma_n)$ in $\prod_{l=1}^nE_2(m_j)$, given by $\underline{\gamma}(t)=(\gamma_1(t),\,\ldots,\,\gamma_n(t))$ for $t\in [0,1]$, where each $\gamma_l\colon [0,1]\to E_2(m_l)$ is a path in $E_2(m_l)$. The monodromy isomorphism is then a natural transformation  $M_{\underline\gamma}$ between the functors 
$(\prod_{l=1}^n a_{\gamma_l(0)})^*$ and $(\prod_{l=1}^n a_{\gamma_l(1)})^*$
or their restrictions, and enjoys similar properties as the one induced by paths in $E_2(n)$.}
\end{remark}

From now on, when clear from the context, for any $\varphi\in E_2(n)$ and any $d,n\in\NN$, the symbol $b_\varphi$ will indicate any of the open embeddings $a_\varphi$,  $a_{\varphi,\leq d}$, or $a_\varphi^{\leq d}$, between the appropriate spaces and $b_\varphi^*$ will indicate the corresponding pull-back functor between the appropriate categories of perverse sheaves.\label{page:b}

\medskip

\begin{lem}\label{lem:compat-a-M}
Let $\varphi,\varphi'\in E_2(n)$ and, for $l=1,\,\ldots,\,n$, let $\psi_l\in E_2(m_l)$ and $m:=\sum_{l=1}^nm_l$. Let  $\gamma\colon U^{\coprod{} n}\times[0,1]\to U$  be a path in $E_2(n)$ from $\varphi$ to $\varphi'$ and $M_\gamma\colon a_\varphi^*\to a_{\varphi'}^*$ be the corresponding monodromy isomorphism. Let also  $\eta := \varphi \circ (\psi_1,\,\ldots,\,\psi_n)$ and $\eta' := \varphi' \circ (\psi_1,\,\ldots,\,\psi_n)$, and  $\widetilde{\gamma}\colon U^{\coprod{}m}\times [0,1]\to U$ be the path in $E_2(m)$ from $\eta$ to $\eta'$ defined by $\widetilde{\gamma}(-,t):=\gamma(-,t)\circ(\psi_1,\,\ldots,\,\psi_n)$ for $t\in[0,1]$, with corresponding monodromy isomorphism $M_{\widetilde\gamma}\colon a_\eta^*\to a_{\eta'}^*$. Then,  for any $\FF\in\PPP(1)$ the diagram below commutes:
\begin{equation}\label{dgm:compatibility-aM}
\begin{tikzcd}[row sep=huge]
(\prod_{l=1}^na_{\psi_l})^* (a_\varphi^*\FF )
\arrow[rr, "\prod_{l=1}^na_{\psi_l}^* (M_{\gamma,\FF})"] 
\arrow[d, "\equiv"']& 
& (\prod_{l=1}^na_{\psi_l})^*( a_{\varphi'}^*\FF)  \arrow[d,swap, "\equiv"']
\\
 a_\eta^*\FF \arrow[rr, "M_{\widetilde\gamma,\FF}"] 
 & &
a_{\eta'}^*\FF
\end{tikzcd} 
\end{equation}
In particular, for  $n=m_1=m=1$ and for any $\varphi,\psi\in E_2(1)$ there is a unique monodromy isomorphism 
\begin{equation}M_{\varphi,\psi}\colon a_{\varphi}^*\to a_{\psi}^*\end{equation}
and there holds 
\begin{equation}\label{eq:monodromy-n=1}M_{\varphi\circ\psi,\id_U}=M_{\psi,\id_U}\circ M_{\varphi\circ\psi,\psi},\quad M_{\varphi\circ\psi,\psi}=a^*_\psi(M_{\varphi,\id_U}).\end{equation}

Analogous results hold for $\FF\in \PPP_{\leq d}(1)$ and $\FF\in \PPP^{\leq d}(1)$ for any $d\geq 1$.
 \end{lem}
\begin{proof} First of all, the vertical arrows are identifications stemming from \eqref{eq:composition-a}. By construction, $\widetilde{\gamma}$ is a path from $\eta$ to $\eta'$, hence $M_{\widetilde{\gamma}}$ and $\prod_{l=1}^na_{\psi_l}^* (M_{\gamma,\FF})$ are both isomorphisms of functors $a_\eta^*\to a_{\eta'}^*$. For $t\in
\{0,1\}$ and $q\in\{n,m\}$, let  $i_{t,q}\colon \sym(U)^q\to \sym(U)^q\times[0,1]$ be defined by $i_{t,q}(x)=(x,t)$. Recall from Lemmata \ref{lem:equiv2} and \ref{lem:mndrmy}, from which we retain notation, that $M_{\widetilde\gamma,\FF}=i_{1,m}^*\widetilde\alpha_\FF$, where $\widetilde\alpha\colon \pi^*a_\eta^*\to a_{\widetilde\gamma}^*$ satisfies $i_{0,m}^*\alpha_\FF=\id_{a_\eta^*\FF}$ and that $M_{\gamma,\FF}=i_{1,n}^*\alpha_\FF$, where $\alpha\colon \pi^*a_\varphi^*\FF\to a_{\gamma}^*\FF$ satisfies $i_{0,1}^*\alpha_\FF=\id_{a_\varphi^*\FF}$. Then commutativity of the diagram follows because 
$(\prod_{l=1}^m a_{\psi_l})\circ i_{t,m}=i_{t,n}\circ(\prod_{l=1}^m a_{\psi_l})$. 

\medskip

As for the case $n=1$, since every $1$-ary embedding is homotopic to the identity embedding, we obtain the uniqueness of the monodromy for $\varphi,\psi\in E_2(1)$, and  
\eqref{eq:monodromy-n=1} follows from uniqueness and \eqref{dgm:compatibility-aM}. The statements for $\PPP_{\leq d}(1)$ and $\PPP^{\leq d}(1)$ follow from compatibility of the involved functors with restriction to $\sym_{\leq d}(U)^q$ and truncation to degree $d$.\end{proof}

We use the convention for $b_{\varphi}$ and $b^*_{\varphi}$  from page \pageref{page:b}.

\begin{lem}\label{lem:compat-a-M-factorized}
Let $n\in\NN_{\geq1}$  and  $\theta\in E_2(n)$. For $l=1,\,\ldots,\,n$, let $m_l\in\NN_{\geq 1}$, let $\varphi_l,\psi_l\in E_2(m_l)$ and let $\gamma_l$ be a continuous path in $E_2(m_l)$ from $\varphi_l$ to $\psi_l$. We set $m:=\sum_{l=1}^nm_l$. 
Let $\tilde{\gamma}$ be the continuous path in $E_2(m)$ from $\eta:=\theta\circ (\varphi_1,\,\ldots,\,\varphi_n)$ to $\zeta:=\theta\circ (\psi_1,\,\ldots,\,\psi_n)$ obtained by operadically composing the constant path $\theta$ with the path $\underline\gamma=(\gamma_1,\,\ldots,\,\gamma_n)$ in $\prod_{l=1}^nE_2(m_l)$. 

Let $d\in\NN$ and let $\FF$ be an object in $\PPP_{\leq d}(1)$,  $\PPP^{\leq d}(1)$, or $\PPP(1)$. Assume that there is an isomorphism $\mu_\theta\colon \FF^{\boxtimes_\theta n}\to b_\theta^*(\FF)$. Then, 
\begin{equation}\label{eq:mon-b}M_{\widetilde\gamma,\FF}=M_{\underline\gamma,b^*_{\theta}\FF}\end{equation} where the right hand side is as in Remark \ref{rem:monodromy-product} and the following diagram commutes
\begin{equation}\label{dgm:compatibility-aM-factorized*}
\begin{tikzcd}[row sep=huge]
\boxtimes_\theta(b^*_{\varphi_1}\FF,\ldots,\,b_{\varphi_n}^*\FF)
\arrow{rr}{(\prod_{j=1}^nb_{\varphi_j})^* (\mu_\theta)} \arrow{d}[swap]{\boxtimes_\theta(M_{\gamma_1\FF},\,\ldots,\,M_{\gamma_n\FF})}&&b_\eta^*\FF \arrow{d}{M_{\widetilde\gamma,\FF}} \\
\boxtimes_\theta(b^*_{\psi_1}\FF,\ldots,\,b_{\psi_n}^*\FF)\arrow{rr}{(\prod_{j=1}^nb_{\psi_j})^* (\mu_\theta)}&&
b_{\zeta}^*\FF 
\end{tikzcd} 
\end{equation}
\end{lem}
\begin{proof}
Let $Y$ be either $\sym(U)$, $\sym_{\leq d}(U)$ or $\sym^{\leq d}(U)$.
First of all we show that, after the appropriate identifications stemming from \eqref{eq:composition-a}, we have \eqref{eq:mon-b}. We consider the  stratified maps $b_{\underline\gamma}\colon (\prod_{l=1}^n Y^{m_l})\times [0,1]\to Y^n$ and 
$b_{\widetilde\gamma}\colon Y^m\times [0,1]\to Y$ obtained by applying $\sym\times\id$ to $\underline\gamma\colon \coprod_{l=1}^nU^{\coprod m_l}\times [0,1]\to U^{\coprod n}$ and $\widetilde{\gamma}\colon U^{\coprod m}\times[0,1]\to U$. Then 
$b_{\widetilde\gamma}$ is the composition of $b_{\underline\gamma}$ with the stratified open embedding $b_\theta$ and the desired equality follows from Remark \ref{rem:open-embedding} with $\alpha=b_{\underline\gamma}$ and $f=b_\theta$. 

Next, we observe that the diagram below commutes by naturality of $M_{\underline\gamma}$ applied to the isomorphism $\mu_\theta$
\begin{equation}\label{dgm:compatibility-aM-partial}
\begin{tikzcd}[row sep=huge]
(\prod_{l=1}^nb_{\varphi_l})^*\FF^{\boxtimes_\theta n}
\arrow{rr}{(\prod_{l=1}^nb_{\varphi_l})^* (\mu_\theta)} \arrow{d}[swap]{M_{\underline\gamma,\FF^{\boxtimes_\theta n}}}&&(\prod_{l=1}^nb_{\varphi_l})^*(b_\theta^*\FF) \arrow{d}{M_{\underline\gamma,b^*_\theta\FF}} \\
(\prod_{l=1}^nb_{\psi_l})^*\FF^{\boxtimes_\theta n}\arrow{rr}{(\prod_{l=1}^nb_{\psi_l})^* (\mu_\theta)}&&
(\prod_{l=1}^nb_{\psi_l})^*(b_\theta^*\FF)
\end{tikzcd} 
\end{equation}

Finally, the identifications $(\prod_{l=1}^nb_{\varphi_l})^*\FF^{\boxtimes_\theta n}= \boxtimes_\theta (b_{\varphi_l}\FF)^*$ and $(\prod_{l=1}^nb_{\psi_l})^*\FF^{\boxtimes_\theta n}= \boxtimes_\theta (b_{\psi_l}\FF)^*$ from Lemma \ref{lem:b-passa} yield an identification of the components of the corresponding monodromies, that is, $\boxtimes_\theta(M_{\gamma_l,\FF})_{l=1}^n=M_{\underline\gamma,\FF^{\boxtimes_\theta n}}$. \end{proof}

We now show that monodromy is compatible with extension functors. 
\begin{lem}
Let $c,d\in \NN$ with $d\leq c$, let ${\mathbf j}\colon Y\to Y'$ be either $j_{\leq d}$ or  $j_{\leq d,\leq c}$ as in \eqref{eq:embedding-j}.
For $\varphi\in E_2(n)$, let $b_\varphi\colon Y^n\to Y$  and $b'_\varphi\colon (Y')^n\to Y'$ be the multiplication embeddings as in \eqref{a-phi} or \eqref{a-phi-leq}.  Let $n\in \NN_{\geq1}$ and let $\gamma$ be a path from $\varphi$ to $\psi$ in $E_2(n)$. Let $M_\gamma$ and $M_\gamma'$ stand for the monodromy isomorphism over $Y^n$ and $(Y')^n$, respectively. Then, for any $\FF\in \PS(Y,\Sigma, \Vu)$, and $\bullet$ any of the extensions $!, *$ or $!*$ there holds
\begin{equation}\label{eq:monodromy-estension}
{}^p(\mathbf j^n)_\bullet M_{\gamma,\FF}=M'_{\gamma,{}^p\mathbf j_\bullet\FF}.
\end{equation}
\end{lem}
\begin{proof}Let $(a,b)$ be a real interval with $a<0$ and $b>1$. For $u\in [0,1]$, let $i_u\colon Y^n\to Y^n\times (a,b)$ and $i'_u\colon (Y')^n\to (Y')^n\times (a,b)$ be as in Subsection \ref{mndrmy}, from which we retain notation. We have a commutative diagram
\begin{equation}\label{dgm:inclusion}
\begin{tikzcd}
Y^n\arrow{rr}{i_u} \arrow{d}[swap]{\bf j^n}&&Y^n\times (a,b)\arrow{d}{\bf j^n\times\id} \\
(Y')^n\arrow{rr}{i_u'}&&(Y')^n\times (a,b)
\end{tikzcd} 
\end{equation}
Recall that $M'_{\gamma,{}^p\mathbf j_\bullet\FF}=(i_1')^*\alpha'_{{}^p\mathbf j_\bullet\FF}$, where $\alpha'$ is the natural transformation determined by the equality $(i_0')^*\alpha'_{{}^p\mathbf j_\bullet\FF}=\id_{b^{'*}_{\varphi} ({}^p\mathbf j_\bullet\FF)}$, and that $M_{\gamma,\FF}=i_1^*\alpha_{\FF}$, where $\alpha$ is determined  by the equality  $i_0^*\alpha_{\FF}=\id_{b^{*}_{\varphi} (\FF)}$. Now, for any $u\in[0,1]$ there holds
$(\mathbf j^n)_! i_u^*=(i'_u)^* (\mathbf j^n\times \id)_!$ by base change, \cite[p. 625]{dCM}. 

Let $\pi\colon Y^n\times (a,b)\to Y^n$ and $\pi'\colon (Y')^n\times (a,b)\to (Y')^n$ be the natural projections. Then, for any complex $\mathcal K$ in $D^b_c({Y^n\times(a,b)},\Sigma^n\times S_0,\Vu)$ and any $l\in (a,b)$ there holds $\mathcal K\simeq\pi^*\pi_*\mathcal K\simeq \pi^*(K|_{Y^n})=\pi^*i_l^*\mathcal K$ making use of \cite[Lemma 3.5.4]{DCM}. Therefore
\begin{align*}\label{eq:ji=ij}
 (\mathbf j^n)_*i_l^*\mathcal K&=(\mathbf j^n)_*\pi_*\mathcal K=(\mathbf j^n\pi)_*\mathcal K=(\pi'(\mathbf j^n\times\id))_*\mathcal K\\
 &=\pi'_*(\mathbf j^n\times \id)_*\mathcal K=(i_l')^*(\mathbf j^n\times\id)_*\mathcal K.
\end{align*}
Now, if $\mathcal F$ is a perverse sheaf on $Y^n$ and $a_\gamma$ is as in the proof of Lemma \ref{lem:mndrmy}, then $i_l^*a_\gamma^*\mathcal F=a_{\gamma\circ i_l}^*\mathcal F$ is also perverse, and so for $\bullet=*$ or $!$, with the appropriate truncation, we have 
\begin{align*}
{}^p(\mathbf j^n)_\bullet i_l^*a_\gamma^*\mathcal F&={}^p\tau_{\leq0}\circ{}^p\tau_{\geq 0}((\mathbf j^n)_\bullet i_l^*a_\gamma^*\mathcal F)={}^p\tau_{\leq0}\circ{}^p\tau_{\geq 0}(((i'_l)^*(\mathbf j^n\times \id)_\bullet (a_\gamma^*\mathcal F)[1])[-1])\\
&=\left((i'_l)^*\circ {}^p(\mathbf j^n\times \id)_\bullet (a_\gamma^*\mathcal F)[1]\right)[-1].
\end{align*}
Let $\Omega\colon {}^p((\mathbf j^n\times \id)_! (a_\gamma^*\mathcal F)[1])\to {}^p((\mathbf j^n\times \id)_* (a_\gamma^*\mathcal F)[1])$ be the unique morphism extending the identity on $Y^n\times (0,1)$. Then, $i_l^*(\Omega)[1]\colon {}^p(\mathbf j^n)_! i_l^*a_\gamma^*\mathcal F\to {}^p(\mathbf j^n)_* i_l^*a_\gamma^*\mathcal F$ is the unique morphism extending the identity on $Y^n$. Passing to the images, we obtain the equality
\begin{equation}
{}^p(\mathbf j^n)_{!*} i_l^*a_\gamma^*\mathcal F=\left((i'_l)^*\circ ({}^p(\mathbf j^n\times \id)_{!*} (a_\gamma^*\mathcal F)[1])\right)[-1].
\end{equation}
Making use of Remark \ref{rem:open-embedding} (2) there holds, for $\bullet=!, *$ or $!*$
\begin{align*}
 &{}^p(\mathbf j^n)_\bullet M_{\gamma,\FF}={}^p(\mathbf j^n)_\bullet  i_1^*\alpha_{\FF} =(i'_1)^*\, {}^p(\mathbf j^n\times \id)_\bullet\alpha_\FF\mbox{ and }\\
 &(i'_0)^*\, {}^p(\mathbf j^n\times \id)_\bullet\alpha_\FF={}^p(\mathbf j^n)_\bullet  (i_0)^*\alpha_{\FF}={}^p(\mathbf j^n)_\bullet \id_{b_{\varphi}^*\FF}=\id_{b^{'*}_{\varphi} ({}^p\mathbf j_\bullet\FF)}
 \end{align*} where we used \eqref{dgm:lower*and!} for the last equality. 
Hence $\alpha'_\FF={}^p(\mathbf j^n\times \id)_\bullet\alpha_\FF$, giving \eqref{eq:monodromy-estension}.
\end{proof}

\section{factorized sheaves}\label{facdata}

The functors $\boxtimes_{\varphi}$,  $a_\varphi^*$ for $\varphi\in E_2(n)$ and $n\in\NN$, allow us to introduce the notion of factorization data  for objects of $\mathcal{P}(1)$ and factorized perverse sheaves on $\sym(U)$ as in \cite{KS3}.
The restricted and truncated analogues  defined in Subsections \ref{sec:a} and \ref{sec:tens-phi} for $d\in\NN$ allow us to expand these notions and introduce the truncated analogues: factorization data for objects in   $\mathcal{P}_{\leq d}(1)$ and $\mathcal{P}^{\leq d}(1)$ and (truncated) factorized perverse sheaves on $\mathrm{Sym}_{\leq d}(U)$ and $\sym^{\leq d}(U)$. We use the convention for $b_{\varphi}$ and $b^*_{\varphi}$  from page \pageref{page:b}.

\begin{defn}\label{def:facdat}
    Let $d\in\NN$ and let $\mathcal{F}$ be an object of $\mathcal{P}(1)$,  $\mathcal{P}_{\leq d}(1)$,   $\mathcal{P}^{\leq d}(1)$, respectively.  
    A factorization datum on $\mathcal{F}$ is the datum of a family of isomorphisms
   \begin{equation}\label{eq:mu-phi}
    \mu_\varphi : \mathcal{F}^{\boxtimes_\varphi n}
    \rightarrow
     b_\varphi^\ast \mathcal{F}\ \ , \ \ 
    \varphi \in E_2(n),\,n\in\NN
    \end{equation}
    satisfying the conditions: 
    \begin{enumerate}
                \item\label{item:b}
        {\bf Braiding-monodromy compatibility.} If $\gamma$ represents the unique homotopy class of paths in $E_2(1)$ joining $\varphi$ to $\id_U$, the diagram below is commutative
        \begin{equation}\label{eq:mon-bra}
            \begin{tikzcd}
\boxtimes_{\varphi} \FF 
\arrow[rr, "\mu_\varphi"] 
\arrow[dr, swap, "\underline{R}_{\gamma,\FF}"]& 
& b_\varphi^\ast \mathcal{F}
\arrow[dl, "M_{\gamma,\FF}"]\\
&\boxtimes_{\id_U} \FF=\FF=b_{\id_U}^\ast \mathcal{F}
\end{tikzcd} 
        \end{equation}
where $\underline{R}_{\gamma,\FF}$ and $M_{\gamma,\FF}$ are the components relative to $\FF$ of the isomorphisms of functors $\underline{R}_\gamma$  and  $M_\gamma:=M_{\varphi,\id_U}$ as in \eqref{eq:defR} and Lemma \ref{lem:compat-a-M}.
        \item \label{a}
        {\bf Compatibility with operadic composition.} For every $\varphi\in E_2(n)$ and every $\psi_l\in E_2(m_l)$ for $l=1,\,\ldots,\,n$, with composition
        $\eta = 
        \varphi \circ 
        (\psi_1, \ldots, \psi_n)$, the diagram

   \begin{equation}\label{eq:operadic}
\begin{tikzcd}
 (\Pi_{l=1}^n b_{\psi_l})^\ast\mathcal{F}^{\boxtimes_\varphi n}     \arrow[rr, "(\Pi_{l=1}^n b_{\psi_l})^\ast\mu_\varphi"]  & &(\Pi_{l=1}^n b_{\psi_l})^\ast
b_\varphi^\ast \mathcal{F}  
      \arrow[d, "="]  \\
 \boxtimes_{\varphi}b_{\psi_l}^\ast\mathcal{F} 
\arrow[u, "="] & &b_\eta^\ast \mathcal{F}\\
\arrow[u, "\boxtimes_\varphi(\mu_{\psi_l})_{l=1}^n"]
\boxtimes_\varphi(\FF^{\boxtimes_{\psi_1}m_1},\,\ldots,\,\FF^{\boxtimes_{\psi_n}m_n})
      \arrow[rr, "="]
 & & \mathcal{F}^{\boxtimes_\eta n}
 \arrow[u, "\mu_\eta"']\\
\end{tikzcd} 
\end{equation}
is commutative.
\item\label{item:c} {\bf Symmetric equivariance.} For every $\varphi\in E_2(n)$ and for every $\sigma\in {\mathbb S}_n$, the isomorphisms
\begin{equation}\label{eq:equivariance}
\mu_{\varphi^\sigma}\colon \mathcal{F}^{\boxtimes_{\varphi^\sigma}n}\to b_{\varphi^\sigma}^*{\mathcal F}\quad \mbox{ and }
\quad \sigma^*(\mu_\varphi)\colon \sigma^*\left(\mathcal{F}^{\boxtimes_{\varphi}n}\right)\to \sigma^*(b_{\varphi}^*{\mathcal F})
\end{equation} coincide up to the identification \eqref{eq:identification} and application of \eqref{eq:a-phi*-equivariance}.
 \end{enumerate}

\medskip

We denote by $\mathrm{FD}(\mathcal{F})$ the set of factorization data on $\mathcal{F}$. 
Any pair $(\FF,\mu)$ with $\mu\in \mathrm{FD}(\mathcal{F})$ is called a factorized perverse sheaf.  
\label{page:FD}
\end{defn}

\begin{remark}{\rm Let $(\FF,\mu)$ be a factorized perverse sheaf on $\sym(U),\,\sym_{\leq d}(U)$ or $\sym^{\leq d}(U)$ and let ${\rm Aut}(\FF)$ be the group of perverse sheaves automorphisms of $\FF$. For $\upsilon\in {\rm Aut}(\FF)$, the collection $\upsilon\cdot\mu$ of isomorphisms $\upsilon\cdot\mu_{\varphi}\colon \FF^{\boxtimes_\varphi n}\to b_\varphi^*(\FF)$ for $\varphi\in E_2(n)$ for $n\in \NN$ defined as 
$\upsilon\cdot\mu_{\varphi}:=b_\varphi^*(\upsilon)\circ \mu_\varphi\circ \boxtimes_\varphi(\upsilon^{-1})_{l=1}^n$ is again a factorization datum for $\FF$. Indeed: naturality of the monodromy and braiding give commutativity of \eqref{eq:mon-bra}; \eqref{eq:operadic} follows from Lemma \ref{lem:b-passa}, functoriality of $\boxtimes$ and compatibility with operadic composition of $\boxtimes$; and \eqref{eq:equivariance} follows from \eqref{eq:a-phi*-equivariance} and \eqref{eq:boxtimes-sigma}. 
Functoriality of $b_\varphi^*$ and $\boxtimes_\varphi$ ensures that $(\upsilon\circ\upsilon')\cdot\mu=\upsilon\cdot(\upsilon'\cdot \mu)$ for any $\upsilon,\upsilon'\in {\rm Aut}(\FF)$, so ${\rm Aut}(\FF)$ acts on $\mathrm{FD}(\FF)$. }
\end{remark}

\medskip

Let $\varphi\in E_2(n)$, let $Y$ be either $\sym(U)$, $\sym_{\leq d}(U)$ or $\sym^{\leq d}(U)$ and let $\alpha=(\alpha_1,\ldots,\alpha_n)\in \mathbb N^n$. We consider the restriction $b_{\varphi,\alpha}$ of $b_{\varphi}$ to the intersection of $Y^n$ with $\sym^\alpha(U):=\sym^{\alpha_1}(U)\times\cdots\times\sym^{\alpha_n}(U)$. The morphism $\mu_{\varphi}$ as in $\eqref{eq:mu-phi}$ is a graded isomorphism of perverse sheaves on $Y^n$ and we denote by $\mu_{\varphi,\alpha}$ its restriction to $Y^n\cap\sym^{\alpha}(U)$.  
Let now $m_l\in\NN$ for $l=0,\,\ldots,\,n$, with $m_0=0$, and set $m=\sum_{l=1}^n m_l$ and let $\varphi\in E_2(n)$ and $\underline{\psi}=(\psi_1,\,\ldots,\,\psi_n)\in \prod_{l=1}^nE_2(m_l)$. For 
$\theta=(\theta_1,\,\ldots,\theta_m)\in \NN^m$, set $\beta_l=(\theta_{m_1+\cdots+ m_{l-1}+1},\ldots,\theta_{m_1+\cdots+m_l})$ for $l=1,\ldots, n$ so $|\beta_l|=\sum_{j=m_1+\cdots+ m_{l-1}+1}^{m_1+\cdots+ m_{l-1}+m_l}\theta_j$. Then, \eqref{eq:operadic} restricted to $\prod_{j=1}^m\sym^{\theta_j}(U)$ becomes \label{p:a-phi-alpha}
\begin{equation}\label{eq:factdat_a_partition}
  \mu_{\varphi\circ\underline{\psi},\theta} = 
((\Pi_{l=1}^n b_{\psi_l,\beta_l})^\ast(\mu_{\varphi,(|\beta_1|,\,\ldots,\,|\beta_n|)}))\circ
 \boxtimes_\varphi(\mu_{\psi_l,\beta_l})_{l=1}^n.    
\end{equation}


\begin{example}\label{ex:F0}
{\rm Let $Y$ be either $\sym(U)$, $\sym^{\leq d}(U)$ or $\sym_{\leq d}(U)$, let $(\FF,\mu)$
be a factorized perverse sheaf on $Y$. We consider the restrictions $\FF_0$ and $\FF_1$ of $\FF$ to $\sym^0(U)$ and $\sym^1(U)$, respectively and the compatibility morphisms $\mu_{\varphi,\alpha}$ for $\varphi\in E_2(n)$ and $|\alpha|\leq 1$.}
\begin{enumerate}{\rm
    \item\label{item:0} The complex $\FF_0$  is a perverse sheaf on $Y(0)=\sym^0(U)=\{\so\}$, a point, so $\FF_0$ is just the stalk $\FF_{\so}$, an object in $\Vu$, concentrated in degree 0. In addition, $E_2(0)$ contains only the (trivial) linear embedding $\varphi_0\colon U^{\coprod 0}=\emptyset \to U$. Applying  the functor $\sym$ to it, possibly followed by restriction or truncation, we get the inclusion $b_{\varphi_0}\colon \{\so\}\to Y$, whose restriction $(b_{\varphi_0})^{\leq0}$ to the $0$-th component is just the identity map $\{\so
    \}\to\{\so\}$. Hence, the restriction $(\mu_{\varphi_0})^{\leq 0}$ of $\mu_{\varphi_0}$ to  $\{\so\}$ is an isomorphism ${\mathbf 1}_{\Vu}\simeq\FF_{\so}^{\boxtimes_{\varphi_0}}\to ((b_{\varphi_0})^{\leq0})^*(\FF_0)=\FF_{\so}$. \\
    Let now $\varphi\in E_2(n)$ for $n\in\NN_{\geq1}$ and $\alpha=(0,\ldots,0)$. The restriction $b_{\varphi,\alpha}$ of $b_\varphi$ is the constant map $\{\so\}^n\to\{\so\}$ and  $\varphi\circ(\varphi_0^n)=\varphi_0$, hence by  \eqref{eq:factdat_a_partition} the restriction 
    $\mu_{\varphi,\alpha}\colon \left(\FF^{\boxtimes_\varphi n}\right)_{(\so,\ldots,\so)}\to (b_{\varphi}^*\FF)_{(\so,\ldots,\so)}$ is  the isomorphism  
  \begin{equation*}
 \left(\FF^{\boxtimes_\varphi n}\right)_{(\so,\ldots,\so)}\simeq \FF_{\so}^{\boxtimes_\varphi n}\longrightarrow
  {\mathbf 1}_\Vu^{\boxtimes_{\varphi}n}\simeq  {\mathbf 1}_\Vu\longrightarrow \FF_{\so}\simeq b_{\varphi,\alpha}^*\FF_0\simeq (b_{\varphi}^*\FF)_{(\so,\ldots,\so)},\end{equation*} where the arrows are the restrictions of $\left(\boxtimes_{\varphi}((\mu_{\varphi_0})^{\leq 0})^n\right)^{-1}$ and $\mu_{\varphi_0}$, respectively. In other words, with the appropriate identifications, $\mu_{\varphi,\alpha}$ corresponds to ${\mathbf 1}_\Vu^{\boxtimes_{\varphi}n}\simeq  {\mathbf 1}_\Vu$. 
    \item\label{item:mu-1} Let $d\in\NN_{\geq1}$.  Since $Y(1)=\sym^1(U)=U$ is contractible, $\FF_1$ can be identified with its stalk at a point in $U$, that is, a complex given by an object of $\Vu$ concentrated in degree $-1$. Then, for $n\in\NN_{\geq1}$
\[Y^n(1)=
\coprod_{l=0}^{n-1}\left(\{\so\}^l\times U\times\{\so\}^{n-l-1}\right)\] 
and so for any $\psi\in E_2(n)$ and any $l\in\{0,\,\ldots,n-1\}$
the restriction $b_{\psi, (0^{l}, 1,0^{n-l-1})}$ of $b_\psi$ is homotopic to the natural identification $\{\so\}^l\times U\times\{\so\}^{n-l-1}\to U$. In addition, if $\psi_{l+1}$ is the $(l+1)$-th unary embedding in $\psi$, there holds $\psi\circ (\varphi_0^{l},\id_U,\varphi_0^{n-l-1})=\psi_{l+1}$. Applying \eqref{eq:factdat_a_partition} we see that on $\{\so\}^l\times U\times\{\so\}^{n-l-1}$, after the appropriate identifications induced by $\mu_{\varphi_0}$, the isomorphism
\begin{equation*}\mu_{\psi, (0^{l}, 1, 0^{n-l-1})}\colon \FF^{\boxtimes_\psi n}\simeq\FF_1\to \FF_1\simeq b_{\psi, (0^{l}, 1,0^{n-l-1})}^*\FF\end{equation*} corresponds to the isomorphism  $\mu_{\psi_{j+1}}\colon \FF_1^{\boxtimes_{\psi_{l+1}}1}\simeq \FF_1\to b_{\psi_{l+1}}^*(\FF_1)$, which is uniquely determined by  \eqref{eq:mon-bra}.}
\end{enumerate}
\end{example}

\begin{defn} \label{def:FP}
   For $d\in\NN$, the categories $\mathcal{FP}$, 
     $\mathcal{FP}_{\leq d}$,  $\mathcal{FP}^{\leq d}$, are the categories whose objects are factorized perverse sheaves on $\sym(U)$, $\sym_{\leq d}(U)$, and $\sym^{\leq d}(U)$, respectively. The morphisms from the object $(\mathcal{F}, \mu)$ to $(\GG, \nu)$ 
     are the morphisms of perverse sheaves
    \[
    f: \mathcal{F} \rightarrow \GG
    \]
   satisfying:
   \begin{enumerate}
   \item[(a)] the $0$-th component $f_0\colon{\mathbf 1}_\Vu\simeq\FF_0\to\GG_0\simeq{\mathbf 1}_\Vu$ is the identity;
   \item[(b)] for every $n\in\NN$ and any $\varphi\in E_2(n)$ the diagram
    \begin{equation}\label{eq:def-morphisms}
            \begin{tikzcd}[row sep=huge]
 \mathcal{F}^{\boxtimes_{\varphi} n} 
\arrow[rr, "\mu_\varphi"] 
\arrow[d, swap, "f^{\boxtimes_\varphi n}"]& 
& b_\varphi^\ast \mathcal{F}
\arrow[d, "b_\varphi^\ast f"]\\
 \GG^{\boxtimes_{\varphi} n}
\arrow[rr, "\nu_{\varphi}"]  & 
& b_{\varphi}^\ast \GG
\end{tikzcd} 
        \end{equation}
        is commutative. 
        \end{enumerate}

               \end{defn}

\begin{remark}{\rm Even though under our assumptions on $\Vu$ the category $\PPP(1)$ is abelian and $k$-linear, the category $\FP$ is not abelian nor $k$-linear, as linear combinations of morphisms are no longer morphisms. 

The category $\FP$ has been extensively studied in \cite{KS3}, where it is shown  to be equivalent to the category of connected bialgebras in $\Vu$. The main motivation to introduce $\FP_{\leq d}$ and $\FP^{\leq d}$ is to give an analogous result for the truncated version of connected bialgebras defined in \cite{CERyD-A} and to give a geometric counterpart to the $d$-th approximation functor constructed there.}
\end{remark}

It follows from Example \ref{ex:F0} \eqref{item:mu-1} that the restriction to $\sym^1(U)$ on perverse sheaves and morphisms in $\PPP(1)$, $\PPP_{\leq d}(1)$, or $\PPP^{\leq d}(1)$ for $d\in\NN_{\geq 1}$ induces functors 
 \begin{equation}\label{eq:varrho}\varrho\colon\FP\longrightarrow\Vu,\quad \varrho^d\colon\FP^{\leq d}\longrightarrow\Vu,\mbox{ and }\varrho_d\colon\FP_{\leq d}\longrightarrow\Vu.\end{equation}

Using these functors the categories  $\FP^{\leq0}$ and $\FP^{\leq1}$ can be described very concretely. 
\begin{prop}\label{prop:varrho}With the above notation
 \begin{enumerate}
 \item $\varrho^1$  and $\varrho_1$ are faithful.
     \item     The categories $\FP_{\leq0}$ and $\FP^{\leq0}$ are both equivalent to the category whose only object is ${\mathbf 1}_\Vu$ and whose only morphism is the identity morphism.
     \item The functor $\varrho^1$ is an equivalence  $\mathcal{FP}^{\leq 1}\to\Vu$. 
 \end{enumerate}
\end{prop}
\begin{proof}
(1) By Example \ref{ex:F0}, a morphism $f\colon\FF\to\GG$ in $\FP^{\leq1}$ is a pair $(f_0,f_1)$ where $f_0=\id_{\mathbf 1_\Vu}$ and  $f_1\colon\FF_1\to\GG_1$ is a morphism of objects in $\Vu$, hence $\varrho^1$  is faithful. 

We turn to $\varrho_1$. 
Observe that, for any $n\in\NN_{\geq1}$, any point in $\sym^n_{\leq 1}(U)=\sym^n_{\neq}(U)$ has a neighbourhood of the form $b_{\varphi,(1^n)}(\sym^1(U)^n)$. If two morphisms $f,f'\colon \FF\to\GG$ in $\FP_{\leq1}$ coincide on $\FF_1$, then they coincide on any such neighbourhood by the restriction of \eqref{eq:def-morphisms} to $(\sym^1(U))^n$.

(2) By definition $\FP_{\leq0}=\FP^{\leq0}$. 
Example \ref{ex:F0} \eqref{item:0} shows that $\FP^{\leq0}$ has only the trivial object, and by definition of morphisms, the only morphism is the identity.

(3)  Any object $V$ in $\Vu$,  viewed as a complex concentrated in degree $-1$ is a perverse sheaf on $\sym^1(U)$, so $({\mathbf 1}_\Vu,V)\in\PPP^{\leq1}(1)$ by Example \ref{ex:F0} \eqref{item:0}. 
We equip $({\mathbf 1}_\Vu,V)$ with a factorization datum $\nu$ as follows. Let $\varphi\in E_2(n)$. If $n=0$, then $\nu_{\varphi}:=\id_{{\mathbf 1}_\Vu}$; if $n=1$, then $\nu_{\varphi}$ is defined as to satisfy \eqref{eq:mon-bra}; if $n\geq2$, then $\nu_{\varphi}$ is as in Example \ref{ex:F0} \eqref{item:mu-1}. By construction, $\nu$ is compatible with operadic composition and symmetric equivariance. 
Thus $(({\mathbf 1}_\Vu,V),\nu)\in\FP^{\leq 1}$ and $\varrho^1(({\mathbf 1}_\Vu,V),\nu)=V$ so $\varrho^1$ is essentially surjective. Let now $(\FF,\mu)$, and  $(\GG,\nu)$ be objects in $\FP^{\leq 1}$, and let $f_1\colon \FF_1\to\GG_1$ be a morphism in $\Vu$. Then, setting $f_0\colon {\mathbf 1}_\Vu\to {\mathbf 1}_\Vu$ gives a morphism $f\colon \FF\to\GG$ in $\PPP^{\leq1}(1)$. Condition \eqref{eq:def-morphisms} is satisfied for $\varphi\in E_2(0)$ by construction, for $\varphi\in E_2(1)$ by naturality of braiding and monodromy and \eqref{eq:mon-bra}, and for $\varphi\in E_2(n)$ for $n\geq 2$ in virtue of Example \ref{ex:F0} \eqref{item:mu-1}. Hence, $\varrho^1$ is full.  
\end{proof}

\begin{remark}{\rm The category $\FP$ is not monoidal: indeed by \cite{KS3} it is equivalent to the category $\mathcal{CB}(\Vu)$ of connected bialgebras in $\Vu$, which is not monoidal if $\Vu$ is not symmetric, see for instance \cite[Proposition 1.10.12, Corollary 1.6.10]{HS}. Even if by Proposition \ref{prop:varrho} we can equip $\FP^{\leq 1}$ with a braided monoidal structure,
we do not expect $\FP^{\leq d}$ nor $\FP_{\leq d}$ to have a monoidal structure for $d>1$. }
\end{remark}

The compatibility between braiding and monodromy extends to all components of a factorization datum. 

\begin{prop}\label{prop:mon-bra-basta1} Let $d\in\NN$ and let $\mathcal{F}$ be an object of $\mathcal{P}(1)$,  $\mathcal{P}_{\leq d}(1)$,   $\mathcal{P}^{\leq d}(1)$, respectively and let $\nu=(\nu_{\varphi})_{\varphi\in E_2(n), n\in\NN}$ be a collection of isomorphisms
$\FF^{\boxtimes_\varphi n}\to b_{\varphi}^*(\FF)$ of perverse sheaves on $\sym(U)^n,\,\sym_{\leq d}(U)^n$ or $\sym(U)^n(\leq d)$ satisfying \eqref{eq:mon-bra} and \eqref{eq:operadic} in the appropriate spaces.
Then, for every path $\gamma$ in $E_2(n)$ joining $\varphi$ to $\psi$, the diagram below is commutative
        \begin{equation}\label{eq:mon-bra-extended}
            \begin{tikzcd}[row sep=huge]
 \mathcal{F}^{\boxtimes_{\varphi} n} 
\arrow[rr, "\nu_\varphi"] 
\arrow[d, swap, "\underline{R}_{\gamma,\FF^n}"]& 
& b_\varphi^\ast \mathcal{F}
\arrow[d, "M_{\gamma,\FF}"]\\
 \mathcal{F}^{\boxtimes_{\psi} n}
\arrow[rr, "\nu_{\psi}"]  & 
& b_{\psi}^\ast \mathcal{F}
\end{tikzcd} 
        \end{equation}
where: $\underline{R}_{\gamma,\FF^n}$ is the component relative to $(\FF,\,\ldots,\,\FF)$ of the isomorphism of functors  $\underline{R}_\gamma$  as in \eqref{eq:defR} and $M_{\gamma,\FF}$ is the component relative to $\FF$ of the isomorphism of functors   as in  Lemma \ref{lem:mndrmy}.
\end{prop}
\begin{proof}By Remark \ref{piE1} (2) it is enough to consider the case in which there exist
$\theta\in E_2(n)$ and $\varphi_j,\psi_l\in E_2(1)$ for $j,l=1,\,\ldots,\,n$ such that
$\varphi=\theta\circ(\varphi_1,\ldots,\,\varphi_n)$, $\psi=\theta\circ(\psi_1,\ldots,\,\psi_n)$  and such that $\gamma=\theta\circ(\gamma_1,\ldots,\,\gamma_n)$, for some paths
$\gamma_j$ in $E_2(1)$ from $\varphi_j$ to $\psi_j$, for $j=1,\,\ldots,\,n$.

By Remark \ref{rem:erre-operadic} and \eqref{eq:operadic}, commutativity of  \eqref{eq:mon-bra-extended} is equivalent to commutativity of
 \begin{equation}\label{eq:doppio}
            \begin{tikzcd}[row sep=huge]
\boxtimes_{\theta}\left(\boxtimes_{\varphi_l}(\FF)_{l=1}^n\right) \arrow[rr, "\boxtimes_\theta(\nu_{\varphi_l})_{l=1}^n"] \arrow[d, swap, "\boxtimes_\theta(\underline{R}_{\gamma_l,\FF})_{l=1}^n"]&&\boxtimes_\theta (b_{\varphi_l}^*(\FF))_{l=1}^n\arrow[d,"\boxtimes_{\theta}(M_{\gamma_l,\FF})_{l=1}^n"]\arrow[rr,"(\prod_{l=1}^nb_{\varphi_l})^*(\nu_\theta)"]&&b_{\theta\circ(\varphi_1,\ldots,\,\varphi_n)}^*(\FF)
\arrow[d, "M_{\gamma,\FF}"]\\
\boxtimes_{\theta}\left(\boxtimes_{\psi_l}(\FF)_{l=1}^n\right) \arrow[rr, "\boxtimes_\theta(\nu_{\psi_l})_{l=1}^n"] &&\boxtimes_\theta (b_{\psi_l}^*(\FF))_{l=1}^n\arrow[rr,"(\prod_{l=1}^nb_{\psi_l})^*(\nu_\theta)"]&&b_{\theta\circ(\psi_1,\ldots,\,\psi_n)}^*(\FF)
\end{tikzcd} 
        \end{equation}
        on $\sym(U)^n$, $\sym_{\leq d}(U)^n$, or $\sym(U)^n(\leq d)$. 
The right square of \eqref{eq:doppio} commutes in virtue of Lemma \ref{lem:compat-a-M-factorized}, so it is enough to prove commutativity of the diagram below for each $l=1,\,\ldots,\,n$
\begin{equation}
            \begin{tikzcd}[row sep=huge]
\boxtimes_{\varphi_l}\FF\arrow[rr, "\nu_{\varphi_l}"] \arrow[d, swap, "\underline{R}_{\gamma_l,\FF}"]&&b_{\varphi_l}^*(\FF)\arrow[d,"M_{\gamma_l,\FF}"]\\
\boxtimes_{\psi_l}\FF\arrow[rr, "\nu_{\psi_l}"]&&b_{\psi_l}^*(\FF) 
\end{tikzcd} 
        \end{equation}
 The latter follows from \eqref{eq:mon-bra} and compatibility of monodromy and braiding with composition of paths, once we decompose  the homotopy class of $\gamma_l$ as the composition  of the unique homotopy class of paths from $\varphi_l$ to $\id_U$ with the inverse of the unique homotopy class of paths from $\psi_l$ to $\id_U$. 
\end{proof}

\subsection{Vertical factorization data}

The definition of factorization data requires the existence of isomorphisms for any $\varphi\in E_2(n)$, which might be hard to verify. It is therefore convenient to determine a more controllable set of requirements ensuring existence of a factorization datum. We show in this section how one can deduce the existence of a factorization datum once  isomorphisms are given for any $\varphi$ running in the smaller set $E^v_2(n)$ of vertically disjoint linear embeddings as in Definition \ref{def:linearemb}.

\begin{defn}\label{def:vfact_dat}
     Let $d\in\NN$ and let $\mathcal{F}$ be an object of $\mathcal{P}(1)$,  $\mathcal{P}_{\leq d}(1)$,  $\mathcal{P}^{\leq d}(1)$, respectively.

     A {vertical} factorization datum on $\mathcal{F}$ is the datum of a family of isomorphisms
    \[
    \mu_\varphi : 
    \mathcal{F}^{\boxtimes_\varphi n}
    \rightarrow
    b_\varphi^\ast \mathcal{F}, \ \ 
    \varphi \in E^v_2(n)
    \]
    satisfying: 
    \eqref{eq:mon-bra}, \eqref{eq:operadic} and \eqref{eq:equivariance} for 
    $\varphi, \eta, \psi_1, \ldots, \psi_n$ {\rm vertically disjoint} linear embeddings, together with commutativity of diagram \eqref{eq:mon-bra-extended} whenever $\varphi\in E_2(2)$ and $\gamma$ is the elementary braiding in $\Pi_1(E^v_2(2))$ from $\varphi$ to $\psi=\varphi^{(1,2)}$.  

\medskip

We denote by $\mathrm{FD}^v(\mathcal{F})$  the set of vertical factorization data on $\mathcal{F}$.
\end{defn}

\begin{defn}\label{def:vertical-categories} Let $d\in\NN$ and let $Y$ be either $\sym(U)$ or $\sym_{\leq d}(U)$ or $\sym^{\leq d}(U)$.
    A vertically factorized perverse sheaf on $Y$ is a pair $(\mathcal{F}, \mu)$, where $\mathcal{F}$ is a perverse sheaf on $Y$, and  $\mu\in$ $\mathrm{FD}^v(\mathcal{F})$. We define the categories
 $\mathcal{FP}^v$,  $(\mathcal{FP}_{\leq d})^v$, and $(\mathcal{FP}^{\leq d})^v$, of vertically factorized perverse sheaves on $\sym(U)$, $\sym_{\leq d}(U)$ and $\sym^{\leq d}(U)$, respectively, where objects are vertically factorized perverse sheaves and morphisms are morphisms of perverse sheaves that are the identity in degree $0$ and satisfy \eqref{eq:def-morphisms} for  $\mu$ and $\nu$  vertical factorization data.
\end{defn}

The definition of vertical  factorization datum in Definition~\ref{def:vfact_dat} requires commutativity of diagram \eqref{eq:mon-bra-extended} only for the simple braiding of $E_2 (2)$. The following lemma shows that this property together with \eqref{eq:mon-bra} and \eqref{eq:operadic} implies compatibility of braiding and monodromy  for  paths in $E_2(n)$ for any $n\geq 2$.

\begin{lem}\label{lem:vbfd}Let $d\in\NN$ and let $Y$ be either $\sym(U)$ or $\sym_{\leq d}(U)$ or $\sym^{\leq d}(U)$.
    Let $\FF$ be a perverse sheaf on $Y$ and  let $\mu \in FD^v (\mathcal{F})$. Then, for any  $n\geq2$ and  any $\varphi,\,\psi\in E_2^v(n)$, 
       the  diagram \eqref{eq:mon-bra-extended} commutes
     for any morphism $\gamma$ from $\varphi$ to $\psi$ in $(\Pi_1 E_2(n))^v$, i.e., for any homotopy class of paths in $E_2(n)$ from $\varphi$ to $\psi$.
\end{lem}

\begin{proof} Since the morphisms in the fundamental groupoid $\Pi_1 E_2(n)$ are compositions of elementary braidings, that are paths 
from $\varphi\in E_2^v(n)$ to $\varphi^{(l,l+1)}$ for some $l\in\{1,\,\ldots,n-1\}$,  it is sufficient to show commutativity of \eqref{eq:mon-bra-extended} for such paths. Operadic composition guarantees that we can decompose $\varphi$ and  $\varphi^{(l,l+1)}$ as follows \[
    \varphi = \theta \circ (\theta_1, \varepsilon, \theta_2) \ \ \mathrm{ and }\ \ 
    \varphi^{(l,l+1)} = \theta \circ (\theta_1,  \varepsilon^{(1,2)}, \theta_2) 
    \]
    for appropriate $\theta\in E_2^v(3),  \theta_1\in E_2^v(l-1), \theta_2\in E_2^v(n-l-1)$, and $\varepsilon\in E_2^v(2)$.

    \medskip

    Let $\sigma$ be the elementary braiding of $\Pi_1 E_2(2)$ joining $\varepsilon \in E_2^v (2)$ to  $\varepsilon^{(1,2)}$ and let $\gamma$ be the elementary braiding starting at $\varphi$ and ending at $\varphi^{(l,l+1)}$, obtained 
  operadically as
    \[
    \gamma(t) = \theta \circ ({\theta_1}, {\sigma}(t), {\theta_2}), \quad \mbox{ for }t\in[0,1].
    \]
  
  Equation \eqref{eq:operadic}  gives 
    \begin{equation}\label{eq:uno}\mu_\varphi = (b_{\theta_1}^*,b_\varepsilon^*, b_{\theta_2}^*)(\mu_\theta) \circ \boxtimes_\theta (\mu_{\theta_1}, \mu_{\varepsilon}, \mu_{\theta_2}),\end{equation}
    \begin{equation}\label{eq:due}\mu_{\varphi^{(l,l+1)}} = (b_{\theta_1}^*,b_{\varepsilon^{(1,2)}}^*, b_{\theta_2}^*)(\mu_\theta) \circ \boxtimes_\theta (\mu_{\theta_1}, \mu_{\varepsilon^{(1,2)}}, \mu_{\theta_2}).\end{equation}  
We  consider then the diagram   
\begin{equation}
  \begin{tikzcd}
  \mathcal{F}^{\boxtimes_{\varphi} n} \arrow[rrr, "\mu_\varphi"] \arrow[d, swap,"\equiv"]
   &&& b_\varphi^\ast \mathcal{F}\arrow{d}{(b^*_{\theta_1},b_{\varepsilon}^*,b_{\theta_2}^*)(\mu_\theta)^{-1}}\\
\boxtimes_{\theta}(\FF^{\boxtimes_{\theta_1}(l-1)},\FF^{\boxtimes_{\varepsilon}2},\FF^{\boxtimes_{\theta_2}(n-l-1)})
\arrow{d}[swap]{\boxtimes_{\theta}(\id,\underline{R}_{\sigma,\FF^2},\id)} 
\arrow{rrr}{\boxtimes_\theta(\mu_{\theta_1},\mu_{\varepsilon},\mu_{\theta_2})}
&&&\boxtimes_{\theta}(b_{\theta_1}^*\FF, b_{\varepsilon}^*\FF,b_{\theta_2}^*\FF) 
\arrow{d}{\boxtimes_{\theta}(\id,M_{\sigma,\FF},\id)} \\
\boxtimes_{\theta}(\FF^{\boxtimes_{\theta_1}(l-1)},\FF^{\boxtimes_{\varepsilon^{(1,2)}}2},\FF^{\boxtimes_{\theta_2}(n-l-1)})
\arrow[d, swap, "\equiv"] \arrow{rrr}{\boxtimes_\theta(\mu_{\theta_1},\mu_{\varepsilon^{(1,2)}},\mu_{\theta_2})}
&&&\boxtimes_{\theta}(b_{\theta_1}^*\FF, b_{\varepsilon^{(1,2)}}^*\FF,b_{\theta_2}^*\FF)
\arrow{d}{(b^*_{\theta_1},b^*_{\varepsilon^{(1,2)}},b^*_{\theta_2})(\mu_\theta)}
\\
  \mathcal{F}^{\boxtimes_{\varphi^{(l,l+1)}} n} 
\arrow[rrr, "\mu_{\varphi^{(l,l+1)}}"]&&& b_{\varphi^{(l,l+1)}}^\ast\FF
\end{tikzcd}
\end{equation}
The middle square is obtained by applying $\boxtimes_\theta$ to the commutative square \eqref{eq:mon-bra-extended} for $\sigma$, 
making use of the fact that braiding and monodromy are trivial on the constant paths$\theta_1$ and $\theta_2$. 
The bottom and upper squares are \eqref{eq:uno} and \eqref{eq:due}, hence they are commutative. Now, the composition of the left vertical arrows is $\underline{R}_{\gamma,\FF^n}$ by Remark \ref{rem:erre-operadic}. Finally, $M_{\gamma,\FF}$ is  the composition of the right vertical arrows in virtue of Lemma \ref{lem:compat-a-M-factorized}. Commutativity of the external square gives the claim. 
\end{proof}

\begin{prop}\label{prop:extension}
Let $\mathcal{F}\in \mathcal{P}(1)$, $\mathcal{P}^{\leq d}(1)$ or  $\mathcal{P}_{\leq d}(1)$ and let $\mu\in\mathrm{FD}^v(\FF)$. Then, $\mu$ extends uniquely to an element $\widetilde{\mu}\in \mathrm{FD}(\FF)$.
\end{prop}
\begin{proof}
Let $n\in\NN$ and $\varphi\in E_2(n)$. Then there is a $\psi\in E_2^v(n)$ and a path $\gamma$ in $E_2(n)$ from $\varphi$ to $\psi$. Then $\mu_\psi$ is defined. 
 Imposing commutativity of diagram \eqref{eq:mon-bra-extended} uniquely determines an isomorphism $\Tilde{\mu}_\varphi\colon\FF^{\boxtimes_\varphi}\to b_\varphi^*(\FF)$. 

\medskip
 
 Lemma~\ref{lem:vbfd} ensures that the vertical factorization datum $\mu$ satisfies \eqref{eq:mon-bra-extended} for every path joining two vertically disjoint linear embeddings. This, combined with compatibility of $\underline R$ and $M$ with composition of paths, implies that $\Tilde{\mu}_\varphi$ does not depend on the choices of $\psi$ and $\gamma$, and that the collection $(\Tilde{\mu}_\varphi)_{\varphi\in E_2(n), n\in\NN}$ extends $\mu$. As a consequence, it satisfies \eqref{eq:mon-bra}.

 \medskip

We verify compatibility with operadic composition of $\widetilde{\mu}$. For $n\in\NN$, and $l=1,\,\ldots,\, n$ let $m_l\in\NN$ and  $\varphi\in E_2(n)$, $\psi_l\in E_2(m_l)$. Let then $\varphi'\in E^v_2(n)$, $\psi'_l\in E_2^v(m_l)$, and let $\gamma_0$ be a path in $E_2(n)$ from $\varphi$ to $\varphi'$ and for each $l$, let $\gamma_l$ be a path from $\psi_l$ to $\psi_l'$, so $\gamma:=\gamma_0\circ(\gamma_1,\,\ldots,\,\gamma_l)$ is a path from $\varphi\circ (\psi_1,\ldots,\psi_l)$ to $\varphi'\circ (\psi'_1,\ldots,\psi'_l)$.

We consider the diagram
\begin{equation}\label{eq:diagrammone}
\begin{tikzcd}[scale=0.5]
\boxtimes_\varphi(\FF^{\boxtimes_{\psi_l}m_l})_{l=1}^n\arrow[r,"\boxtimes_\varphi(\widetilde{\mu}_{\psi_l})_{l=1}^n"]\arrow[dd,"\boxtimes_{\varphi}(\underline{R}_{\gamma_l,\FF^{m_l}})_{l=1}^n"]&\boxtimes_\varphi(b_{\psi_l}^*\FF)_{l=1}^n\arrow[r,"\simeq"]\arrow[dd,"\boxtimes_{\varphi}(M_{\gamma_l,\FF})_{l=1}^n"]&(\prod_{l=1}^nb_{\psi_l})^*\FF^{\boxtimes_\varphi n}\arrow[rr,"(\prod_{l=1}^nb_{\psi_l})^*\widetilde\mu_{\varphi}"]\arrow[dd,  "(\prod_{l=1}^nb_{\psi_l})^*\underline{R}_{\gamma_0,\FF^n}"]
&&(\prod_{l=1}^nb_{\psi_l})^*b^*_\varphi\FF\arrow[dd, near start, swap, above=5pt, "(\prod_{l=1}^nb_{\psi_l})^*M_{\gamma_0,\FF}"]\\
\\
\boxtimes_\varphi(\FF^{\boxtimes_{\psi'_l}m_l})_{l=1}^n\arrow[r,"\boxtimes_\varphi({\mu}_{\psi'_l})_{l=1}^n"]\arrow[dd,"\underline{R}_{\gamma_0,(\boxtimes_{\psi_l'}^*\FF^{m_l})_{l=1}^n}"]   &\boxtimes_\varphi(b_{\psi'_l}^*\FF)_{l=1}^n\arrow[dd,"\underline{R}_{\gamma_0,(b_{\psi_l'}^*\FF)_{l=1}^n}"] &(\prod_{l=1}^nb_{\psi_l})^*\FF^{\boxtimes_{\varphi'} n}\arrow[rr,"(\prod_{l=1}^nb_{\psi_l})^*\mu_{\varphi'}"]\arrow[dd,  "M_{(\gamma_l)_{l=1}^n,\boxtimes_{\varphi'}\FF^n}"]&&(\prod_{l=1}^nb_{\psi_l})^*b^*_{\varphi'}\FF\arrow[dd,swap, "M_{(\gamma_l)_{l=1}^n,b_{\varphi'}^*\FF}"]\\
\\
\boxtimes_{\varphi'}(\FF^{\boxtimes_{\psi'_l}m_l})_{l=1}^n\arrow[r,"\boxtimes_{\varphi'}({\mu}_{\psi'_l})_{l=1}^n"]&\boxtimes_{\varphi'}(b_{\psi'_l}^*\FF)_{l=1}^n\arrow[r,"\simeq"]&(\prod_{l=1}^nb_{\psi'_l})^*\FF^{\boxtimes_{\varphi'} n}\arrow[rr,"(\prod_{l=1}^nb_{\psi'_l})^*\mu_{\varphi'}"]&&(\prod_{l=1}^nb_{\psi'_l})^*b^*_{\varphi'}\FF
\end{tikzcd}
\end{equation}
The upper left and right squares of \eqref{eq:diagrammone} commute by definition of $\widetilde{\mu}_{\psi_l}$ for $l=1,\,\ldots,\,n$, and $\widetilde{\mu}_\varphi$, respectively. 
The lower left and right squares commute by naturality of $\underline{R}$ and $M$, respectively. 
The middle vertical square commutes by naturality of monodromy, and the equality $(\prod_{l=1}^nb_{\psi_l})^*\underline{R}_{\gamma_0,\FF^n}=\underline{R}_{\gamma_0,(b_{\psi_l}^*\FF)_{l=1}^n}$, due to functoriality of $\underline{R}$. Hence, the outer square of of \eqref{eq:diagrammone} commutes. 
The composition of the lower horizontal arrows is $\mu_{\varphi'\circ(\psi_1',\,\ldots,\,\psi'_n)}$ because $\mu$ is compatible with operadic composition. The composition of the left and right vertical arrows are, respectively,  $\underline{R}_{\gamma,\FF^{\sum_{l=1}^nm_l} }$ and $M_{\gamma,\FF}$ because $\underline{R}$ and $M$ are compatible with  composition of paths. Hence, the composition of the upper horizontal arrows is $\widetilde{\mu}_{\varphi\circ(\psi_1,\,\ldots,\,\psi_n)}$ by definition of $\widetilde{\mu}$, giving \eqref{eq:operadic}. 

\medskip

 We verify symmetric equivariance for $\Tilde{\mu}$. For $n\geq2$, let $\varphi\in E_2(n)$ and $\psi\in E_2^v(n)$ be joined by the path $\gamma$, so $\widetilde{\mu}_\varphi:=M_{\gamma,\FF}^{-1}\circ\mu_\psi\circ \underline{R}_{\gamma,\FF^n}$. For $\sigma\in \Sn$ the path $\gamma^{\sigma}$ joins $\varphi^\sigma$ to $\psi^\sigma$ and so \(\widetilde{\mu}_{\varphi^{\sigma}}:=M_{\gamma^{\sigma},\FF}^{-1}\circ\mu_{\psi^\sigma}\circ \underline{R}_{\gamma^\sigma\FF^n}\).
 Now, for $\sigma$ acting on $\sym(U)^n$ we have   $\sigma^*(\mu_\psi)=\mu_{\psi^\sigma}$ by definition of vertical factorization datum. In addition, \(\underline{R}_{\gamma^\sigma,\FF^n}=\sigma^*(\underline{R}_{\gamma,\FF^n})\) by Lemma \ref{lem:equivarianceRbox} and 
 $M_{\gamma^{\sigma},\FF}=\sigma^*(M_{\gamma,\FF})$ by Lemma \ref{lem:mndrmy}. 
 Hence
 \[\widetilde{\mu}_{\varphi^{\sigma}}:=\sigma^*(M^{-1}_{\gamma,\FF})\circ \sigma^*(\mu_{\psi})\circ \sigma^*(\underline{R}_{\gamma,\FF^n})=\sigma^*(\widetilde{\mu}_\varphi).\]

We prove uniqueness. Assume $\widetilde\mu$ and $\widetilde\mu^{\prime}$ are  in $\mathrm{FD}(\mathcal{F})$ and that $\widetilde\mu_\psi=\widetilde\mu'_\psi=\mu_\psi$ for all $\psi\in E_2^v(n)$. Let $\varphi\in E_2(n)$, for $n\in\NN$ and let $\gamma$ be a path in $E_2(n)$ from $\varphi$ to some $\psi\in E_2^v(n)$. By Proposition \ref{prop:mon-bra-basta1} we have the following commutative diagram:
\[\begin{tikzcd}
	{\mathcal{F}^{\boxtimes_{\varphi}n}} && {a^{\ast}_{\varphi}(\mathcal{F})} \\
	\\
	{\mathcal{F}^{\boxtimes_{\psi}n}} && {a^{\ast}_{\psi}(\mathcal{F})}
	\arrow["{\underline{R}_{\gamma,\FF^n}}"', from=1-1, to=3-1]
	\arrow["{M_{\gamma,\FF}}", from=1-3, to=3-3]
	\arrow["{\mu_{\psi}=\mu_{\psi}^{\prime}}"', from=3-1, to=3-3]
	\arrow["{\mu_{\varphi}^{\prime}}"', shift right=1, from=1-1, to=1-3]
	\arrow["{\mu_{\varphi}}", shift left=1, from=1-1, to=1-3]
\end{tikzcd}\]
Since $\underline{R}_{\gamma,\FF^n}$ and $M_{\gamma,\FF}$ are isomorphisms, we conclude that  $\mu_{\varphi}=\mu_{\varphi}^{\prime}$. \end{proof}
\begin{prop}\label{prop:factvb}
Let $d\in\NN$. The forgetful functors 
\begin{equation}\label{eq:forgetful}\FP^{\leq d}\to  (\mathcal{FP}^{\leq d})^v,\quad \FP_{\leq d}\to  \mathcal{FP}_{\leq d}^v, \mbox{ and } \FP\to  \mathcal{FP}^v
\end{equation}
given: 
\begin{itemize}
\item on objects by neglecting the components of the factorization data indexed by embeddings in $E_2(n)\setminus E_2^v(n)$ for $n\in\NN_{\geq 2}$, 
\item on morphisms as the identity
\end{itemize}
are isomorphisms of categories. 
   \end{prop}
\begin{proof}
We need to prove that these functors are bijective on objects and fully faithful. 
The first property follows from Proposition \ref{prop:extension}. Faithfulness follows by construction. We show that they are full. 

Let $f\colon (\FF,\mu)\to(\GG,\nu)$ be a morphism in  $(\mathcal{FP}^{\leq d})^v$, $\mathcal{FP}_{\leq d}^v$, or  $\mathcal{FP}^v$. Let $\varphi\in E_2(n)$  and let $\gamma$ be a path in $E_2(n)$ from $\varphi$ to $\psi\in E_2^v(n)$. We consider the diagram

\begin{equation*}
 \begin{tikzcd}
  \mathcal{F}^{\boxtimes_{\varphi} n} \arrow[rr,"\underline{R}_{\gamma,\FF}"] \arrow[d, "f^{\boxtimes_\varphi n}"]&& \mathcal{F}^{\boxtimes_{\psi} n} \arrow[rr,"\mu_{\psi,\FF}"] \arrow[d, "f^{\boxtimes_\psi n}"] && b_\psi^\ast \mathcal{F} \arrow[d, "b_\psi^*(f)"] \arrow[rr, "M_{\gamma,\FF}^{-1}"]&& b_\phi^\ast \mathcal{F} \arrow[d, "b_\varphi^*(f)"]\\
   \mathcal{G}^{\boxtimes_{\varphi} n} \arrow[rr,"\underline{R}_{\gamma,\GG}"] && \mathcal{G}^{\boxtimes_{\psi} n} \arrow[rr,"\nu_{\psi,\GG}"]  && b_\psi^\ast \mathcal{G}  \arrow[rr, "M_{\gamma,\GG}^{-1}"]&& b_\varphi^\ast \mathcal{G}
 \end{tikzcd}
\end{equation*}
The left square commutes by naturality of $\underline{R}$, the middle square commutes because $\psi\in E_2^v(n)$ and the right square commutes by naturality of $M$. Since the composition of the top horizontal arrow is $\mu_\varphi$ and the composition of the bottom horizontal arrow is $\nu_\varphi$, we have commutativity of \eqref{eq:def-morphisms} for $\varphi$. Hence, the forgetful functors are full.
\end{proof}

\subsection{Restriction functors}

In this section we show that the truncation functors $(\ )^{\leq d}$ and $(\ )^{\leq d,\leq c}$, and the restriction functors induced by the open embeddings $j_{\leq d,\leq c}$ and $j_{\leq d}$  for $c\geq d\geq 0$ are well behaved with respect to factorization data. 

\begin{prop}\label{prop:def-funtori-restrizione}
Let $c,d,e\in \NN$ with $e\leq d\leq c$. The chain of restriction  functors \eqref{dgm:upperstar}  gives rise to  a chain of functors
\begin{equation}\label{dgm:upperstarFP}
\begin{tikzcd}
\FP
\arrow[rr, swap, "(j_{\leq c})^\ast"']& 
&  \FP_{\leq c} 
\arrow[rr, swap, "(j_{\leq d,\leq c})^\ast"'] &&  \mathcal{FP}_{\leq d} 
\arrow[rr, swap, "(\phantom{a})^{\leq d}"'] &&
\FP^{\leq d} 
\arrow[rr, swap, "(\phantom{a})^{\leq e,\leq d}"']&& \mathcal{FP}^{\leq e} 
\end{tikzcd} 
\end{equation}
\end{prop}
\begin{proof}For the truncation functors $(\ )^{\leq d}$ and $(\ )^{\leq d,\leq c}$ this is immediate. Let $\mathbf j\colon Y\to Y'$ be either $j_{\leq d}$ or $j_{\leq d,\leq c}$ and let
$(\mathcal{F},\mu)$ be a factorized perverse sheaf on $Y'$.
 
The diagram \eqref{dgm-tensor-upperstar} gives
    $({\mathbf j}^n)^\ast  \mathcal{F}^{\boxtimes_\varphi n}=({\mathbf j}^\ast\mathcal{F})^{\boxtimes_\varphi n}$ for any $\varphi\in E_2(n)$ and any $n\in\NN$. Commutativity of \eqref{dgm:upperstar} allows  
    to define ${\mathbf j}^*(\mathcal {F}, \mu):=({\mathbf j}^*\mathcal{F}, \nu)$ where   $\nu_{\varphi}:=({\mathbf j}^n)^*\mu_{\varphi}$ for any $\varphi\in E_2(n)$. 
   We show that this assignment gives indeed a factorized perverse sheaf on $Y$.

\medskip

 First of all we prove braiding-monodromy compatibility for $\nu$. Let $\gamma$ be a path in $E_2(1)$ from $\varphi$ to $\id_U$, let $M'$ and $\underline{R}'$ denote monodromy and braiding relative to the space $Y'$, and $M$ and $\underline{R}$ denote monodromy and braiding relative to the space $Y$. Applying the functor $\mathbf j^*$ to the diagram \eqref{eq:mon-bra} for $\underline{R}_\gamma',M_\gamma'$ and $\mu$ gives the statement, provided 
\begin{equation}\label{eq:comp-braiding-inclusion}\mathbf j^*{M}'_{\gamma, \FF}={M}_{\gamma,\, \mathbf j^*(\FF)}\mbox{ and }\mathbf j^\ast\underline{R'}_{\gamma,\FF}=\underline{R}_{\gamma,\, \mathbf j^*(\FF)}. \end{equation} 

The former holds by construction of $M_\gamma$. Since $\mathbf j\colon Y\to Y'$ is a stratified morphism,  the functor ${\mathbf j}^*\colon D^b_c(Y',\Sigma,\Vu)\to D^b_c(Y,\Sigma,\Vu)$ is braided monoidal, where the coherence morphisms are given by the natural identifications. The definition of the outer tensor product and Theorem \ref{thm:del} ensure that $\mathbf j^*$ satisfies \eqref{eq:comp-braiding-inclusion}.

\medskip
   
    The compatibility with operadic composition \eqref{eq:operadic} in this case would read as follows:
    \begin{equation}\label{eq:restr-operadic}\nu_{\eta}=({\mathbf j}^n)^*\mu_{\eta}= (\prod_{l=1}^n b_{\psi_{l}})^*(\mathbf j^n)^* (\mu_\varphi)
    \circ \boxtimes_{\varphi}\prod_{l=1}^n({\mathbf j}^{m_l})^*(\mu_{\psi_{l}})
    \end{equation} for $\eta=\varphi\circ(\psi_1,\,\ldots,\,\psi_n)$ with $\varphi\in E_2(n)$, $\psi_l\in E_2(m_l)$ for $l=1,\,\ldots,\,n$ and $m=\sum_{l=1}^nm_l$. 
   It can be obtained by applying  $(\mathbf j^m)^*$ to the equality  \eqref{eq:operadic} for $\mu$ making use of functoriality of $(\mathbf j^m)^*$ at the level of perverse sheaves and commutativity of \eqref{dgm:product-tanti} and \eqref{dgm:tensor-tanti}.

\medskip
 
We finally show symmetric equivariance. For $n\geq2$, let $\varphi\in E_2(n)$ and $\sigma\in \Sn$. The open embedding ${\mathbf j}^n$ commutes with the action of $\sigma$ on $Y^n$ and $(Y')^n$. Combining this with Lemma \ref{lem:tensoreq} and \eqref{eq:a-phi*-equivariance} 
gives the identifications 
\begin{align*}({\mathbf j^n})^*\FF^{\boxtimes_{\varphi^\sigma}n}&\simeq  
\sigma^*({\mathbf j^n})^*(\FF^{\boxtimes_{\varphi}n}),&&({\mathbf j}^n)^*\, (b^*_{\varphi^\sigma}\FF)\simeq ({\mathbf j}^n)^*(\sigma^* b^*_{\varphi}\FF)\simeq \sigma^*(\mathbf j^n)^* (b^*_{\varphi}\FF),\\
&{\rm and} &&\nu_{\varphi^\sigma}=({\mathbf j^n})^*\mu_{\varphi^\sigma}=({\mathbf j^n})^*\sigma^*(\mu_\varphi)=\sigma^* ({\mathbf j^n})^*\mu_\varphi=\sigma^*\nu_{\varphi}.
\end{align*}
   Hence, $\mathbf j^*(\FF,\nu)$ is a factorized perverse sheaf on $Y$. Good behavior with respect to morphisms follows, similarly, from compatibility of ${\mathbf j}^*$ with $\boxtimes_\varphi$ and $b^*_{\varphi}$, cf. \eqref{dgm:upperstar} and \eqref{dgm-tensor-upperstar}.
\end{proof}

The restriction  functors can also be defined on the (isomorphic) categories of vertical  factorized perverse sheaves, giving the following.

 \begin{cor}
The truncation  functors $(\ )^{\leq d}$ and the restriction functors $j^*_{\leq d}$ fit into the following commutative diagram, where the vertical arrows are the forgetful isomorphisms from Proposition \ref{prop:factvb}.  
   \begin{equation*}
   \begin{tikzcd}
   \FP \arrow[d,"\cong"] \arrow[rr, "j^*_{\leq d}"]&&\FP_{\leq d} \arrow[d,"\cong"] \arrow[rr,"(\ )^{\leq d}"]&&\FP^{\leq d} \arrow[d,"\cong"]\\
   \FP^v  \arrow[rr, "j^*_{\leq d}"]&&\FP^v_{\leq d}  \arrow[rr,"(\ )^{\leq d}"]&&(\FP^{\leq d})^v 
   \end{tikzcd}
   \end{equation*}
        \end{cor}
\begin{proof}
Commutativity of the diagram follows from the definition of the restriction and truncation functors in Proposition \ref{prop:def-funtori-restrizione}, and of the forgetful functors from Proposition \ref{prop:factvb}.
\end{proof}

\section{An equivalence of categories}\label{sec:equivalence}

In this Section we show that a factorization datum on $\sym^{\leq d}(U)$ contains enough information to be extended to $\sym_{\leq d}(U)$. This will allow us to construct a quasi-inverse $(\ \ )_{\leq d}$ to the truncation functor \(
(\ \ )^{\leq d}: \mathcal{FP}_{\leq d} \rightarrow \mathcal{FP}^{\leq d}
\). This is one of the main results of the paper.

\medskip

We first need to fix some additional notation. 

For $\alpha=(\alpha_1,\,\ldots,\,\alpha_r)\in\NN^r$ and $\varphi\in E_2(r)$, we set $\sym^{\alpha}(U):=\prod_{j=1}^r\sym^{\alpha_j}(U)$ and denote by
 $V_{\varphi,\alpha}$ the image of the open embedding $a_{\varphi,\alpha}\colon \sym^{\alpha}(U)\to \sym^{|\alpha|}(U)$. We call such an open subset a {\it special open subset}.

We denote by $c_{\varphi,\alpha}\colon V_{\varphi,\alpha}\to \sym^\alpha(U)$ the unique morphism for which $c_{\varphi,\alpha}\circ a_{\varphi,\alpha}=\id_{\sym^\alpha(U)}$, so $a_{\varphi,\alpha}\circ c_{\varphi,\alpha}\colon V_{\varphi,\alpha}\to \sym^{|\alpha|}(U)$ is the natural inclusion.

\medskip

\begin{remark}\label{rem:refinement}
{\rm Let $\alpha\in\NN^r$, $\varphi\in E_2(r)$, $\beta\in \NN^{r'}$ and $\psi\in E_2(r')$. Then, 
\begin{enumerate}
\item $V_{\varphi,\alpha}=V_{\psi,\beta}$  if and only if $r=r'$ and  there is $\sigma\in\mathbb S_r$ such that $\psi=\varphi^\sigma$ and $\beta_l=\alpha_{\sigma^{-1}(l)}$ for all $l=1,\,\ldots,\,r$;
\item $V_{\psi,\beta}\subset V_{\varphi,\alpha}$ if and only if, up to a permutation action as above, $\beta$ is a refinement of $\alpha$, that is, there are some $n_l\in \NN$ for $l=1,\,\ldots,\,r$, such that $\beta$ is obtained by juxtaposition of $r$ sequences $\beta_l\in \NN^{n_l}$ satisfying $|\beta_l|=\alpha_l$, and $\psi=\varphi\circ (\xi_1,\,\ldots,\,\xi_r)$ for some  $\xi_l\in E_2(n_l)$.
\end{enumerate}}
\end{remark}

The collection $\{V_{\varphi,\alpha}, \;\alpha\in\NN^r, r\in \NN_{\geq 1}, \varphi\in E_2(r)\}$ is a basis for the topology of $\sym(U)$.

\medskip

Let $d\in \NN_{\geq 1}$. In the rest of the section we write $b_\varphi$ instead of $a_{\varphi,\leq d}$ to lighten notation. Observe that if $\alpha\in(\NN_{\leq d})^r$ and $\varphi\in E_2(r)$, then $V_{\varphi,\alpha}=b_{\varphi,\alpha}(\sym^\alpha(U))\subseteq\sym_{\leq d}(U)$. Hence, $b_{\varphi,\alpha}\circ c_{\varphi,\alpha}$ is the natural embedding of $V_{\varphi,\alpha}$ in $\sym^{|\alpha|}_{\leq d}(U)$.
The collection  
\begin{equation*}\mathcal{O}_d:=\{V_{\varphi,\alpha},\; \alpha\in(\NN_{\leq d})^r,\; r\in \NN_{\geq 1},\; \varphi\in E_2(r)\}\end{equation*} is a basis for the topology of $\sym_{\leq d}(U)$. Indeed, if $x\in\sym^n_{\leq d}(U)$, then $x\in \sym_\lambda(U)$ for some $\lambda=(\lambda_1,\ldots,\lambda_r)\in{\mathrm P}(n)$ with $d(\lambda)\leq d$. Then we can choose a suitable embedding $\varphi\in E_2(r)$ such that $x\in b_{\varphi,\lambda}(\prod_{l=1}^r\sym^{\lambda_l}_{\neq}(U))\subset V_{\varphi,\lambda}\cap \sym_{\lambda}(U)$. In fact, the special open subsets in $\mathcal O_d$ generate a sieve $\mathcal C_d$ on $\sym_{\leq d}(U)$. 

\medskip

 We first extend objects in $\FP^{\leq d}$ to perverse sheaves  on $\sym_{\leq d}(U)$ constructing functorially $\mathcal C_d$-locally defined objects in $D^b(\sym_{\leq d}(U),\Sigma,\Vu)$ in the sense of \cite[\S 3.2.3]{BBD}.

\begin{lem}\label{lem:locally-defined}Let $d\in\NN_{\geq1}$ and let $(\mathcal{F},\mu)=((\mathcal F_n)_{n\in\NN_{\leq d}}, \mu)\in \FP^{\leq d}$. For $r\in \NN_{\geq 1}$, $\alpha\in(\NN_{\leq d})^r$  and $\varphi\in E_2(r)$,
let $\widetilde{\mathcal{F}}_{\varphi, \alpha}$ be the perverse sheaf on $V_{\varphi,\alpha}$ defined by
\begin{align}\label{eq:ext}
\widetilde{\mathcal{F}}_{\varphi, \alpha} := c_{\varphi, \alpha}^\ast\left( \boxtimes_{\varphi}(\mathcal{F}_{\alpha_1}, \ldots, \mathcal{F}_{\alpha_r})\right).
\end{align}
Then, the collection $$\{\widetilde{\mathcal{F}}_{\varphi, \alpha},\;r\in \NN_{\geq 1},\;\alpha\in(\NN_{\leq d})^r,\;\varphi\in E_2(r)\},$$ 
fits into a $\mathcal C_d$-locally defined object in $D^b(\sym_{\leq d}(U),\Sigma,\Vu)$ determining a perverse sheaf $\widetilde{\mathcal F}$ on $\sym_{\leq d}(U)$ and a collection of compatible isomorphisms 
\begin{align}\label{eq:effe}f_{\varphi,\alpha}\colon \widetilde{\FF}|_{V_{\varphi,\alpha}}\longrightarrow \widetilde{\FF}_{\varphi,\alpha},&&V_{\varphi,\alpha}\in\mathcal O_d\end{align} such that the diagram below commutes

 \begin{equation}\label{eq:effe-compa}
 \begin{tikzcd}
\widetilde\FF|_{V_{\varphi\circ(\xi_1,\ldots,\xi_r),\underline\beta}}
\arrow[rr, "f_{\varphi\circ(\xi_1,\ldots,\xi_r),\underline\beta}"] 
\arrow[dr, swap, "{c_{\varphi\circ(\xi_1,\ldots,\xi_r),\underline\beta}^*b_{\varphi\circ(\xi_1,\ldots,\xi_r),\underline\beta}^*(f_{\varphi,\alpha}})"]
&& \widetilde\FF_{\varphi\circ(\xi_1,\ldots,\xi_r),\underline\beta}\\
&\widetilde\FF_{\varphi,\alpha}|_{V_{\varphi\circ(\xi_1,\ldots,\xi_r),\underline\beta}}
\arrow[ur, swap, "{c_{\varphi\circ(\xi_1,\ldots,\xi_r),\underline\beta}^*(\boxtimes_{\varphi}(\mu^{-1}_{\xi_1,\beta_1},\ldots,\mu^{-1}_{\xi_r,\beta_r}))}"]
\end{tikzcd} 
        \end{equation}
for any refinement $\underline\beta=(\beta_1,\,\ldots,\,\beta_r)$ of $\alpha$, with $\beta_l\in(\NN_{\leq d})^{n_l}$, $n_l\in\NN$, $|\beta_l|=\alpha_l$ and $\xi_l\in E_2(n_l)$ for $l=1,\,\ldots,\,r$.
\end{lem}
\begin{proof}
First of all we show that if $V_{\varphi,\alpha}=V_{\varphi',\alpha'}$ then 
$\widetilde{\mathcal{F}}_{\varphi, \alpha}=\widetilde{\mathcal{F}}_{\varphi', \alpha'}$. By Remark \ref{rem:refinement} (1) there is $\sigma\in \mathbb S_r$ such that $\alpha'=\,^\sigma\alpha=(\alpha_{\sigma^{-1}(1)},\ldots,\alpha_{\sigma^{-1}(r)})$ and $\varphi'=\varphi^\sigma$.  
Applying \eqref{eq:equivariance-a-phi} and \eqref{eq:boxtimes-sigma} yields the following chain of equalities of sheaves on $V_{\varphi,\alpha}$:
\begin{align*}
    \widetilde{\mathcal{F}}_{\varphi^\sigma, ^\sigma\!\alpha} &=c^*_{\varphi^\sigma, ^\sigma\!\alpha}\left(\boxtimes_{\varphi^\sigma}(\mathcal{F}_{\alpha_{\sigma^{-1}(1)}}, \ldots, \mathcal{F}_{\alpha_{\sigma^{-1}(r)}})\right)\\
    &= c_{\varphi,\alpha}^* (\sigma^{-1})^*\left(\sigma^*\left(\boxtimes_{\varphi}\sigma(\mathcal{F}_{\alpha_{\sigma^{-1}(1)}}, \ldots, \mathcal{F}_{\alpha_{\sigma^{-1}(r)}})\right)\right)= c_{\varphi,\alpha}^*\left(\boxtimes_{\varphi}(\mathcal{F}_{\alpha_1}, \ldots, \mathcal{F}_{\alpha_r})\right) = \widetilde{\mathcal{F}}_{\varphi, \alpha}.
    \end{align*}
    We now show that for any inclusion $V_{\psi,\beta}\subset V_{\varphi,\alpha}$
with $\beta$ and $\psi$  as in Remark \ref{rem:refinement} (2), 
the factorization datum $\mu$ induces isomorphisms $ \widetilde  {\mathcal F}_{\varphi,\alpha}|_{V_{\psi,\beta}}\simeq \widetilde{\mathcal F}_{\psi,\beta}$.
Applying Lemma \ref{lem:b-passa} we obtain the sequence of equalities
\begin{align*}
\widetilde{\mathcal F}_{\varphi,\alpha}|_{V_{\psi,\beta}}&=(c_{\psi,\beta}\circ b_{\psi,\beta})^* \widetilde{\mathcal F}_{\varphi,\alpha}=c_{\psi,\beta}^*(\prod_{l=1}^rb_{\xi_l,\beta_l})^*b_{\varphi,\alpha}^*c_{\varphi,\alpha}^*\left(\boxtimes_\varphi(\mathcal{F}_{\alpha_1}, \ldots, \mathcal{F}_{\alpha_r})\right)\\
 &=c_{\psi,\beta}^*(\prod_{l=1}^rb_{\xi_l,\beta_l})^*\left(\boxtimes_\varphi(\mathcal{F}_{\alpha_1}, \ldots, \mathcal{F}_{\alpha_r})\right)=c_{\psi,\beta}^*\left(\boxtimes_\varphi(b_{\xi_1,\beta_1}^*\mathcal F_{\alpha_1},\,\ldots,b_{\xi_r,\beta_r}^*\mathcal F_{\alpha_r})\right).
 \end{align*}
We observe that if $d(\alpha)\leq d$, then $|\beta_l|\leq d$ for any $l=1,\ldots,r$ so  
$c_{\psi,\beta}^*(\boxtimes_{\varphi}(\mu_{\xi_l,\beta_l}^{-1})_{l=1}^r)$ is a well-defined isomorphism, giving 
\begin{align*}
 \widetilde{\mathcal F}_{\varphi,\alpha}|_{V_{\psi,\beta}}&\stackrel{\sim}{\longrightarrow} c_{\psi,\beta}^*\left(\boxtimes_\varphi(\boxtimes_{\xi_1}(\mathcal F_{\beta_1},\,\ldots,\mathcal F_{\beta_{j_1}}),\ldots,\boxtimes_{\xi_r}(\mathcal F_{\beta_{r'}-j_r+1},\ldots,\,\mathcal F_{\beta_{r'}}))\right)=\widetilde{\mathcal F}_{\psi,\beta}\end{align*}
where the last equality follows from compatibility of $\boxtimes$ with operadic composition. Compatibility of these isomorphisms with compositions of inclusions (that is, the cocycle condition) follows from \eqref{eq:composition-a} and \eqref{eq:factdat_a_partition}. By \cite[Theorem~3.2.4, Corollaire~2.1.23]{BBD}, we have a uniquely determined perverse sheaf $\widetilde{\mathcal{F}}$ on $\sym_{\leq d}(U)$, with a family of identifications $f_{\varphi,\alpha}\colon \widetilde\FF|_{V_{\varphi,\alpha}}\to \widetilde\FF_{\varphi,\alpha}$. Equation \eqref{eq:effe-compa} spells out the 
compatibility of $f_{\varphi,\alpha}$ with restrictions.  
\end{proof}

Let $\FF$ and $\widetilde\FF$ be as in Lemma \ref{lem:locally-defined}.
If $\varphi=\id_U$ and $\alpha=(m)$ for some $m\in\NN_{\leq d}$ then $V_{\varphi,\alpha}=\sym^{m}(U)$ and 
\eqref{eq:effe} gives a natural identification $f_{\id,(m)}\colon\widetilde{\FF}_{m}\to \widetilde{\FF}_{\id,(m)}=\FF_{m}$. 

\medskip

%

For any $r\in \NN$  and any $\varphi\in E_2(r)$ we have an isomorphism of perverse sheaves on $\sym^{\leq d}(U)^r$, defined on the connected component $\sym^\alpha(U)$ for 
 $\alpha\in(\NN_{\leq d})^r$ as follows
\begin{align}\label{eq:extension-mu-piccolo}
        \widetilde{\mu}_{[\varphi, \alpha]}\colon 
    \boxtimes_{\varphi}(\widetilde{\mathcal{F}}_{\alpha_1}, \ldots, \widetilde{\mathcal{F}}_{\alpha_r}) &\longrightarrow   \boxtimes_{\varphi}(\mathcal{F}_{\alpha_1}, \ldots, \mathcal{F}_{\alpha_r})= b_{\varphi,\alpha}^*(\widetilde{\FF}_{\varphi,\alpha})
   \longrightarrow b_{\varphi, \alpha}^\ast (\widetilde{\mathcal{F}}|_{V_{\varphi,\alpha}})=b_{\varphi, \alpha}^\ast \widetilde{\mathcal{F}}_{|\alpha|}\\
  \nonumber  \widetilde{\mu}_{[\varphi, \alpha]}&:=b_{\varphi,\alpha}^*(f_{\varphi,\alpha}^{-1})\circ\boxtimes_{\varphi}(f_{\id,(\alpha_1)},\ldots,f_{\id,(\alpha_r)}). 
       \end{align}

%
 
\begin{lem}\label{lem:mu-tilde-piccolo-mon-bra}
Let $\varphi\in E_2(1)$, and let $\gamma$ be a path in $E_2(1)$ from $\varphi$ to $\id_U$.  Let $m\leq d$ and let $\widetilde{\mu}_{[\varphi, (m)]}$ be in \eqref{eq:extension-mu-piccolo}. Then, the diagram below is commutative 
  \begin{equation}\label{eq:mon-bra-ext-piccolo}
            \begin{tikzcd}
\boxtimes_{\varphi} \widetilde\FF 
\arrow[rr, "\widetilde\mu_{[\varphi,(m)]}"] 
\arrow[dr, swap, "\underline{R}_{\gamma,\widetilde\FF}"]& 
& b_\varphi^\ast \mathcal{F}
\arrow[dl, "M_{\gamma,\widetilde\FF}"]\\
&\widetilde\FF
\end{tikzcd} 
        \end{equation}
\end{lem}
\begin{proof}
By definition of $\widetilde\mu_{[\varphi,(m)]}$ we need to prove commutativity of the external square of the diagram
 \begin{equation}
            \begin{tikzcd}
\boxtimes_{\varphi} \widetilde\FF 
\arrow[rr, "\boxtimes_{\varphi}f_{\id,(m)}"] 
\arrow[d, swap, "\underline{R}_{\gamma,\widetilde\FF_m}"]& 
&\boxtimes_{\varphi}\FF_m=b_{\varphi,(m)}^*\widetilde\FF_{\varphi,(m)}\arrow[rr,"b_{\varphi,(m)}^*(f_{\varphi,(m)}^{-1})"] \arrow[d, swap, "\underline{R}_{\gamma,\FF_m}"]&&b_{\varphi,(m)}^*\widetilde\FF_m\arrow[d,"M_{\gamma,\widetilde\FF_m}"]\\
\widetilde\FF_m\arrow[rr,"f_{\id,(m)}"]&&\FF_m\arrow[rr,"f^{-1}_{\id,(m)}"]&&\widetilde\FF_m
\end{tikzcd} 
        \end{equation}
The left square commutes by naturality of $\underline R$.  The right square fits into the diagram
\begin{equation}
            \begin{tikzcd}
\boxtimes_{\varphi}\FF_m=b_{\varphi,(m)}^*\widetilde\FF_{\varphi,(m)}\arrow[rrrr,"b_{\varphi,(m)}^*(f_{\varphi,(m)}^{-1})"] \arrow[dd, swap, "\underline{R}_{\gamma,\FF_m}"]\arrow[drr,"\mu_{\varphi,(m)}"]&&&&b_{\varphi,(m)}^*\widetilde\FF_m\arrow[dd,"M_{\gamma,\widetilde\FF_m}"]\\
&&b_{\varphi,(m)}^*\FF_m\arrow[urr,"b_{\varphi,(m)}^*f_{\id,(m)}^{-1}"]\arrow[dll,"M_{\gamma,\FF_m}"]\\
\FF_m\arrow[rrrr,swap,"f^{-1}_{\id,(m)}"]&&&&\widetilde\FF_m
\end{tikzcd} 
        \end{equation}
   The left triangle commutes in virtue of \eqref{eq:mon-bra} for $\mu$, the top  triangle commutes because of \eqref{eq:effe-compa}, and the bottom right diagram commutes by naturality of $M$.  
\end{proof}

    We  aim at extending the isomorphisms \eqref{eq:extension-mu-piccolo} from $\sym^{\leq d}(U)^r$ to $\sym_{\leq d}(U)^r$ for $r\in\NN$. We first define isomorphisms on products of special open subsets $V_{\underline{\psi},\underline\beta}:=V_{\psi_1,\beta_1}\times \cdots\times V_{\psi_r,\beta_r}$ where $V_{\psi_l,\beta_l}\in \mathcal O_d$ for any $l=1,\ldots, r$. 

\medskip
     
     Let $\alpha=(\alpha_1,\,\ldots,\,\alpha_r)\in\NN^r$ and $\varphi\in E_2(r)$, and let $\underline\beta:=(\beta_1,\,\ldots,\,\beta_r)$ be a refinement of $\alpha$, with $\beta_l\in(\NN_{\leq d})^{m_l}$  and $\alpha_l:=|\beta_l|$ for $l=1,\ldots,r$. We define an isomorphism of perverse sheaves on $V_{\underline\psi,\underline\beta}$ by setting
          \begin{align}
      \nonumber  (\widetilde{\mu}_{[\varphi, \alpha]})_{V_{\underline\psi,\underline\beta}}&\colon 
    \left(\boxtimes_{\varphi}(\widetilde{\mathcal{F}}_{\alpha_1}, \ldots, \widetilde{\mathcal{F}}_{\alpha_r})\right)|_{V_{\underline\psi,\underline\beta}} \longrightarrow  \left(b^*_{\varphi,\alpha}(\widetilde\FF_{|\alpha|})\right)|_{V_{\underline\psi,\underline\beta}}\\
   \nonumber (\widetilde{\mu}_{[\varphi, \alpha]})_{V_{\underline\psi,\underline\beta}}&:=(\prod_{l=1}^r c_{\psi_l,\beta_l})^*\left(\widetilde{\mu}_{[\varphi\circ(\psi_1,\,\ldots,\,\psi_r),\underline\beta]}\circ (\boxtimes_{\varphi\circ(\psi_1,\,\ldots,\,\psi_r)}((f^{-1}_{\id,\beta_{lp}})_{l=1}^r)_{p=1}^{m_l})\right)\circ\left( \boxtimes_{\varphi}(f_{\psi_l,\beta_l})_{l=1}^r\right)\\
   \label{eq:extension-mu-special-open} &=(\prod_{l=1}^r c_{\psi_l,\beta_l})^*\left(b_{\varphi\circ(\psi_1,\,\ldots,\,\psi_r),\underline\beta}^*(f_{\varphi\circ(\psi_1,\,\ldots,\,\psi_r),\beta}^{-1})\right)\circ\left( \boxtimes_{\varphi}(f_{\psi_l,\beta_l})_{l=1}^r\right).
    \end{align}

    We point out that if $\psi_l=\id_U$ for $l=1,\ldots,\,r$ and $\alpha=\beta$ then we recover \eqref{eq:extension-mu-piccolo}, so there is no ambiguity in the notation.  
In addition, there holds
\begin{align}\label{eq:useful}
\nonumber  (\prod_{l=1}^r b_{\psi_l\beta_l})^*  (\widetilde{\mu}_{[\varphi, \alpha]})_{V_{\underline\psi,\underline\beta}}&=\widetilde{\mu}_{[\varphi\circ(\psi_1,\,\ldots,\,\psi_r),\underline\beta]}\circ\left(\boxtimes_{\varphi}\left(\boxtimes_{\psi_l}(f^{-1}_{\id,\beta_{lp}})_{p=1}^{m_l}\circ b_{\psi_l\beta_l}^*(f_{\psi_l,\beta_l})\right)_{l=1}^r\right)\\
  &=\widetilde{\mu}_{[\varphi\circ(\psi_1,\,\ldots,\,\psi_r),\underline\beta]}\circ \boxtimes_{\varphi}(\widetilde\mu^{-1}_{[\psi_l,(\beta_l)]})_{l=1}^r
\end{align}

\begin{lem}\label{lem:ext-mu-morphism}Let $d\in\NN_{\geq 1}$. Let $(\mathcal{F},\mu)=((\mathcal F_n)_{n\in\NN_{\leq d}}, \mu)\in \FP^{\leq d}$ and let $\widetilde{\mathcal F}$ and $f_{\varphi,\alpha}$ for $\varphi\in E_2(r)$, $r\geq 1$, $\alpha\in (\NN_{\leq d})^r$ be respectively the perverse sheaf on $\sym_{\leq d}(U)$ and  the collection of isomorphisms as in Lemma \ref{lem:locally-defined}.
Then, there is a uniquely determined family of  isomorphisms of perverse sheaves on $\sym_{\leq d}(U)^r$ 
\begin{align*}\widetilde{\mu}_{\varphi,\alpha}\colon \widetilde{\FF}^{\boxtimes_{\varphi} r}\to b_{\varphi,\alpha}^*(\widetilde{\FF}),&& r\in\NN_{\geq1},\;\varphi\in E_2(r),\;\alpha\in\NN^r\end{align*} 
satisfying $\widetilde{\mu}_{\varphi,\alpha}|_{V_{\underline\psi,\underline\beta}}=(\widetilde{\mu}_{[\varphi,\alpha]})_{V_{\underline\psi,\underline\beta}}$ for any $r\in\NN$, $\varphi\in E_2(r)$, $\alpha\in \NN^r$,  $\psi_l\in E_{2}(m_l)$, 
$\beta_l\in(\NN_{\leq d})^{m_l}$, $m_l\in\NN$ and $l=1,\ldots,r$.
\end{lem}
\begin{proof}
First of all, we show that $(\widetilde{\mu}_{\varphi, \alpha})_{V_{\underline\psi,\underline\beta}}$  depends on 
 $V_{\underline\psi,\underline\beta}$ only. Let $\sigma_l\in \mathbb S_{m_l}$ for $l=1,\,\ldots,r$. Then
 $V_{\underline\psi,\underline\beta}=\prod_{l=1}^rV_{\psi_l,\beta_l}$ and 
 $V_{\underline\psi^{\sigma},\,^{\sigma}\!\underline\beta}:=\prod_{l=1}^rV_{\psi_l^{\sigma_l},\,^{\sigma_l}\!\beta_l}$ coincide and 
\begin{align*}(\widetilde{\mu}_{[\varphi, \alpha]})_{V_{\underline\psi^{\sigma},\,^{\sigma}\!\underline\beta}}&=(\prod_{l=1}^r c_{\psi_l^{\sigma_l},\,^{\sigma_l}\!\beta_l})^*\left(b_{\varphi\circ(\psi_1^{\sigma_1},\,\ldots,\,\psi_r^{\sigma_r}),^{\sigma}\beta}^*(f_{\varphi\circ(\psi_1^{\sigma_1},\,\ldots,\,\psi_r^{\sigma_r}),\,^{\sigma}\beta}^{-1})\right)\circ\left( \boxtimes_{\varphi}(f_{\psi^{\sigma_l}_l,^{\sigma_l}\!\beta_l})_{l=1}^r\right).
\end{align*} 
Lemma \ref{lem:locally-defined} gives 
\begin{align*}&\widetilde\FF_{\psi^{\sigma_l},^{\sigma_l}\!\beta_l}=\widetilde\FF_{\psi_l,\beta_l},&&f_{\psi_l^{\sigma_l},^{\sigma_l}\!\beta_l}=f_{\psi_l,\beta_l},\\
&f_{\id,^{\sigma_l}\!\beta_l}=f_{\id,\beta_l},&&f_{\varphi\circ(\psi_1^{\sigma_1},\ldots,\psi_r^{\sigma_r}),(^{\sigma_1}\!\beta_1,\ldots, ^{\sigma_r}\!\beta_r)}=f_{\varphi\circ(\psi_1,\ldots,\psi_r),(\beta_1,\ldots, \beta_r)}.
\end{align*}
Invoking Lemma \ref{lem:tensoreq}, \eqref{eq:a-phi*-equivariance}, \eqref{eq:composition-a}, and \eqref{eq:extension-mu-special-open}  we get
\begin{align*}&(\widetilde{\mu}_{[\varphi, \alpha]})_{V_{\underline\psi^{\sigma},\,^{\sigma}\!\underline\beta}}=\\
&=(\prod_{l=1}^rc_{\psi_l,\beta_l}^*(\sigma_l^{-1})^*)\left((\prod_{l=1}^r\sigma_l^*b_{\psi_l,\beta_l}^*)b_{\varphi,\alpha}^*\left(f^{-1}_{\varphi\circ(\psi_1,\ldots,\psi_r),(\beta_1,\,\ldots,\,\beta_r)}\right)\right)\circ\boxtimes_{\varphi}(f_{\psi_l,\beta_l})_{l=1}^r\\
&=(\prod_{l=1}^rc_{\psi_l,\beta_l}^*)\left((\prod_{l=1}^rb_{\psi_l,\beta_l}^*)b_{\varphi,\alpha}^*\left(f^{-1}_{\varphi\circ(\psi_1,\ldots,\psi_r),(\beta_1,\,\ldots,\,\beta_r)}\right)\right)\circ\boxtimes_{\varphi}(f_{\psi_l,\beta_l})_{l=1}^r=(\widetilde{\mu}_{[\varphi, \alpha]})_{V_{\underline\psi,\,\underline\beta}}.
\end{align*}

We now show that the data are compatible with restriction to products of special open subsets, that is, we show that for  $V_{\psi'_l,\beta_l'}\subseteq V_{\psi_l,\beta_l}$ for $l=1,\,\ldots, r$ there holds
\begin{equation}\label{eq:glue-mu}(\widetilde{\mu}_{[\varphi, \alpha]})_{V_{\underline\psi,\underline\beta}}|_{V_{\underline\psi',\underline\beta'}}=(\widetilde{\mu}_{[\varphi, \alpha]})_{V_{\underline\psi',\underline\beta'}}.\end{equation}

It is enough to prove the statement in the case in which  $\psi'_l=\psi_l$ and $\beta_l'=\beta_l$  for all but one index $l$. Without loss of generality we  assume that $\psi'_l=\psi_l$ and $\beta_l'=\beta_l$ for $l=1,\ldots, r-1$ and that $\beta'_r=(\delta_1,\,\ldots,\,\delta_{m_r})$ is a refinement of $\beta_r$, with $\delta_p\in(\NN_{\leq d})^{n_p}$ for $p=1,\ldots,m_r$ and that $\psi_r'=\psi_r\circ(\xi_1,\,\ldots,\,\xi_{m_r})$ with $\xi_p\in E_2(n_p)$. Now, using \eqref{eq:extension-mu-special-open} we have
\begin{align*}
  (\widetilde{\mu}_{[\varphi, \alpha]})_{V_{\underline\psi,\underline\beta}}|_{V_{\underline\psi',\underline\beta'}}=(\prod_{l=1}^rc_{\psi'_l,\beta_l'})^*b_{\varphi\circ(\psi_1',\,\ldots,\psi'_r),\beta'}^*(f^{-1}_{\varphi\circ (\psi_1,\,\ldots,\psi_r),\beta})\circ \left(\boxtimes_{\varphi}(c_{\psi'_l,\beta_l'}^*b_{\psi'_l,\beta_l'}^*f_{\psi_l,\beta_l})_{l=1}^r\right)\end{align*}
 and, applying \eqref{eq:effe-compa}
 \begin{align*}(\widetilde{\mu}_{[\varphi, \alpha]})_{V_{\underline\psi',\underline\beta'}}
   &=(\prod_{l=1}^rc_{\psi'_l,\beta_l'})^*b_{\varphi\circ(\psi_1',\,\ldots,\psi'_r),\beta'}^*(f^{-1}_{(\varphi\circ (\psi'_1,\,\ldots,\psi'_r))\circ(\id^{r-1},\xi_1,\,\ldots,\xi_{m_r}),\beta'})\circ\boxtimes_{\varphi}( f_{\psi'_l,\beta'_l})_{l=1}^r\\
   &=(\prod_{l=1}^rc_{\psi'_l,\beta_l'})^*b_{\varphi\circ(\psi_1',\,\ldots,\psi'_r),\beta'}^*(f^{-1}_{\varphi\circ (\psi_1,\,\ldots,\psi_r),\beta})\\
   &\circ (\prod_{l=1}^rc_{\psi'_l,\beta_l'})^*\left(\boxtimes_{\varphi\circ(\psi_1,\,\ldots,\,\psi_r)}(\id^{r-1},\mu_{\xi_1,\delta_1}\,\ldots,\mu_{\xi_{m_r},\delta_{m_r}})\right)\circ \boxtimes_{\varphi}( f_{\psi'_l,\beta'_l})_{l=1}^r\\
   &=(\prod_{l=1}^rc_{\psi'_l,\beta_l'})^*b_{\varphi\circ(\psi_1',\,\ldots,\psi'_r),\beta'}^*(f^{-1}_{\varphi\circ (\psi_1,\,\ldots,\psi_r),\beta})\\
   &\circ \boxtimes_{\varphi}
   (c_{\psi_p,\beta_p}^*(\id^{\boxtimes_{\psi_p}m_p})\circ f_{\psi_p,\beta_p})_{p=1}^{r-1},c_{\psi'_r,\beta'_r}^*(\boxtimes_{\psi_r'}(\mu_{\xi_l,\delta_l})_{l=1}^{m_r})f_{\psi_r\circ(\xi_1,\,\ldots,\,\xi_{m_r}),\beta_r'}).
   \end{align*}
Therefore \eqref{eq:glue-mu}  holds if and only if 
\begin{equation}\label{eq:requirement}
c_{\psi'_r,\beta'_r}^*b_{\psi'_r,\beta'_r}^*f_{\psi_r,\beta_r}=c_{\psi'_r,\beta'_r}^*(\boxtimes_{\psi_r'}(\mu_{\xi_l,\delta_l})_{l=1}^{m_r})f_{\psi_r\circ(\xi_1,\,\ldots,\,\xi_{m_r}),\beta_r'}.
\end{equation}
which is immediate from \eqref{eq:effe-compa}.
\end{proof}

 \begin{lem}\label{lem:ext-mu-FD}Let $d\in\NN_{\geq 1}$. Let $(\mathcal{F},\mu)=((\mathcal F_n)_{n\in\NN_{\leq d}}, \mu)\in \FP^{\leq d}$ and let 
 \begin{equation*}\widetilde{\mathcal F},\quad
(f_{\varphi,\alpha})_{\varphi\in E_2(r),\,r\in \NN_{\geq1},\, \alpha\in (\NN_{\leq d})^r}\end{equation*} be respectively the perverse sheaf on $\sym_{\leq d}(U)$ and  the collection of isomorphisms as in Lemma \ref{lem:locally-defined}.
Then, the collection of isomorphisms 
\begin{equation*}
(\widetilde{\mu}_{\varphi,\alpha})_{\varphi\in E_2(r),\,r\in\NN,\alpha\in \NN^r}\end{equation*} from Lemma \ref{lem:ext-mu-morphism}  is a factorization datum for $\widetilde{\mathcal F}$, that is, $(\widetilde\FF,\widetilde\mu)$ is an object in $\FP_{\leq d}$.
\end{lem}
\begin{proof}
We first prove \eqref{eq:factdat_a_partition} on products of special open subsets. Let $r, m_l\in\NN$, $\alpha\in\NN^r$, $\varphi\in E_2(r)$, $\underline{\psi}=(\psi_1,\,\ldots,\,\psi_r)$ with $\psi_l\in E_2(m_l)$, $m=\sum_{l=1}^rm_l$, and $\beta_l=(\beta_{l1},\ldots,\beta_{lm_l})\in(\NN_{\leq d})^{m_l}$ with $|\beta_l|=\alpha_l$ for $l=1,\ldots,r$. Then
$\widetilde{\mu}_{\varphi\circ\underline{\psi},\underline{\beta}}$ is an isomorphism of sheaves on $\sym_{\leq d}^{\underline{\beta}}(U)=\prod_{l=1}^r\sym_{\leq d}^{\beta_{l}}(U)$. For $l=1,\,\ldots,\,r$, let $V_{\underline\xi_l,\underline\delta_l}\subset \sym^{\beta_l}(U)$ be a product of special open subsets in $\sym_{\leq d}^{\beta_{l}}(U)$. Here,  $\underline\xi_{l}=(\xi_{l1},\,\ldots,\,\xi_{lm_l})\in\prod_{p=1}^{m_l}E_2(n_{lp})$ for some $n_{lp}\in\NN$, and $\underline\delta_l=(\delta_{l1},\,\ldots,\,\delta_{lm_l})$ with $\underline\delta_{lp}\in(\NN_{\leq d})^{n_{lp}}$ and $|\underline\delta_{lp}|=\beta_{lp}\leq d$.  We set $\underline{\xi}:=(\underline\xi_1,\,\ldots,\,\underline\xi_r)$ and $\underline\delta:=(\underline\delta_1,\,\ldots,\,\underline\delta_r)$ and verify that
\begin{align*}
\widetilde{\mu}_{\varphi\circ\underline{\psi},\alpha}|_{V_{\underline{\xi},\underline\delta}}=
((\Pi_{l=1}^r b_{\psi_l,\beta_l})^\ast(\widetilde\mu_{\varphi,\alpha}))|_{V_{\underline{\xi},\underline\delta}}\circ
\left( \boxtimes_\varphi(\widetilde\mu_{\psi_l,\beta_l})_{l=1}^r\right)|_{V_{\underline{\xi},\underline\delta}}. 
\end{align*}

 Applying \eqref{eq:extension-mu-special-open} gives 
\begin{align*}
 \widetilde{\mu}_{\varphi\circ\underline{\psi},\alpha}|_{V_{\underline{\xi},\underline\delta}}&=
 (\prod_{l,p}c_{\xi_{lp},\delta_{lp}})^*(b_{(\varphi\circ\underline\psi)\circ\underline\xi,\underline\delta}^*(f_{(\varphi\circ\underline\psi)\circ\underline\xi,\underline\delta}^{-1}))\circ\boxtimes_{\varphi\circ\underline\psi}((f_{\xi_{lp,\delta_{lp}}})_{l=1}^r))_{p=1}^{m_l};
\end{align*}

\begin{align*}
((\prod_{l=1}^r b_{\psi_l,\beta_l})^\ast(\widetilde\mu_{\varphi,\alpha}))|_{V_{\underline{\xi},\underline\delta}}&=(\prod_{l,p}c_{\xi_{lp},\delta_{lp}})^*(\prod_{l,p}b_{\xi_{lp},\delta_{lp}})^*(\prod_{l=1}^r b_{\psi_l,\beta_l})^\ast(\widetilde\mu_{\varphi,\alpha}))\\
&=(\prod_{l=1}^r b_{\psi_l,\beta_l})^\ast(\widetilde\mu_{\varphi,\alpha}|_{V_{\psi_1\circ\underline\xi_1,\underline\delta_1}\times\ldots\times V_{\psi_r\circ\underline\xi_r,,\underline\delta_r}})\\
&=(\prod_{l=1}^r b_{\psi_l,\beta_l})^\ast\left(\left((\prod_{l=1}^rc_{\psi_l\circ\underline\xi_l,\delta_l})^*b_{(\varphi\circ\underline\psi)\circ\underline\xi,\underline\delta}^*f_{(\varphi\circ\underline\psi)\circ\underline\xi,\underline\delta}^{-1}\right)
\circ\boxtimes_{\varphi}(f_{\psi_l\circ\underline\xi_l,\underline\delta_l})_{l=1}^r\right)\\
&=(\prod_{l,p}c_{\xi_{lp},\delta_{lp}})^*(b_{(\varphi\circ\underline\psi)\circ\underline\xi,\underline\delta}^*(f_{(\varphi\circ\underline\psi)\circ\underline\xi,\underline\delta}^{-1}))\circ
(\prod_{l=1}^r b_{\psi_l,\beta_l})^\ast\left(\boxtimes_{\varphi}(f_{\psi_l\circ\underline\xi_l,\underline\delta_l})_{l=1}^r\right);
\end{align*}
and
\begin{align*}
\left( \boxtimes_\varphi(\widetilde\mu_{\psi_l,\beta_l})_{l=1}^r\right)|_{V_{\underline{\xi},\underline\delta}}&=(\prod_{l,p}c_{\xi_{lp},\delta_{lp}})^*(\prod_{l,p}b_{\xi_{lp},\delta_{lp}})^*\left( \boxtimes_\varphi(\widetilde\mu_{\psi_l,\beta_l})_{l=1}^r\right)\\
&=\boxtimes_{\varphi}\left((\prod_{p=1}^{m_l}c_{\xi_{lp},\delta_{lp}})^*(\prod_{p=1}^{m_l}b_{\xi_{lp},\delta_{lp}})^*(\widetilde\mu_{\psi_l,\beta_l})\right)_{l=1}^r\\
&=\boxtimes_{\varphi}\left(\widetilde\mu_{{\psi_l,\beta_l}}|_{V_{\underline\xi_l,\underline\delta_l}}\right)_{l=1}^r\\
&=\boxtimes_{\varphi}\left((\prod_{p=1}^{m_l}c_{\xi_{lp},\delta_{lp}})^*b_{\psi_l\circ\underline\xi_l,,\underline\delta_l}^*(f_{\psi_l\circ\underline\xi_l,\underline\delta_l}^{-1})\circ (\boxtimes_{\psi_l}(f_{\xi_{lp},\delta_{lp}})_{p=1}^{m_l})\right)_{l=1}^r
\end{align*}
and the desired equality follows from comparison of the terms. 

\medskip

We verify braiding-monodromy compatibility \eqref{eq:mon-bra} for $\widetilde{\mu}$. First of all, we consider the restriction of $\widetilde{\mu}$ to $\coprod_{n\in\NN} \sym^{\leq d}(U)^n$, where \eqref{eq:extension-mu-piccolo} is in force. Then \eqref{eq:mon-bra} holds by Lemma \ref{lem:mu-tilde-piccolo-mon-bra}. Observe that, since $\widetilde\mu$ is compatible with operadic composition, the proof of Proposition \ref{prop:mon-bra-basta1} carries through. Therefore, for every path $\gamma$ in $E_2(r)$ joining $\eta$ to $\eta'$, and any $\beta\in(\NN_{\leq d})^r$ the diagram below of  isomorphism of sheaves on $\sym^\beta(U)$ is commutative
        \begin{equation}\label{eq:piccolo-mon-bra-esteso}
            \begin{tikzcd}[row sep=huge]
 \widetilde\FF^{\boxtimes_{\eta} r} 
\arrow[rr, "\widetilde\mu_{\eta,\beta}"] 
\arrow[d, swap, "\underline{R}_{\gamma,\widetilde\FF^r}"]& 
& b_{\eta,\beta}^\ast \widetilde\FF
\arrow[d, "M_{\gamma,\widetilde\FF}"]\\
\widetilde\FF^{\boxtimes_{\eta'} r}
\arrow[rr, "\mu_{\eta',\beta}"]  & 
& b_{\eta',\beta}^\ast \widetilde\FF
\end{tikzcd} 
        \end{equation}
Let now $\varphi\in E_2(1)$. We verify \eqref{eq:mon-bra} on $V_{\psi,\beta}$ for 
$V_{\psi,\beta}\in\mathcal O_d$, and $|\beta|=m$. Applying \eqref{eq:useful} and $b_{\psi,\beta}^*$ this boils down to verifying commutativity of the diagram below, on $\sym^\beta(U)\subset\sym^{\leq d}(U)^r$.
\begin{equation}
            \begin{tikzcd}
\boxtimes_{\varphi} b_{\psi,\beta}^*(\widetilde\FF_m) 
\arrow[d, swap, "\underline{R}_{\gamma,b_{\psi,\beta}^*(\widetilde\FF_m)}"]& 
&\arrow[ll, swap, "\boxtimes_{\varphi}\widetilde{\mu}_{\psi,\beta}"] \boxtimes_{\varphi}\boxtimes_\psi(\widetilde\FF_{\beta_1},\ldots,\widetilde\FF_{\beta_r})\arrow[rr,"\widetilde{\mu}_{\varphi\circ\psi,\beta}"]\arrow[d, swap, "\underline{R}_{\gamma,\boxtimes_\psi(\widetilde\FF_{\beta_1},\ldots,\widetilde\FF_{\beta_r})}"] 
&&b_{\varphi\circ\psi,\beta}^*\widetilde\FF_m
\arrow[d,"b_{\psi,\beta}^*(M_{\gamma,\widetilde\FF_m})"]
\\
b_{\psi,\beta}^*(\widetilde\FF_m)&&\boxtimes_\psi(\widetilde\FF_{\beta_1},\ldots,\widetilde\FF_{\beta_r})\arrow[ll,"\widetilde\mu_{\psi,\beta}"]\arrow[rr,"\widetilde\mu_{\psi,\beta}"]&&b_{\psi,\beta}^*(\widetilde\FF_m)
\end{tikzcd} 
        \end{equation}
The left hand square commutes by naturality of $\underline R$. The right hand square commutes in virtue of \eqref{eq:piccolo-mon-bra-esteso} because 
$\underline{R}_{\gamma,\boxtimes_\psi(\widetilde\FF_{\beta_1},\ldots,\widetilde\FF_{\beta_r})}=\underline{R}_{\gamma\circ\psi,(\widetilde\FF_{\beta_1},\ldots,\widetilde\FF_{\beta_r})}$ and $b_{\psi,\beta}^*(M_{\gamma,\widetilde\FF_m})=M_{\gamma\circ\psi, \widetilde\FF_m}$.

\medskip

Finally, we prove symmetric equivariance. By Lemma \ref{lem:locally-defined}, for any $\sigma\in\mathbb S_r$, any $\alpha\in\NN^r$ and any $\varphi\in E_2(r)$  we have
$\widetilde\FF_{\varphi^\sigma,^{\sigma}\!\alpha}=\widetilde\FF_{\varphi,\alpha}$, and $f_{\varphi^{\sigma},^{\sigma}\!\alpha}=f_{\varphi,\alpha}$, therefore \eqref{eq:equivariance} 
follows from \eqref{eq:extension-mu-special-open}, \eqref{eq:a-phi*-equivariance} and \eqref{eq:boxtimes-sigma}.
\end{proof}

\begin{lem}\label{lem:ext-mor}
Let $h=(h_n)_{n\leq d}\colon (\FF,\mu)\to (\GG,\nu)$ be a morphism in $\FP^{\leq d}$, let $(\widetilde\FF,\widetilde\mu)$ and $(\widetilde\GG,\widetilde\nu)$ be obtained from $(\FF,\mu)$ and $(\GG,\nu)$  as in Lemma \ref{lem:locally-defined} and \eqref{eq:extension-mu-special-open}. 
Let 
\begin{align*}f_{\varphi,\alpha}\colon \widetilde\FF|_{V_{\varphi,\alpha}}\to \widetilde\FF_{\varphi,\alpha},&&
g_{\varphi,\alpha}\colon \widetilde\GG|_{V_{\varphi,\alpha}}\to \widetilde\GG_{\varphi,\alpha}, &&\mbox{ for }V_{\varphi,\alpha}\in\mathcal O_d
\end{align*}
be the corresponding natural identifications. Then, the collection of morphisms
\begin{align*}
\widetilde{h}_{\varphi,\alpha}:=g_{\varphi,\alpha}^{-1}\circ c_{\varphi,\alpha}^*\left(\boxtimes_{\varphi}(h_{\alpha_1},\,\ldots,\,h_{\alpha_r})\right)\circ f_{\varphi,\alpha}\colon \widetilde\FF|_{V_{\varphi,\alpha}}\to \widetilde\GG|_{V_{\varphi,\alpha}}
\end{align*}
for $V_{\varphi,\alpha}\in \mathcal O_d$ is well-defined and can be glued to a morphism 
$\widetilde{h}\colon (\widetilde\FF,\widetilde\mu)\to(\widetilde\GG,\widetilde\nu)$ in $\FP_{\leq d}$.
\end{lem}
\begin{proof}
First of all we show that $\widetilde{h}$ is a well-defined morphism of perverse sheaves on any $V_{\varphi,\alpha}\in\mathcal O_d$. Invoking \eqref{eq:boxtimes-sigma},  \eqref{eq:a-phi*-equivariance} and equivariance of $f_{\varphi,\alpha}$ and $g_{\varphi,\alpha}$ we have
\begin{align*}
\widetilde h_{\varphi^\sigma,\,^{\sigma}\!\alpha}&=g^{-1}_{\varphi^\sigma,\,^{\sigma}\!\alpha}\circ c_{\varphi^\sigma,\,^{\sigma}\!\alpha}^*\left(\boxtimes_{\varphi^\sigma}(h_{\alpha_{\sigma^{-1}(1)}},\,\ldots,h_{\alpha_{\sigma^{-1}(r)}})\right)\circ f_{\varphi^\sigma,\,^{\sigma}\!\alpha}\\
&=g^{-1}_{\varphi,\alpha}\circ c_{\varphi,\alpha}^*(\sigma^{-1})^*\sigma^*\left(\boxtimes_{\varphi}(h_{\alpha_1},\,\ldots,h_{\alpha_r})\right)\circ f_{\varphi,\alpha}=\widetilde h_{\varphi,\,\alpha}.
\end{align*}
In order to show that these locally defined data can be glued to give a morphism of perverse sheaves $\widetilde{h}\colon \widetilde{\FF}\to \widetilde{\GG}$, it suffices to show that for $V_{\psi,\beta}\subset V_{\varphi,\alpha}$ as in Remark \ref{rem:refinement} (2) there holds
$\widetilde{h}_{\varphi,\alpha}|_{V_{\psi,\beta}}=\widetilde{h}_{\psi,\beta}$.
It follows from  \eqref{eq:effe-compa} that
\begin{align*}c_{\psi,\beta}^*b_{\psi,\beta}^*(f_{\varphi,\alpha})=c_{\psi,\beta}^*\left(\boxtimes_{\varphi}(\mu_{\xi_l,\beta_l})_{l=1}^r\right)\circ f_{\psi,\beta},&&
c_{\psi,\beta}^*b_{\psi,\beta}^*(g^{-1}_{\varphi,\alpha})=g_{\psi,\beta}^{-1}c_{\psi,\beta}^*\left(\boxtimes_{\varphi}(\nu^{-1}_{\xi_l,\beta_l})_{l=1}^r\right)
\end{align*}
and so
\begin{align*}
&\widetilde{h}_{\varphi,\alpha}|_{V_{\psi,\beta}}=c_{\psi,\beta}^*b_{\psi,\beta}^*\widetilde{h}_{\varphi,\alpha}\\
&=g_{\psi,\beta}^{-1}c_{\psi,\beta}^*\left(\boxtimes_{\varphi}(\nu^{-1}_{\xi_l,\beta_l})_{l=1}^r\right)\circ c_{\psi,\beta}^*(\prod_{l=1}^rb_{\xi_l,\gamma_l})^*\left(\boxtimes_{\varphi}(h_{\alpha_1},\,\ldots,\,h_{\alpha_r})\right) \circ c_{\psi,\beta}^*\left(\boxtimes_{\varphi}(\mu_{\xi_l,\beta_l})_{l=1}^r\right)\circ f_{\psi,\beta}\\
&=g_{\psi,\beta}^{-1}c_{\psi,\beta}^*\left(\boxtimes_{\varphi}(\nu^{-1}_{\xi_l,\beta_l}\circ b_{\xi_l,\beta_l}^*(h_{\alpha_l})\circ \mu_{\xi_l,\beta_l})_{l=1}^r\right)\circ f_{\psi,\beta}=g_{\psi,\beta}^{-1}c_{\psi,\beta}^*\left(\boxtimes_{\varphi}(h_{\alpha_1},\,\ldots,
,h_{\alpha_r})\right)\circ f_{\psi,\beta}\\
&=\widetilde{h}_{\psi,\beta}
\end{align*}
where the second-to-last equality follows because $h\colon \FF\to\GG$ is a morphism in $\FP^{\leq d}$.

We now show that the morphism of perverse sheaves $\widetilde{h}$ so defined is a morphism in $\FP_{\leq d}$. Commutativity of the diagram \eqref{eq:def-morphisms} is equivalent to commutativity of 
 \begin{equation}\label{eq:step1}
  \begin{tikzcd}
\boxtimes_{\varphi}(\widetilde{\FF}|_{V_{\psi_1,\beta_1}},\ldots,\, \widetilde{\FF}|_{V_{\psi_r,\beta_r}})
\arrow[rr, "(\widetilde\mu_{\varphi,\alpha})|_{V_{\underline\psi,\underline\beta}}"] 
\arrow[d, swap, "\boxtimes_{\varphi}(\widetilde h_{\alpha_l}|_{V_{\psi_l,\beta_l}}))_{l=1}^r"]& 
&  (b_{\varphi,\alpha}^*\widetilde\FF_{|\alpha|})|_{V_{\underline\psi,\underline\beta}}
\arrow[d, "(b_{\varphi,\alpha}^*\widetilde{h}_{|\alpha|})|_{V_{\varphi,\alpha}} "]\\
\boxtimes_{\varphi}(\widetilde{\GG}|_{V_{\psi_1,\beta_1}},\ldots,\, \widetilde{\GG}|_{V_{\psi_r,\beta_r}})
\arrow[rr, "(\widetilde\nu_{\varphi,\alpha})|_{V_{\underline\psi,\underline\beta}}"]  & 
&  (b_{\varphi,\alpha}^*\widetilde\GG_{|\alpha|})|_{V_{\underline\psi,\underline\beta}}
\end{tikzcd} 
        \end{equation}
        for any $r, m_l\in\NN$, any $\varphi\in E_2(r)$ any $\underline\psi=(\psi_1,\,\ldots,\,\psi_r)\in \prod_{l=1}^r E_2(m_l)$ any $\alpha\in\NN^r$ and any $\underline\beta=(\beta_1,\,\ldots,\,\beta_r)\in(\NN_{\leq d})^{\sum_lm_l}$ with $|\beta_l|=\alpha_l$ for $l=1,\,\ldots,r $. 
        
    Through the identification $(b_{\varphi,\alpha}^*(\widetilde\FF_{|\alpha|}))|_{V_{\underline\psi,\underline\beta}}=b_{\varphi\circ\underline\psi,\underline\beta}^*(\widetilde\FF_{|\alpha|}|_{V_{\varphi\circ\underline\psi,\underline\beta}})$ the right vertical arrow becomes
\begin{align*}
(b_{\varphi,\alpha}^*\widetilde{h}_{|\alpha|})|_{V_{\underline\psi,\underline\beta}}&=b_{\varphi\circ\underline\psi,\underline\beta}^*(\widetilde{h}_{|\alpha|}|_{V_{\varphi\circ\underline\psi,\underline\beta}})\\
&=(\prod_{l=1}^rc_{\psi_l,\beta_l})^*b_{\varphi\circ\underline\psi,\underline\beta}^*\left(g^{-1}_{\varphi\circ\underline\psi,\underline\beta}\circ(c_{\varphi\circ\underline\psi,\underline\beta})^*(\boxtimes_{\varphi\circ\underline\psi}((h_{lp})_{l=1}^r)_{p=1}^{m_l}))\circ f_{\varphi\circ\underline\psi,\underline\beta}\right)
\end{align*}
so applying \eqref{eq:extension-mu-special-open} gives
\begin{align*}
 &\left(b_{\varphi,\alpha}^*\widetilde{h}_{|\alpha|}\circ \widetilde\mu_{\varphi,\alpha}\right)|_{V_{\underline\psi,\underline\beta}}\\
 &=(\prod_{l=1}^rc_{\psi_l,\beta_l})^*b_{\varphi\circ\underline\psi,\underline\gamma}^*\left(g^{-1}_{\varphi\circ\underline\psi,\underline\beta}\circ(c_{\varphi\circ\underline\psi,\underline\beta})^*(\boxtimes_{\varphi\circ\underline\psi}((h_{lp})_{l=1}^r)_{p=1}^{m_l}))\right)\circ \boxtimes_{\varphi}(f_{\psi_l,\beta_l})_{l=1}^r\\
&= \left(\widetilde\nu_{\varphi,\alpha}\circ \boxtimes_{\varphi}(\widetilde h_{\alpha_l})_{l=1}^r \right)|_{V_{\underline\psi,\underline\beta}}.
 \end{align*}
\end{proof}
We are now in a position to state and prove the main result of this Section.
\begin{thm}\label{thm:equivalence} Let $d\in\NN_{\geq 1}$. 
    The assignments $(\mathcal{F}, \mu) \rightsquigarrow (\widetilde\FF,\widetilde\mu)$ and $h\rightsquigarrow\widetilde h$ as in Lemmata \ref{lem:locally-defined}, \ref{lem:ext-mu-morphism}, \ref{lem:ext-mu-FD} and \ref{lem:ext-mor} yield a functor
    \[
    (\ \ )_{\leq d}: \mathcal{FP}^{\leq d} \rightarrow \mathcal{FP}_{\leq d}
    \]
    which is quasi-inverse to the functor $(\ \ )^{\leq d}$. In particular, $\FP_{\leq 1}$, the category of factorized sheaves on $\sym_{\neq}(U)$, is equivalent to $\Vu$.
\end{thm}
\begin{proof}
Lemmata \ref{lem:locally-defined} and \ref{lem:ext-mu-FD} show that the assignment 
$(\mathcal{F}, \mu) \rightsquigarrow (\widetilde\FF,\widetilde\mu)$ maps objects in $\FP^{\leq d}$ to objects in $\FP_{\leq d}$. Lemma \ref{lem:ext-mor} shows that the assignment $h\rightsquigarrow\widetilde h$
maps morphisms in $\FP^{\leq d}$ to morphisms in $\FP_{\leq d}$, and it is straightforward to verify that composition is preserved. Hence we have a functor 
 $(\ \ )_{\leq d}: \mathcal{FP}^{\leq d} \rightarrow \mathcal{FP}_{\leq d}$. 

\medskip

    Let us now show that $(\ \ )^{\leq d}$ and $(\ \ )_{\leq d}$ are one the quasi-inverse of the other. Recall that for any $n\in\NN_{\leq d}$ there is a natural isomorphism of sheaves $f_{\id,(n)}\colon\widetilde{\FF}_{n}\to \FF_{n}$ and that for any morphism $h\colon \FF\to\GG$ in $\FP^{\leq d}$ the restriction of $\widetilde{h}$ to $\sym^n(U)$ is 
    \begin{align*}
\widetilde{h}_{\id,n}:=g_{\id,n}^{-1}\circ h_n\circ f_{\id,n}\colon \widetilde\FF_n\to \widetilde\GG_n.
\end{align*}
Therefore the morphisms $(f_{\id, (n)})_{n\leq d}$ are the components corresponding to $\FF$ of a  natural isomorphism from $(\ \ )^{\leq d} \circ (\ \ )_{\leq d}$ to $\id_{\mathcal{P}^{\leq d}}$. 

\medskip

We now turn to the factorization data. 
The restriction $(\widetilde\mu)^{\leq d}$  is a family of isomorphisms on $\sym(U)(\leq d)$, i.e., the isomorphisms are defined on the collection $\sym^{\beta}(U)$ for $|\beta|\leq d$, where \eqref{eq:extension-mu-piccolo} is in force. Then, \eqref{eq:effe-compa} applied to $r=1$, $\varphi=\id_U$, $\xi_1=\xi$ gives the equalities of isomorphisms of sheaves on $\sym(U)^m(\leq d)$
\begin{align*}
\mu_{\xi,\beta}\circ(\boxtimes_{\xi}(f_{\id,\beta_l})_{l=1}^m)&= b_{\xi,\beta}
^*(f_{\id,|\beta|})b_{\xi,\beta}
^*(f_{\xi,\beta}^{-1})\circ(\boxtimes_{\xi}(f_{\id,\beta_l})_{l=1}^m)=b_{\xi,\beta}
^*(f_{\id,|\beta|})\widetilde{\mu}_{\xi,\beta}.\end{align*}
Hence, $(f_{id,(n)})_{n\leq d}\colon((\widetilde\FF)^{\leq d},\widetilde\mu)\to (\FF,\mu)$ is an isomorphism in $\FP^{\leq d}$, and it is the component on $(\widetilde\FF,\widetilde\mu)^{\leq d}$ of an  isomorphism of functors $(\ )^{\leq d}\circ (\ )_{\leq d}\to \id_{\FP^{\leq d}}$.

\medskip

    Let us then show that the functor $(\ \ )_{\leq d} \circ (\ \ )^{\leq d}$ is naturally isomorphic to the identity functor $\id_{\mathcal{FP}_{\leq d}}$.
    Let $(\FF', \mu')$ be an object of $\mathcal{FP}_{\leq d}$. 
    Set $(\mathcal{F}, \mu) := (\FF', \mu')^{\leq d}$ and 
    $(\widetilde{\mathcal{F}}, \widetilde{\mu}) := (\mathcal{F}, \mu)_{\leq d}$.
    For each special open subset $V_{\varphi,\alpha}\in\mathcal O_d$  we define the isomorphism of perverse sheaves
\begin{align*}
(\eta_{\FF',\mu'})_{\varphi,\alpha}\colon \widetilde\FF|_{V_{\varphi,\alpha}}\longrightarrow \widetilde\FF_{\varphi,\alpha}&=c_{\varphi,\alpha}^*(\boxtimes_{\varphi}(\FF'_{\alpha_1},\ldots,\,\FF'_{\alpha_r}))\longrightarrow c_{\varphi,\alpha}^*b_{\varphi,\alpha}^*(\FF'_{|\alpha|}
)=\FF'_{|\alpha|}|_{V_{\varphi,\alpha}}\\
(\eta_{\FF',\mu'})_{\varphi,\alpha}&=c_{\varphi,\alpha}^*(\mu'_{\varphi,\alpha})\circ f_{\varphi,\alpha}
\end{align*}
We claim that these isomorphisms glue to give an isomorphism $\eta_{\FF',\mu'}$ in $\PPP_{\leq d}$. Indeed, if $\beta=(\beta_1,\,\ldots,\,\beta_r)$ is a refinement of $\alpha$, with $\beta_r\in(\NN_{\leq d})^{m_l}$ for $m_l\in\NN$ and $\underline\xi:=(\xi_1,\,\ldots,\,\xi_r)\in\prod_{l=1}^rE_2(m_l)$, we have
\begin{align*}
   (\eta_{\FF',\mu'})_{\varphi,\alpha}|_{V_{\varphi\circ\underline\xi,\underline\beta}}&=\left(c_{\varphi\circ\underline\xi,\underline\beta}^*(\prod_{l=1}^rb_{\xi_l,\beta_l})^*(\mu_{\varphi,\alpha}')\right)\circ (f_{\varphi,\alpha})|_{V_{\varphi\circ\underline\xi,\underline\beta}}\\
   (\eta_{\FF',\mu'})_{\varphi\circ\underline\xi,\underline\beta}&=\left(c_{\varphi\circ\underline\xi,\underline\beta}^*(\prod_{l=1}^rb_{\xi_l,\beta_l})^*(\mu_{\varphi,\alpha}')\right)\circ c_{\varphi\circ\underline\xi,\underline\beta}^*(\boxtimes_{\varphi}(\mu'_{\xi_l,\beta_l})_{l=1}^r)\circ f_{\varphi\circ\underline\xi,\underline\beta}
   \end{align*}   
and these two isomorphism coincide in virtue of \eqref{eq:effe-compa}.

\medskip

We show that the isomorphisms $\eta_{\FF',\mu'}\colon \widetilde\FF\to\FF'$ so defined are morphisms in $\FP_{\leq d}$. 

Let $\varphi\in E_2(r)$ for $r\in\NN$.  Spelling out the definition of $\widetilde\mu_{\varphi}$ in \eqref{eq:extension-mu-special-open} and the definition of $\eta_{\FF',\mu'}$, the proof boils down 
to the verification that for any product of special open subsets $V_{\underline\psi,\underline\beta}$, with $\underline\psi=(\psi_1,\,\ldots,\,\psi_r)$ and $\underline\beta=(\beta_1,\,\ldots,\,\beta_r)$, where 
$\beta_l\in(\NN_{\leq d})^{m_l}$ for $m_l\in\NN$,  $\alpha_l:=|\beta_l|$ and $l=1,\,\ldots,\,r$  the diagram below commutes.
 \begin{equation*}
            \begin{tikzcd}
\boxtimes_{\varphi}(\widetilde\FF|_{V_{\psi_l,\beta_l}})_{l=1}^r) 
\arrow[rrrrrr, "(\prod_lc_{\psi_l,\beta_l})^*b_{\varphi\circ\underline\psi,\underline\beta}^*(f_{\varphi\circ\underline\psi,\underline\beta}^{-1})\circ\boxtimes_{\varphi}(f_{\psi_l,\beta_l})_{l=1}^r"] 
\arrow[d, swap, "\boxtimes_{\varphi}(c_{\psi_l,\beta_l}^*(\mu_{\psi_l,\beta_l}')\circ f_{\psi_l,\beta_l})_{l=1}^r"]&&&&&& (b_{\varphi,\alpha}^*(\widetilde\FF))|_{V_{\underline\psi,\underline\beta}}
\arrow[d, "(\mu_{\varphi,\alpha}'|_{V_{\underline\psi,\underline\beta}})\circ b_{\varphi,\alpha}^*f_{\varphi,\alpha}"]\\
\boxtimes_{\varphi}(\FF'|_{V_{\psi_l,\beta_l}})_{l=1}^r)
\arrow[rrrrrr, "\mu'_{\varphi,\alpha}|_{V_{\underline\psi,\underline\beta}}"] 
&&&&&& (b_{\varphi,\alpha}^*(\FF'))|_{V_{\underline\psi,\underline\beta}}
\end{tikzcd} 
        \end{equation*}  
Commutativity follows by virtue of \eqref{eq:effe-compa}.

        \medskip
        
We show naturality of $\eta$. Let $h\colon (\FF', \mu')\to (\GG', \nu')$ be a morphism in $\mathcal{FP}_{\leq d}$. 
  We set 
  \begin{align*}&(\FF, \mu) := (\FF', \mu')^{\leq d},&&(\widetilde{\mathcal{F}}, \widetilde{\mu}) := (\mathcal{F}, \mu)_{\leq d},&&h:=(h')^{\leq d},\\
  &  (\GG, \nu) := (\GG', \nu')^{\leq d}, &&(\widetilde{\mathcal{G}}, \widetilde{\nu}) := (\mathcal{G}, \nu)_{\leq d}, &&\widetilde{h}=(h)_{\leq d}.
  \end{align*}
 We need to verify that for any $r\in \NN$, any $\alpha\in (\NN_{\leq d})^r$ and $\varphi\in E_2(r)$ the diagram below commutes
\begin{equation*}
            \begin{tikzcd}
\widetilde\FF|_{V_{\varphi,\alpha}} \arrow[r, "f_{\varphi,\alpha}"]
\arrow[d, swap, 
"g_{\varphi,\alpha}^{-1}\circ c_{\varphi,\alpha}^*(\boxtimes_{\varphi}(h'_{\alpha_l})_{l=1}^r)\circ f_{\varphi,\alpha}
"]
&c_{\varphi,\alpha}^*(\boxtimes_{\varphi}(\FF_{\alpha_1}',\,\ldots,\FF'_{\alpha_r})) \arrow[rr,"c_{\varphi,\alpha}^*(\mu'_{\varphi,\alpha})"]&&\FF'|_{V_{\varphi,\alpha}}\arrow[d, "h'"]\\
\widetilde\GG|_{V_{\varphi,\alpha}} \arrow[r, "g_{\varphi,\alpha}"]
&c_{\varphi,\alpha}^*(\boxtimes_{\varphi}(\GG_{\alpha_1}',\,\ldots,\GG'_{\alpha_r})) \arrow[rr,"c_{\varphi,\alpha}^*(\nu'_{\varphi,\alpha})"]&&\GG'|_{V_{\varphi,\alpha}}
\end{tikzcd} 
        \end{equation*}  
and this readily follows because $h'$ is a morphism in $\FP_{\leq d}$.
 Therefore $\eta=(\eta_{\FF',\mu'})$ is an isomorphism of functors $(\ )_{\leq d}\circ (\ )^{\leq d}\to \id_{\FP_{\leq d}}$, and this concludes the proof of the equivalence. The statement for $d=1$ follows from Proposition \ref{prop:varrho}. 
\end{proof}

Theorem \ref{thm:equivalence} for $d=1$ shows that, if $\FF$ in $\PPP_{\leq 1}(1)$,  then either $\mathrm{FD}(\FF)\neq\emptyset$  or it is a unique orbit for the automorphism group of $\FF$. The results in \cite[Sections 1B, 1C]{KS-snakes}
suggest that this might be the case also for objects in $\PPP(1)$.

\subsection{Extension functors}

In this section we show that the extension functors induced by the open embeddings $j_{\leq d,\leq c}$ and $j_{\leq d}$  for $c\geq d\geq 0$ are well behaved with respect to factorization data.

\begin{thm}\label{prop:def-estensione}
Let $c,d\in\NN$ with $d\leq c$ and let $\bullet$ be either $*$, $!$ or $!*$. The chain of  extension functors \eqref{dgm:lower*and!} gives rise to a chain of functors
\begin{equation}\label{dgm:lower*and!FP}
\begin{tikzcd}[row sep=huge]
 \mathcal{FP}_{\leq d} \arrow[rr, swap, "^p(j_{\leq d,\leq c})_\bullet"'] && \mathcal{FP}_{\leq c}\arrow[rr, swap, "^p(j_{\leq c})_\bullet"'] &&\mathcal{FP}
\end{tikzcd} 
\end{equation}
     As a consequence, $j^*_{\leq d}$ and $j^*_{\leq c,\leq d}$ are essentially surjective for every $c,\,d\in\NN$ with $c\leq d$. 
\end{thm}
\begin{proof}Let $\mathbf j\colon Y\to Y'$ be either $j_{\leq d,\leq c}$ or $j_{\leq c}$ and let
 $(\mathcal{F},\mu)$ be a factorized perverse sheaf on $Y$.
  In order to define the functor  $^p \mathbf j_\bullet$, we proceed as follows: let $n\geq 0$ and let $\varphi\in E_2(n)$. Let $b'_\varphi$ denote the restriction of $a_{\varphi}$  to $(Y')^n$ and $b_\varphi$ denote the restriction of $a_{\varphi}$  to $Y^n$. 

    By  \eqref{dgm:tensor-lower*and!} and \eqref{dgm:lower*and!}  we have  ${^p}(\mathbf j^n)_\bullet ( \mathcal{F}^{\boxtimes_\varphi n})=({^p}\mathbf j_\bullet\mathcal{F})^{\boxtimes_\varphi n}$ and 
    ${}^{p}(\mathbf j^n)_\bullet (b_{\varphi}^\ast \mathcal{F}) =b_{\varphi}^{'\ast}\, {}^p(\mathbf j_\bullet \mathcal{F}) $. 
    Then one defines  ${^p}\mathbf j_\bullet(\mathcal {F}, \mu):=({^p}\mathbf j_\bullet\mathcal{F}, \nu)$, where $\nu_{\psi}:={^p}(\mathbf j^n)_\bullet\mu_{\psi}$ for $\psi\in E_2(n)$. The perverse truncation may be omitted if $\bullet=!*$.  We show that this assignment gives a factorized perverse sheaf on $Y'$.

\medskip

We prove the braiding-monodromy compatibility. Let $\gamma$ be a path from $\varphi$ to $\id_U$ in $E_2(1)$ and let  $M'$ and $\underline{R}'$ denote monodromy and braiding relative to the space $Y'$, and 
$M$ and $\underline{R}$ denote monodromy and braiding relative to the space $Y$. Then
\begin{equation*}
M'_{\gamma,{}^p{\mathbf j}_\bullet\FF}\circ \nu_{\varphi}={}^p{\mathbf j}_\bullet(M_{\gamma,\FF})\circ {}^p{\mathbf j}_\bullet(\mu_\varphi)
={}^p{\mathbf j}_\bullet(\underline{R}_{\gamma,\FF})=\underline{R}_{\gamma,{}^p{\mathbf j}_\bullet(\FF)}\end{equation*}
where the first equality follows from \eqref{eq:monodromy-estension} and the last one follows from the compatibility of $\underline{R}$ with extensions, Lemma \ref{lem:treni-tensor}.

\medskip
    
Let  $\eta=\varphi\circ(\psi_1,\,\ldots,\,\psi_n)$ with $\varphi\in E_2(n)$, $\psi_l\in E_2(m_l)$ for $l=1,\,\ldots,\,n$ and $m=\sum_{l=1}^nm_l$. Compatibility with operadic composition follows from commutativity of the diagrams \eqref{dgm:big1} and \eqref{dgm:big2}. Indeed, 
\begin{align*}
  \nu_{\eta}&={}^p({\mathbf j^m})_\bullet\mu_{\eta}= {}^p({\mathbf j^m})_\bullet\left((\prod_{l=1}^n b_{\psi_l})^*
  (\mu_{\varphi})\circ \boxtimes_{\varphi} (\mu_{\psi_l})_{l=1}^n\right)\\
  &=\left((\prod_{l=1}^n b_{\psi_l})^* \;{}^p({\mathbf j^n})_\bullet \mu_{\varphi}\right)\circ \boxtimes_{\varphi} (^p(\mathbf j^{m_l})_\bullet \mu_{\psi_l})_{l=1}^n= (\prod_{l=1}^n b_{\psi_{l}})^*
(\nu_{\varphi})\circ \boxtimes_{\varphi} (\nu_{\psi_{l}})_{l=1}^n.
 \end{align*}

\medskip

We prove symmetric equivariance. Let $\sigma\in \Sn$. By base change, \cite[p.625]{dCM}:
\begin{equation*}
  \nu_{\varphi^\sigma}={}^p({\mathbf j^n})_\bullet\mu_{\varphi^\sigma}={}^p({\mathbf j^n})_\bullet\sigma^*(\mu_\varphi)=\sigma^* {}^p({\mathbf j}^n)_\bullet\mu_\varphi=\sigma^*\nu_{\varphi}.
\end{equation*}

\medskip 

Concerning morphisms, it follows from  compatibility of ${}^p{\mathbf j}_\bullet$ with  $\boxtimes_\varphi$, and $b_\varphi^*$, see \eqref{dgm:tensor-lower*and!} and \eqref{dgm:lower*and!}, that the extension of any morphism of factorized perverse sheaves on $Y$ is a morphism of factorized perverse sheaves on $Y'$.

\medskip 

Finally, essential surjectivity follows because ${\mathbf j}^*\circ \,{}^p\mathbf j_\bullet(\FF,\mu)=(\FF,\mu)$ for every factorized sheaf $(\FF,\mu)$ in $\FP_{\leq d}$.
\end{proof}

The extension  functors can also be defined on the (isomorphic) categories of vertical  factorized perverse sheaves, with compatibility.
 \begin{cor}\label{cor:factvb}
   Let $\bullet$ be either $\ast$, $!$, or ${!\ast}$ and let $d\in \NN$.
    The  extension functor ${}^p(j_{\leq d})_\bullet$ fits into the following commutative diagram, where the vertical arrows are the forgetful isomorphisms from Proposition \ref{prop:factvb}.  
    \[   \begin{tikzcd}
\mathcal{FP}_{\leq d} \ar[r,  "{}^p(j_{\leq d})_\bullet"]\ar[d, "\cong"]& \mathcal{FP} \ar[d, "\cong"]  \\
  \mathcal{FP}^v_{\leq d} 
        \ar[r, swap,  "{}^p(j_{\leq d})_\bullet"]
        &\mathcal{FP}^v 
    \end{tikzcd}
    \]
    \end{cor}
\begin{proof}
Commutativity of the diagram follows from the definition of the functors ${^p}(j_{\leq d})_{*}$, ${^p}j_{\leq d, !}$ and $(j_{\leq d})_{!*}$ in Theorem \ref{prop:def-estensione}, and of the forgetful functors from Proposition \ref{prop:factvb}.
\end{proof}

Composing $(\ )_{\leq d}$ with ${}^p(j_{\leq d})_*$, ${}^p(j_{\leq d})_!$ or ${}^p(j_{\leq d})_{!*}$ gives extension functors $\FP^{\leq d}\to \FP$. 
We now show a natural adjunction property for restriction and extension by zero functors.

\begin{thm}
 \label{thm:extension-geometric}Let $d\in \NN_{\geq1}$. The functor ${}^p(j_{\leq d})_!\colon \FP_{\leq d}\to \FP$ is left adjoint to $j_{\leq d}^*$. As a consequence,  ${}^p(j_{\leq d})_!\circ(\ )_{\leq d}\colon \FP^{\leq d}\to \FP$ is left adjoint to the truncation functor $(\ )^{\leq d}$.
 \end{thm}
\begin{proof}
At the level of perverse sheaves, $^p\!(j_{\leq d})_!$ is left adjoint to $j_{\leq d}^*=j_{\leq d}^!$, that is,  we have a natural  bijection $\mathrm{Hom}_{\PPP_{\leq d}(1)}({\mathcal F},j_{\leq d}^*\widetilde{\mathcal G})\to 
\mathrm{Hom}_{\PPP(1)}(\,^p(j_{\leq d})_!{\mathcal F},\widetilde{\mathcal G})$ for any ${\mathcal F}$ in $\PPP_{\leq d}(1)$ and any $\widetilde{\mathcal G}$ in $\PPP(1)$. 

In order to ensure that the bijection carries over at the level of factorized sheaves, it is enough to show that,  if $({\mathcal F},\mu)$ and 
$(\widetilde{\mathcal G},\widetilde{\nu})$ are factorized, then the morphisms of perverse sheaves 
$\iota_{\mathcal F}\colon {\mathcal F}\to j_{\leq d}^*{}^p(j_{\leq d})_! ({\mathcal F})$ and $\iota_{\widetilde{\mathcal G}}\colon {}^p(j_{\leq d})_! j_{\leq d}^*\widetilde{\mathcal G}\to \widetilde{\mathcal G}$ corresponding through the adjunction to  $\id_{{}^p(j_{\leq d})_! {\mathcal F}}$ and $\id_{j_{\leq d}^*\widetilde{\mathcal G}}$, respectively, satisfy condition \eqref{eq:def-morphisms} for the pairs of factorized sheaves $({\mathcal F},\mu)$,  $j_{\leq d}^*\,{}^p(j_{\leq d})_! ({\mathcal F}, \mu)$,  and  
${}^p(j_{\leq d})_!\,j_{\leq d}^*(\widetilde{\mathcal G},\widetilde{\nu})$, $(\widetilde{\mathcal G}, \widetilde{\nu})$, respectively. 

\medskip

Now, $\iota_{\mathcal F}$ is (naturally isomorphic to) the identity and the restriction of  ${}^p(j_{\leq d})_! \mu$ to $\sym_{\leq d}(U)$  is $\mu$, so there is nothing to prove in this case. 

We consider $\iota_{\widetilde{\mathcal G}}$. We need to show that for any $n\in\NN$ and any $\varphi\in E_2(n)$ the diagram below is commutative. 
 \begin{equation}
            \begin{tikzcd}[row sep=huge]
( {}^p(j_{\leq d})_! j_{\leq d}^*\widetilde{\mathcal G})^{\boxtimes_{\varphi} n} 
\arrow[rr, "{}^p(j^n_{\leq d})_!\,(j^n_{\leq d})^*\widetilde\nu_\varphi"] 
\arrow[d, swap, "\iota_{\widetilde{\GG}}^{\boxtimes_\varphi n}"]& 
& a_\varphi^\ast \mathcal{F}
\arrow[d, "a_\varphi^\ast (\iota_{\widetilde \GG})"]\\
 \widetilde\GG^{\boxtimes_{\varphi} n}
\arrow[rr, "\widetilde\nu_{\varphi}"]  & 
& a_{\varphi}^\ast \widetilde\GG
\end{tikzcd} 
        \end{equation}

By definition, $\iota_{\widetilde\GG}$ is the unique morphism of perverse sheaves  on $\sym(U)$ whose restriction to $\sym_{\leq d}(U)$ is the identity. Functoriality of $j_{\leq d}^*$ and \eqref{dgm:upperstar} give
\begin{align*}
(j_{\leq d}^n)^*\left(\widetilde\nu^{-1}_{\varphi}\circ a_\varphi^\ast( \iota_{\widetilde \GG})\circ {}^p(j^n_{\leq d})_!\,(j^n_{\leq d})^*\widetilde\nu_\varphi\right)&=
(j_{\leq d}^n)^*(\widetilde\nu^{-1}_{\varphi})\circ (j_{\leq d}^n)^*(a_\varphi^\ast(\iota_{\widetilde \GG}))\circ (j_{\leq d}^n)^*{}^p(j^n_{\leq d})_!\,(j^n_{\leq d})^*\widetilde\nu_\varphi\\
&=(j_{\leq d}^n)^*(\widetilde\nu^{-1}_{\varphi})\circ a^*_{\varphi,\leq d}(j^*_{\leq d}(\iota_{\widetilde\GG}))\circ (j_{\leq d}^n)^*(\widetilde\nu_{\varphi})\\
&=\id_{(j^*_{\leq d}\widetilde\GG)^{\boxtimes_\varphi n}}.
\end{align*} 
Hence, $\widetilde\nu^{-1}_{\varphi}\circ a_\varphi^\ast \iota_{\widetilde \GG}\circ {}^p(j^n_{\leq d})_!\,(j^n_{\leq d})^*\widetilde\nu_\varphi\colon ( {}^p(j_{\leq d})_! j_{\leq d}^*\widetilde{\mathcal G})^{\boxtimes_{\varphi} n} \longrightarrow  \widetilde\GG^{\boxtimes_{\varphi} n}$ is the only morphism of perverse sheaves on $\sym(U)^n$ whose restriction to $\sym_{\leq d}(U)^n$ is $\id_{(j^*_{\leq d}\widetilde\GG)^{{\boxtimes_\varphi}n}}$.
On the other hand, \eqref{dgm:tensor-lower*and!} implies
\begin{equation*}
(j_{\leq d}^n)^*(\iota_{\widetilde\GG}^{\boxtimes_\varphi n})=(j^*_{\leq d}\iota_{\widetilde\GG})^{\boxtimes_\varphi n}=\id_{(j^*_{\leq d}\widetilde\GG)^{\boxtimes_\varphi n}}
\end{equation*}
so commutativity of the diagram follows from uniqueness.

The last statement follows from Theorem \ref{thm:equivalence} and the compatibility \eqref{dgm:compatibility-limit}, which holds also for factorized sheaves.
\end{proof}

\begin{remark}\label{rem:uniqueness}{\rm For a factorized perverse sheaf $({\mathcal F},\mu)$ in $\FP_{\leq d}$,  the collection ${}^p(j_{\leq d})_!(\mu)$ is the unique element in $\mathrm{FD}({}^p(j_{\leq d})_!{\mathcal F})$ extending $\mu$. Indeed, if $\nu\in \mathrm{FD}({}^p(j_{\leq d})_!{\mathcal F})$ extends $\mu$,  by  the adjunction at the level of factorized perverse sheaves there is a unique morphism 
$({}^p(j_{\leq d})_!{\mathcal F},{}^p(j_{\leq d})_!(\mu))\to({}^p(j_{\leq d})_!{\mathcal F},\nu)$ extending the identity 
$({\mathcal F},\mu)\to ({}^p(j_{\leq d})_!{\mathcal F},\,j_{\leq d}^*(\nu))=({\mathcal F},\,\mu)$. However,  adjunction at the level of perverse sheaves implies that the identity is the unique morphism of perverse sheaves 
${}^p(j_{\leq d})_!{\mathcal F}\to \;{}^p(j_{\leq d})_!{\mathcal F}$ corresponding to the identity in $\mathrm{Hom}_{\PPP_{\leq d}(1)}({\mathcal F},j_{\leq d}^*\,({}^p(j_{\leq d})_!{\mathcal F}))=\mathrm{Hom}_{\PPP_{\leq d}(1)}({\mathcal F},{\mathcal F})$.  Therefore,  the identity is a morphism $({}^p(j_{\leq d})_!{\mathcal F},\,{}^p(j_{\leq d})_!(\mu))\to({}^p(j_{\leq d})_!{\mathcal F},\nu)$,  so \eqref{eq:def-morphisms} forces $\nu={}^p(j_{\leq d})_!(\mu)$.}
\end{remark}

\section{Inverse limits}

In this Section we show that the categories $\FP^{\leq d}$ and $\FP_{\leq d}$ fit into compatible systems of categories whose limit is $\FP$. We first introduce the necessary terminology.

\medskip
 
A projective system of categories  indexed by $\NN$ is a pair $(({\mathcal C}_d)_{d\in {\mathbb N}},(\pi_{de})_{d\leq e})$
where $({\mathcal C}_d)_{d\in {\mathbb N}}$ is a collection of categories indexed by ${\mathbb N}$ and $\pi_{de}\colon{\mathcal C}_e\to {\mathcal C}_d$ for $d,\,e\in\NN$ with $d\leq e$ denote a collection of  functors satisfying $\pi_{de}\circ\pi_{ef}=\pi_{df}$ if $d\leq e\leq f$ and $\pi_{dd}=\id_{\mathcal C_d}$ for all $d\in \NN$. 

In the spirit of \cite[\S 5.1]{Sch} we define the inverse limit  category 
 $\varprojlim_{d\in {\mathbb N}}{\mathcal C}_d$,\label{limit} 
the category whose objects  are systems $\left((E_d)_{d\in {\mathbb N}},(\phi_{de})_{d\leq e}\right)$, where each $E_d$ is an object in ${\mathcal C}_d$ and each $\phi_{de}$ is an isomorphism 
$\phi_{de}\colon \pi_{de}(E_e)\to E_d$ in ${\mathcal C}_d$ satisfying $\phi_{de}\circ \pi_{de}(\phi_{ef})=\phi_{df}$ for $d\leq e\leq f$. In particular, $\phi_{ee}\circ\phi_{ee}=\phi_{ee}$ for every $e\in\NN$, so for the categories we focus on, $\phi_{ee}=\id_{E_e}$.  Morphisms from $\left((E_d)_{d\in \NN},(\phi_{de})_{d\leq e}\right)$ to $\left((E'_d)_{d\in \NN},(\phi'_{de})_{d\leq e}\right)$ are defined as families of morphisms $\psi_d\colon E_d\to E_d'$ in ${\mathcal C}_d$ for each $d\in {\mathbb N}$ such that the diagram below commutes for all $e\geq d$:
\begin{equation}\label{eq:CDlimit}
\begin{CD}
\pi_{de}E_e@>{\phi_{de}}>>E_d\\
@V{\pi_{de}(\psi_e)}VV @VV{\psi_d}V\\
\pi_{de}E'_e@>{\phi'_{de}}>>E'_d\\
\end{CD}
\end{equation}
This construction is a special case of the quasi-limit, \cite[p. 201]{G}. 

\medskip

Let $(({\mathcal C}_d)_{d\in {\mathbb N}},(\pi_{de})_{d\leq e})$ and $(({\mathcal C}'_d)_{d\in {\mathbb N}}, (\pi'_{de})_{d\leq e})$ be projective systems of categories. A sequence of functors $(G_d\colon {\mathcal C}_d\to {\mathcal C}'_d)_{d\in{\mathbb N}}$ that are compatible with the $\pi_{de}$ and $\pi_{de}'$ for $d\leq e$ determines a  functor $G\colon \varprojlim_{d\in {\mathbb N}}{\mathcal C}_d\to \varprojlim_{d\in {\mathbb N}}{\mathcal C}'_d$ mapping an object $\left((E_d)_{d\in \!{\mathbb N}},(\phi_{de})_{d\leq e}\right)$ to
$\left(\!(G_dE_d)_{d\in {\mathbb N}},\!(G_d(\phi_{de}))_{d\leq e}\right)$ and a morphism $(f_d)_{d\in {\mathbb N}}\colon\!\left((E_d)_{d\in {\mathbb N}},(\phi_{de})_{d\leq e}\right)\to\!\left((E'_d)_{d\in {\mathbb N}},(\phi'_{de})_{d\leq e}\!\right)$ to $(G_d(f_d))_{d\in {\mathbb N}}$.  We call this functor the limit of the sequence $(G_d)_{d\in {\mathbb N}}$.

\medskip

Therefore, the collections  
\begin{equation}
((\FP_{\leq d})_{d\in\NN}, (j_{\leq d,\leq c}^*)_{d\leq c}),\quad\mbox{ and }\quad ((\FP^{\leq d})_{d\in\NN}, ((\ )^{\leq d,\leq c})_{d\leq c}) 
\end{equation}
are projective systems of categories indexed by $\NN$
and the sequence of equivalences given by the truncation functors $(\ )^{\leq d}\colon \FP_{\leq d}\to \FP^{\leq d} $ for each $d\in \NN$ fit into commutative diagrams of functors  for any $d\leq c$
\begin{equation}
\begin{tikzcd}
\FP_{\leq c} \arrow[rr, "(\ )^{\leq c}"]\arrow[d, swap, "j_{\leq d,\leq c}^*"]&&\FP^{\leq c}\arrow[d, "(\ )^{\leq d,\leq c}"]\\
\FP_{\leq d} \arrow[rr, "(\ )^{\leq d}"]&&\FP^{\leq d}
    \end{tikzcd}
   \end{equation}
We consider now their limits.

\begin{thm}\label{thm:limit}\begin{enumerate}
   \item The  inverse limit of $((\FP^{\leq d})_{d\in\NN}, ((\ )^{\leq d,\leq c}))_{d\leq c})$ is equivalent to $\FP$. 
     \item The  inverse limit of $((\FP_{\leq d})_{d\in\NN}, (j_{\leq d,\leq c}^*)_{d\leq c})$ is equivalent to $\FP$. 
    \item Through the equivalences $\FP\simeq\varprojlim_{d\in {\mathbb N}}(\FP_{\leq d})$ and $\FP\simeq \varprojlim_{d\in {\mathbb N}}(\FP^{\leq d})$, the limit of the sequence of the truncation functors $((\ )^{\leq d}\colon \FP_{\leq d}\to\FP^{\leq d})_{d\in \NN}$ is naturally isomorphic to the identity functor on $\FP$.
\end{enumerate}
\end{thm}
\begin{proof} (1) Let ${\mathcal Y}\colon \FP\to \varprojlim_{d\in\NN} \FP^{\leq d}$ be the functor given by 
\begin{equation*}(({\mathcal F}_n)_{n\in\NN},(\mu_\varphi)_{\varphi\in E_2(m), m\in\NN})
\mapsto((({\mathcal F}_n)_{n\leq d},(\mu_\varphi^{\leq d})_{\varphi\in E_2(m), m\in\NN}) _{d\in\NN},(\id_{dc})_{d\leq c})\end{equation*} on objects and by $(f_n)_{n\in\NN}\mapsto ((f_n)_{n\leq d})_{d\in\NN})$ on morphisms.  It is fully faithful by construction. We show that it is essentially surjective. Let $(({\mathcal G}^d,\mu^d)_{d\in\NN_{\geq1}},(\phi_{dc})_{d\leq c})$ be an object in $ \varprojlim_{d\in\NN} \FP^{\leq d}$, that is,  ${\mathcal G}^d=( {\mathcal G}^d_n)_{n\leq d}\in\PPP^{\leq d}(1)$ and  $\mu^d\in \mathrm{FD}(\GG^d)$, and  $\phi_{dc}\colon (\GG^c,\mu^c)^{\leq d}\to(\GG^d,\mu^d)$ 
for $d\leq c$ are isomorphisms satisfying 
$\phi_{df}=\phi_{de}\circ (\phi_{ef})^{\leq d,\leq e}$ for all $d\leq e\leq f$. 
Then $\widetilde{\GG}=({\mathcal G}^d_d)_{d\in\NN}$ is an  object in $\PPP(1)$.  We show that it
admits a factorization datum $\mu=(\mu_{\varphi})_{\varphi\in E_2(m), m\in\NN}$ such that $
{\mathcal Y}(({\mathcal G}^d_d)_{d\in\NN},\mu)\simeq(({\mathcal G}^d,\mu^d)_{d\in\NN_{\geq1}},(\phi_{cd})_{c\leq d})$.  

\medskip

Let $n\in\NN$ and $\varphi\in E_2(n)$. On the connected component of $(\sym(U)^n)(d)$ corresponding to $\alpha=(\alpha_1,\,\ldots,\,\alpha_n)$, we define the isomorphism of sheaves 
\begin{align*}
  \mu_{\varphi,\alpha}&\colon \boxtimes_{\varphi}(\GG_{\alpha_1}^{\alpha_1},\,\ldots,\GG_{\alpha_n}^{\alpha_n})\longrightarrow a_{\varphi,\alpha}^*(\GG^d_d)  \\
  \mu_{\varphi,\alpha}&:=\mu_{\varphi,\underline{d}}^d\circ\boxtimes_{\varphi}((\phi^{\alpha_1}_{\alpha_1 d})^{-1},\,\ldots,\,(\phi^{\alpha_n}_{\alpha_n d})^{-1})
\end{align*}
where $\phi^{\alpha_i}_{\alpha_i d}$ is the restriction of $\phi_{\alpha_i d}$ to the top degree component $\sym^{\alpha_i}(U)$ of $\sym^{\leq \alpha_i}(U)$. This way we obtain a collection of isomorphisms $\mu:=(\mu_{\varphi})_{\varphi\in E_2(n), n\in\NN}$.

\medskip

We verify that $\mu$ satisfies braiding-monodromy compatibility. Let $\gamma$ be a path in $E_2(1)$ from $\varphi$ to $\id_U$. Since $\phi_{dd}=\id$, the diagram \eqref{eq:mon-bra} for $\GG^d_d$, on $\sym^{d}(U)$ becomes
\begin{equation*}\label{eq:limit-mon-bra}
\begin{tikzcd}
 \boxtimes_{\varphi}\GG_{d}^{d}
 \arrow[rr,"{\boxtimes_{\varphi}\id}"]
 \arrow[drr, swap, "\underline{R}_{\gamma,\widetilde{\GG}}"]
 && \boxtimes_{\varphi}\GG_{d}^{d} \arrow[rr, "{\mu_{\varphi,\alpha}^d}"]  \arrow[d,"{\underline{R}_{\gamma,\GG^d}}"] && a_{\varphi,\alpha}^\ast{\GG^d_d}\arrow[dll, "{M_{\gamma,\widetilde{\GG}}}"]
 \\
 && \GG_{d}^{d} 
\end{tikzcd} 
        \end{equation*}
The right triangle is braiding-monodromy compatibility for $\GG^d$ and the left square commutes by naturality of  $\underline{R}$.
Hence, the exterior square commutes, giving \eqref{eq:mon-bra}.  

\medbreak
 
We verify that $\mu$ satisfies \eqref{eq:factdat_a_partition}. Observe that the morphism property of $\phi_{cd}$ with $c\leq d$ restricted to the connected component of $\sym^\beta(U)$ corresponding to the composition $\beta=(\beta_1,\,\cdots,\,\beta_n)$ of $c$, for $\psi\in E_2(n)$ gives the commutativity of the diagram
\begin{equation}\label{eq:phi-top}
\begin{tikzcd}
    \boxtimes_{\psi}(\GG_{\beta_l}^d)_{l=1}^n
    \arrow[d,swap,"{\boxtimes_{\psi}(\phi_{cd}^{\beta_1},\,\ldots,\,\phi_{cd}^{\beta_n})}"]
   \arrow[rr,"\mu_{\psi,\beta}^d"]&&a_{\psi,\beta}^*(\GG^d)\arrow[d,"a_{\psi,\beta}^*(\phi_{cd}^d)"]\\
  \boxtimes_{\psi}(\GG_{\beta_l}^c)_{l=1}^n
   \arrow[rr,"\mu_{\psi,\beta}^c"]&&a_{\psi,\beta}^*(\GG^c)  
     \end{tikzcd}
\end{equation} and that the compatibility of the isomorphisms $\phi_{cd}$ gives $\phi_{ce}^c=\phi_{cd}^c\circ \phi_{de}^c$ for $c\leq d\leq e$.  

For $l=0,\,\ldots,\,n$, let $m_l\in\NN$ with $m_0=0$, and set $m=\sum_{l=1}^n m_l$. Let $\varphi\in E_2(n)$ and 
$\psi_l\in E_2(m_i)$ for each $l\geq 1$. For 
$\alpha=(\alpha_1,\,\ldots,\alpha_m)\in \NN^m$ with $|\alpha|=d$, we set $\beta_l=(\alpha_{m_1+\cdots+ m_{l-1}+1},\,\ldots,\alpha_{m_1+\cdots+m_l})$ for $l=1,\,\ldots, n$, so that $\beta_{lp}=\alpha_{m_1+\cdots+ m_{l-1}+p}$ and $|\beta_l|=\sum_{p=m_1+\cdots+ m_{l-1}+1}^{m_1+\cdots+ m_{l-1}+m_l}\alpha_p$ and $\sum_{l=1}^m|\beta_l|=d$. Then, using in sequence:  the definition of $\mu$; functoriality of $a_{\psi_l\beta_l}^*$; commutativity of \eqref{eq:phi-top} and functoriality of $\boxtimes_{\varphi}$; compatibility of operadic composition of $\boxtimes$,  compatibility among the $\phi_{cd}$ and compatibility with operadic composition of $\mu^d$ we obtain 
\begin{align*} 
&((\Pi_{l=1}^n a_{\psi_l,\beta_l})^\ast(\mu_{\varphi,(|\beta_1|,\,\ldots,\,|\beta_n|)}))\circ
 \boxtimes_\varphi(\mu_{\psi_l,\beta_l})_{l=1}^n\\
 &=\left((\Pi_{l=1}^n a_{\psi_l,\beta_l})^\ast(\mu^d_{\varphi,(|\beta_1|,\,\ldots,\,|\beta_n|)}\circ \boxtimes_{\varphi}((\phi_{|\beta_l|,d}^{|\beta_l|})^{-1})_{l=1}^n)\right)\circ\boxtimes_{\varphi}
 \left(\mu_{\psi_l,\beta_l}^{|\beta_1|}\circ\boxtimes_{\psi_l}((\phi_{\beta_{lp}|\beta_l|}^{\beta_{lp}})^{-1})_{p=1}^{m_l}\right)_{l=1}^n\\
 &=(\Pi_{l=1}^n a_{\psi_l,\beta_l})^\ast(\mu^d_{\varphi,(|\beta_1|,\,\ldots,\,|\beta_n|)})\circ \boxtimes_{\varphi}(a_{\psi_l,\beta_l}^*((\phi_{|\beta_l|d}^{|\beta_l|})^{-1})_{l=1}^n)\circ \boxtimes_{\varphi}
 \left(\mu_{\psi_l,\beta_l}^{|\beta_1|}\circ\boxtimes_{\psi_l}((\phi_{\beta_{lp}|\beta_l|}^{\beta_{lp}})^{-1})_{p=1}^{m_l}\right)_{l=1}^n\\
&=(\Pi_{l=1}^n a_{\psi_l,\beta_l})^\ast(\mu^d_{\varphi,(|\beta_1|,\,\ldots,\,|\beta_n|)})\circ 
\boxtimes_{\varphi}(\mu_{\psi_l,\beta_l}^d)_{l=1}^n\circ \boxtimes_{\varphi}(\boxtimes_{\psi_l}((\phi_{|\beta_l|d}^{\beta_{lp}})^{-1}(\phi_{\beta_{lp}|\beta_l|}^{\beta_{lp}})^{-1})_{p=1}^{m_l})_{l=1}^n\\
&=\mu_{\varphi\circ\underline{\psi},\alpha}^d\circ\boxtimes_{\varphi\circ\underline{\psi}}((\phi^{\alpha_1}_{\alpha_1 d})^{-1},\,\ldots,\,(\phi^{\alpha_m}_{\alpha_m d})^{-1})= \mu_{\varphi\circ\underline{\psi},\alpha}.\end{align*}

\medskip

It is straightforward to verify that $\mu_{\varphi,\alpha}$ satisfies symmetric equivariance. Hence, $(\widetilde{\GG},\mu)$ is an object in $\FP$. A tedious but straightforward calculation shows that the collection $(\phi^{-1}_{nd})_{n\leq d}\colon ({\mathcal G}^n_n)_{n\leq d}\to {\mathcal G}^d$ for each $d\in \NN$ gives an isomorphism between ${\mathcal Y}{(\widetilde{\GG},\mu)}=((({\mathcal G}^n_n)_{n\leq d},\mu)_{d\geq 1},(\id_{cd})_{c\leq d})$ and $(({\mathcal G}^d,\mu^d)_{d\in\NN_{\geq1}},(\phi_{cd})_{c\leq d})$. In fact, the assignment 
\begin{equation*}(({\mathcal G}^d,\mu^d)_{d\in\NN_{\geq1}},(\phi_{cd})_{c\leq d})\to ((({\mathcal G}^n_n)_{n\leq d},\mu)_{d\geq 1},(\id_{cd})_{c\leq d})\end{equation*} defines a quasi-inverse $\mathcal{Y}'$ of $\mathcal{Y}$.

\medskip

(2) This follows from (1) and Theorem \ref{thm:equivalence}.
An explicit equivalence is given by the functor $\mathcal Z\colon \FP\to \varprojlim_{d\in\NN} \FP_{\leq d}$  given by 
\begin{equation*}
{\mathcal Z}({\mathcal F},\mu)=
((j_{\leq d}^*({\mathcal F},\mu))_{d\in\NN},(\id_{cd})_{c\leq d})\end{equation*} on objects and by ${\mathcal Z} (f)=(j^*_{\leq d}f)_{d\in\NN}$ on morphisms.  

\medskip

(3) Let $(\FF,\mu),(\GG,\nu)\in\FP$, let $f\colon (\FF,\mu)\to(\GG,\nu)$ be a morphism in $\FP$, and let $\varprojlim_{d\in\NN} ((\ )^{\leq d}))_{d\in\NN}$ be the limit functor of the sequence of truncation functors.  Then, 
\begin{align*}
\varprojlim_{d\in\NN} ((\ )^{\leq d}))_{d\in\NN}\circ {\mathcal Z}(\FF,\mu)&= (((\FF,(\mu))^{\leq d})_{d\in\NN},(\id_{cd})_{c\leq d})={\mathcal Y}(\FF,\mu),\\
\varprojlim_{d\in\NN} ((\ )^{\leq d}))_{d\in\NN}\circ {\mathcal Z}(f)&=(f_d)_{d\in\NN}={\mathcal Y}(f),\end{align*} giving the claim. 
\end{proof}

\section{Acknowledgements}
Project funded by the EuropeanUnion – NextGenerationEU under the National
Recovery and Resilience Plan (NRRP), Mission 4 Component 2 Investment 1.1 -
Call PRIN 2022 No. 104 of February 2, 2022 of Italian Ministry of
University and Research; Project 2022S8SSW2 (subject area: PE - Physical
Sciences and Engineering) ``Algebraic and geometric aspects of Lie theory". G.C. and L. R. y D. are  members of the INdAM group GNSAGA.

\newpage

\begin{center}
\begingroup
\setlength{\tabcolsep}{10pt} 
\renewcommand{\arraystretch}{1.5} 
\begin{tabular}{c|c|c}
SYMBOL& DESCRIPTION & PAGE OR EQUATION\\
\hline
$\NN_{\leq d}$, resp. $\NN_{\geq d}$& natural numbers $\leq d$, resp. $\geq d$ &\\
$I=(0,1)$ &unit interval in $\RR$ &page \pageref{UU}\\
$U=I\times I$& unit square in  ${\mathbb R}^2\simeq\CC$ & page \pageref{UU}\\
$\varprojlim_{d\in {\mathbb N}}{\mathcal C}_d$& inverse limit of categories & page \pageref{limit}\\
$\boxtimes_{\varphi}$ & external tensor product &\pageref{boxphi}\\
$\underline{R}_\gamma$& external braiding & \eqref{eq:defR}\\
$\mathrm{Top}^{\coprod{}}$&category of topological spaces $\otimes=\coprod$& Example \ref{ex:top}\\
$\mathrm{Top}^{\prod{}}$& category of topological spaces  $\otimes=\times$&Example \ref{ex:top}\\
$E_2(n)$& space of $n$-ary linear embeddings& Section \ref{sec:little}\\
$U^{\coprod n}$&$\underbrace{U\amalg \cdots\amalg U}_{n \mbox{ \tiny times}}$&\\
$E_2$& operad of little 2-cubes&Section \ref{sec:little}\\ $\Pi_1(E_2)$,& fundamental groupoid of  $E_2$& page \pageref{p:groupoid} \\
${\mathcal S}trat$& category of stratified spaces &page \pageref{p:strat}\\
$\sym(Y)$ & symmetric product of $Y$ & Section \ref{sec:sym-prod} \\
${\mathbf s}_0$& unique point in $\sym^0(Y)$ for any $ Y\subseteq\mathbb C$ & page \pageref{s0}\\
$a_{\varphi}$& multiplication embedding &\eqref{a-phi}\\
$(a_{\varphi})_{\leq d}$& multiplication embedding &\eqref{a-phi-leq}\\
$\sym_{\leq d}(U)$ &  open subset of multiplicity $\leq d$ & \eqref{eq:truncated-strata}\\
$ \sym^{\leq d}(U)$ & $\coprod_{j\leq d}\sym^j(U)$ & \eqref{eq:truncated-strata}\\
$\sym^n_{\neq}(U)$& configuration space of $n$ points in $U$ & \eqref{eq:truncated-strata}\\
$\Sigma$& diagonal stratification& page \pageref{stratification}\\
$\mathrm{P}(n)$ & partitions of $n$& page \pageref{stratification}\\
$\sym_\lambda(U)$& stratum corresponding to $\lambda$ &page \pageref{stratification}\\
$\mathcal{P}(n)$ & perverse sheaves on $\sym(U)^n$&page \pageref{page:P(n)}\\
$d(\lambda)$& depth of a partition or composition& page \pageref{depth}\\
$|\lambda|$& sum of the components & page \pageref{depth}\\
\end{tabular}

\newpage

\setlength{\tabcolsep}{10pt} 
\renewcommand{\arraystretch}{1.5} 
\begin{tabular}{c|c|c}
SYMBOL& DESCRIPTION & PAGE OR EQUATION\\
\hline
$j_{\leq d,\leq c}, j_{\leq d}$& inclusions & \eqref{eq:j-inclusions}\\
$\mathcal{P}_{\leq d}(n)$& perverse sheaves on $\sym_{\leq d}(U)^n$&page \pageref{perverse-X}\\
$\mathcal{P}^{\leq d}(n)$& perverse sheaves on  $\sym(U)^n(\leq d)$&page \pageref{perverse-X}\\
$\mathcal{P}$, $\mathcal{P}_{\leq d}$, $\mathcal{P}^{\leq d}$& $\coprod_n \mathcal{P}(n)$, $\coprod_n \mathcal{P}_{\leq d}(n)$ and $\coprod_n \mathcal{P}^{\leq d}(n)$ &page \pageref{perverse-X}\\
$j_{\leq d}^*$, $(j_{\leq c,\leq d})^*$ & restriction functors to $\PPP_{\leq d}$ & page \pageref{p:funtori-j}\\
$(\ )^{\leq d}$, $(\ )^{\leq c,\leq d}$ & truncation functors to $\PPP^{\leq d}$ & page \pageref{p:funtori-i}\\
$M_\gamma$& monodromy natural transformation& \eqref{eq:emmegamma}\\
$b_\varphi$& $a_\varphi$, $a_{\varphi,\leq d}$, or $a_{\varphi}^{\leq d}$& page \pageref{page:b}\\
$\mu_\varphi$& factorization datum &\eqref{eq:mu-phi}\\
$\mathrm{FD}(\FF)$& factorization data of $\FF$ & page \pageref{page:FD} \\
$a_{\varphi,\alpha}$& restriction of $a_\varphi$ to $\sym^\alpha(U)$& page \pageref{p:a-phi-alpha}\\
$\mu_{\varphi,\alpha}$& restriction of $\mu_\varphi$ to $\sym^\alpha(U)$& page \pageref{p:a-phi-alpha}\\
$\FP$&
 factorized perverse sheaves on $\sym(U)$ & Definition \ref{def:FP}\\
 $\FP_{\leq d}$&
 factorized perverse sheaves on $\sym_{\leq d}(U)$ & Definition \ref{def:FP}\\
 $\FP^{\leq d}$&
 factorized perverse sheaves on $\sym^{\leq d}(U)$& Definition \ref{def:FP}\\
 $\varrho,\varrho_{\leq d},\varrho^{\leq d}$& restriction functors to $\sym^1(U)$& \eqref{eq:varrho}\\
$\mathrm{FD}^v(\FF)$& vertical factorization data  &Definition \ref{def:vfact_dat}\\
$\mathcal{FP}^v$& vertically factorized sheaves & Definition \ref{def:vertical-categories} \\
$\mathcal{FP}_{\leq d}^v$& vertically factorized sheaves & Definition \ref{def:vertical-categories} \\
$(\mathcal{FP}^{\leq d})^v$& vertically factorized sheaves & Definition \ref{def:vertical-categories} \\
$j_{\leq d}^*, j^*_{\leq d,\leq c}$& restriction functors to $\FP_{\leq d}$ & Proposition \ref{prop:def-estensione}\\
$(\ )^{\leq d,\leq e}$, $(\ )^{\leq d}$& truncation functors to $\FP^{\leq d}$ & Proposition \ref{prop:def-estensione}\\
$j_{\leq d,\bullet}, j_{\leq d,\leq c,\bullet}$& extension functors from $\sym_{\leq d}(U)$ & Proposition \ref{prop:def-funtori-restrizione}\\
$(\ )_{\leq d}$& quasi-inverse to $(\ )^{\leq d}$&Theorem \ref{thm:equivalence}
\end{tabular}
\endgroup
\end{center}

\newpage


\normalsize

\begin{thebibliography}{50}

 \bibitem{achar} P. Achar, Perverse Sheaves and Applications to Representation Theory, Mathematical Surveys and Monographs Vol. 258, AMS, (2021). 

 \bibitem{ayoub} J. Ayoub,  Les six op\'erations de Grothendieck et le formalisme des cycles \'evanescents dans le monde motivique, I.
Ast\'erisque (2007), no. 314, x + 466 pp. (2008).

 
    \bibitem{BBD} A.A. Be\u{\i}linson, J. Bernstein, P. Deligne,
    Faisceaux pervers. (French) [Perverse sheaves] Analysis and topology on singular spaces, I (Luminy, 1981), 5–171,
    Astérisque, 100, Soc. Math. France, Paris, 1982. 

\bibitem{BFS}R. Bezrukavnikov, M. Finkelberg, V. Schechtman, Factorizable Sheaves and Quantum Groups, LNM 1691, Springer, 1998.
 
\bibitem{CERyD-A} G. Carnovale, F. Esposito, L. Rubio y Degrassi, {Approximation of graded bialgebras}, arXiv:2412.00234v2. 


\bibitem{dCM}{ M.  A.  de Cataldo,  L. Migliorini}, The decomposition theorem,  perverse sheaves and the topology of algebraic maps.  Bull.  Amer.  Math.  Soc. (N.S.) 46 (2009),  4,  535--633.

\bibitem{DCM}{ M.  A.  de Cataldo,  L. Migliorini}, The Hodge theory of algebraic maps, Ann. Scient. \'Ec. Norm. Sup.,
4e s\'erie, t. 38, (2005), 693--750. 
  

    \bibitem{EGNO} P. Etingof, S. Gelaki, D. Nykshyck, V. Ostrik, Tensor categories, MSM vol 205, AMS, 2015. 

 
    \bibitem{Fresse}B.  Fresse, {Homotopy of operads and Grothendieck-Teichm\"uller groups. Part 1. The algebraic theory and its topological background.} Mathematical Surveys and Monographs, 217, American Mathematical Society, Providence, RI, (2017).

    \bibitem{Ga}D. Gaitsgory, Twisted Whittaker model and factorizable sheaves, Sel. Math., New ser. 13, 617 (2008). 

\bibitem{G} J. W. Gray, Formal Category Theory, Lecture Notes in Math.391, Springer-Verlag, 1974.

   \bibitem{GMcP}M. Goresky, R. MacPherson,  {Intersection homology II}. Invent. Math. 72 (1983), no. 1, 77–129.

 \bibitem{HS}I. Heckenberger, H. J. Schneider,  Hopf algebras and root systems. Math. Surveys Monogr., 247 American Mathematical Society, 2020. 

   \bibitem{KS1}M. Kapranov,   V. Schechtman, Perverse sheaves over real hyperplane arrangements, Ann. of Math. (2) 183 (2016), no. 2, 619–679.



    \bibitem{KS3}M. Kapranov, V. Schechtman,  {Shuffle algebras and perverse sheaves}. Pure Appl. Math. Q. 16 (2020), no. 3, 573–657.

    \bibitem{KS-snakes}M. Kapranov, V. Schechtman, PROBs and perverse sheaves I: symmetric products, Selecta Mathematica 31(2) (2025). 

  \bibitem{Sch}O. M. Schn\"urer, Equivariant sheaves on flag varieties, Math. Zeitschrift 267 1-2, (2011)
27-80.



\bibitem{ter}L. Terenzi, Monoidal structures on triangulated fibered categories, Theory and Applications of Categories, to appear.

   \bibitem{Treu}D.  Treumann,  {Exit paths and constructible stacks}. Compos. Math. 145 (2009), no. 6, 1504--1532. 



\end{thebibliography}
\end{document}